\documentclass{amsart}

\usepackage{amsmath,amssymb,amsfonts}
\usepackage{esint, cancel}
\usepackage{amsthm}
\usepackage{mathrsfs}
\usepackage{geometry}
\usepackage{verbatim}
\usepackage{graphicx}
\usepackage{color}

\usepackage[colorlinks=true,urlcolor=blue, citecolor=red,linkcolor=blue,
linktocpage,pdfpagelabels, bookmarksnumbered,bookmarksopen]{hyperref}
\usepackage[hyperpageref]{backref}

\newtheorem{theorem}{Theorem}[section]
\newtheorem{lemma}[theorem]{Lemma}
\newtheorem{proposition}[theorem]{Proposition}
\newtheorem{corollary}[theorem]{Corollary}

\theoremstyle{definition}
\newtheorem{remark}[theorem]{Remark}
\newtheorem{example}[theorem]{Example}
\newtheorem{definition}[theorem]{Definition}
\newtheorem*{open}{Open problem}
\newtheorem*{ack}{Acknowledgments}

\numberwithin{equation}{section}

\author[Bozzola]{Francesco Bozzola}
\address[F.\ Bozzola]{DIME Dipartimento di ingegneria meccanica, energetica, gestionale e dei trasporti 
	\newline\indent
	Universit\`a di Genova
	\newline\indent
	via alla Opera Pia 15, 16145 Genova, Italy}
\email{francesco.bozzola@edu.unige.it}

\author[Talluri]{Matteo Talluri}
\address[M.\ Talluri]{Dipartimento di Matematica
	\newline\indent
	Alma Mater Studiorum  Universit\`a di Bologna
	\newline\indent
	piazza di Porta San Donato, 5
	40126 Bologna, Italy}
\email{matteo.talluri@unibo.it}

\date{\today}


\title[Maz'ya-type bounds]{Maz'ya-type bounds for sharp constants in fractional Poincar\'e-Sobolev inequalities}

\begin{document}
	
	\begin{abstract}
		We prove estimates for the sharp constants in fractional Poincar\'e-Sobolev inequalities associated to an open set, in terms of a nonlocal capacitary extension of its inradius. This work builds upon previous results obtained in the local case by Maz'ya and Shubin and by the first author and Brasco. We rely on a new Maz'ya-Poincar\'e inequality and, incidentally, we also prove new fractional Poincar\'e-Wirtinger-type estimates. These inequalities display sharp limiting behaviours with respect to the fractional order of differentiability. As a byproduct, we obtain a new criterion for the embedding of the homogeneous Sobolev space $\mathcal{D}^{s,p}_0(\Omega)$ in $L^q(\Omega)$, valid in the subcritical regime and for $p \le q < p^*_s$. Our results are new even for the first eigenvalue of the fractional Laplacian and contain an optimal characterization for the positivity of the fractional Cheeger's constant.
	\end{abstract}
	
	\maketitle
	
	\begin{center}
		\begin{minipage}{10cm}
			\small
			\tableofcontents
		\end{minipage}
	\end{center}
	
	\section{Introduction}
	
	\subsection{Overview and goal of the paper}
	Let $\Omega \subseteq \mathbb{R}^N$ be an open set, its {\it fractional principal frequencies} are given by
	\begin{equation} \label{defi:freq}
		\lambda^{s}_{p, q}(\Omega) := \displaystyle \inf_{u \in C^{\infty}_0(\Omega)} \left\{\iint_{\mathbb{R}^N \times \mathbb{R}^N} \frac{|u(x) - u(y)|^p}{|x-y|^{N+s\,p}}\,dxdy \,:\, \|u\|_{L^q(\Omega)} = 1 \right\}, \qquad s\in(0,1),
	\end{equation}
	where $1 \le p < \infty$ and $q \ge 1$ satisfies 
	\begin{equation} \label{hp:subcritical}
		\begin{cases}
			q \le p^*_s, \quad &\mbox{ if } s\,p \in (0, N) \cup (N, \infty),\\
			\\
			q < p^*_s, \quad &\mbox{ if } s\,p=N, \\
		\end{cases}
	\end{equation}
	Here we adopt the usual notation for the critical fractional Sobolev's exponent 
	\[
	p^*_s := \begin{cases}
		\begin{aligned}
			\dfrac{N\,p}{N-s\,p}, \quad &\mbox{ if } s\,p < N,\\
			\\
			\infty, \quad &\mbox{ if } s\,p \ge N.
		\end{aligned}
	\end{cases}
	\]
	For $s=1$, the previous notation corresponds to the classical critical Sobolev's exponent, which will be denoted by the symbol $p^*$. 
	The quantity which is minimized in \eqref{defi:freq} corresponds to the $p-$power of the {\it Gagliardo-Slobodecki\u{\i} seminorm}, in the sequel denoted by the symbol $[\,\cdot\,]_{W^{s,p}(\mathbb{R}^N)}$. For $p = q$, we will use the distinguished notation $\lambda_{p}^s(\Omega) := \lambda_{p, p}^s(\Omega)$.     
	\par
	By definition, the variational quantity $\lambda_{p, q}^s(\Omega)$ coincides with the {\it sharp constant in the fractional Poincar\'e-Sobolev inequality}
	\begin{equation} \label{intro:poin-sob}
		c_\Omega\,\int_\Omega |u|^q\,dx \le \iint_{\mathbb{R}^N \times \mathbb{R}^N} \frac{|u(x) - u(y)|^p}{|x-y|^{N+s\,p}}\,dxdy, \quad \mbox{ for } u \in C^\infty_0(\Omega),
	\end{equation}
	associated to an open set $\Omega \subseteq \mathbb{R}^N$. Noteworthy, we have the following equivalence 
	\[
	\lambda^s_{p, q}(\Omega) > 0 \quad \Longleftrightarrow \quad \mathcal{D}^{s,p}_0(\Omega) \hookrightarrow L^q(\Omega),
	\]
	and in this case $\mathcal{D}^{s,p}_0(\Omega)= \widetilde{W}^{s,p}_0(\Omega)$ (see for instance \cite[Corollary 1.2, Theorem 1.3 \& Lemma 2.3]{franzina-torsion}). Here $\mathcal{D}^{s,p}_0(\Omega)$ stands for the {\it homogeneous fractional Sobolev space} with nonlocal Dirichlet boundary conditions, suitably introduced in Section \ref{sec:2}.
	
	\begin{remark} \label{rmk:relax}
		By standard approximation arguments, see for instance \cite[Lemma 11]{FSV}, for every $\Omega \subseteq \mathbb{R}^N$ open set  
		\[
		\lambda_{p,q}^s(\Omega) = \inf_{u \in X} \left\{\iint_{\mathbb{R}^N \times \mathbb{R}^N} \frac{|u(x)-u(y)|^p}{|x-y|^{N+s\,p}}dxdy : \|u\|_{L^q(\Omega)} = 1\right\},
		\]
		where $X$ can be indifferently chosen among the spaces $C^\infty_0(\Omega),\,{\rm Lip}_0(\Omega)$ or $\widetilde{W}^{s,p}_0(\Omega) \cap L^q(\Omega)$. This equivalence will be used throughout whole the paper without mention. We refer to Section \ref{sec:2} for the precise definitions of these spaces. 
	\end{remark} 
	In analogy with the local case, for $1 < p < \infty$ if there exists a nonnegative solution $u \in \mathcal{D}^{s,p}_0(\Omega)$ attaining the infimum \eqref{defi:freq}, then it satisfies in the weak sense the following nonlocal {\it Lane-Emden-type equation}
	\[
	\begin{cases}
		\begin{aligned}
			(-\Delta_p)^s u &= \lambda\,u^{q-1}, &&\quad \mbox{ in } \Omega,\\
			\\
			u&=0, &&\quad \mbox{ in } \mathbb{R}^N \setminus \Omega,
		\end{aligned}
	\end{cases}
	\quad \mbox{ with } \lambda = \lambda^s_{p, q}(\Omega),
	\]
	where the nonlocal operator $(-\Delta_p)^s$ stands for the {\it fractional $p-$Laplacian}
	\[
	(-\Delta_p)^s u(x) := 2\,\lim_{\varepsilon \to 0} \int_{\mathbb{R}^N \setminus B_\varepsilon(x)}\,\frac{|u(x) - u(y)|^{p-2}\,(u(x) - u(y))}{|x-y|^{N+s\,p}}\,dy, \qquad x \in \mathbb{R}^N. 
	\]
	
	In the last decade, there has been growing interest in studying properties of the sharp fractional Poincar\'e-Sobolev constants of an open set, in connection with Spectral Geometry, regularity and properties of extremals, Hardy and Faber-Krahn-type inequalities of nonlocal nature, and several other issues. Without any effort to be exhaustive, we only mention the following papers \cite{AbFelNor, BiaBra22, BiaBra24, BC, BCV, BLP, BraMosSqu, BP, BPS, franzina-torsion, FraPal, LindLind}, which provide themselves, along with the references therein contained, a solid mathematical background on the themes analyzed in this manuscript. 
	\par 
	The same due attention has to be paid to the literature available for the ``local" counterpart of these topics. As one may expect, the corresponding literature is even more vast. Here, as major sources of inspiration, we mention the fundamental Maz'ya's book \cite{Maz}, the recent book by Brasco \cite{Brasco_book} and the following articles  \cite{BozBra, BozBra2, BB_variation,  BraPini, BraTol, BraBriPri, BraDePFran, BraFran, BraLin, BraMaz, BraPriZag1, BraPriZag2, BraRuf}. Along with the references therein contained, these works may be of some help to the interested reader to imagine (with a little of fantasy) a bridge between the local and nonlocal counterpart of the subject we will treat. 
	\vskip.2cm \noindent
	Before entering into the heart of the matter, as a straightforward consequence of the definition we record the following relations 
	\begin{equation} \label{scaling-frequency}
		\lambda_{p, q}^s(B_r(x_0)) = \frac{\lambda^s_{p, q}(B_1)}{r^{s\,p-N+N\,\frac{p}{q}}}, \quad \mbox{ and }  \quad \lambda_{p, q}^s(\Omega_1) \le \lambda_{p, q}^s(\Omega_2), \quad \mbox{ if } \Omega_2 \subseteq \Omega_1. 
	\end{equation}
	These entails the following sharp upper bound on the frequencies
	\begin{equation} \label{monotonicity}
		\lambda_{p, q}^s(\Omega) \le \frac{\lambda^s_{p, q}(B_1)}{r_\Omega^{s\,p-N+N\,\frac{p}{q}}}, 
	\end{equation}
	where $r_\Omega$ denotes the {\it inradius of $\Omega$}, namely
	\begin{equation} \label{inradius}
		r_\Omega = \sup\Big\{r > 0 : \exists x_0 \in \Omega \mbox{ such that } B_r(x_0) \subseteq \Omega  \Big\}.
	\end{equation}
	The upper bound \eqref{monotonicity} can be interpreted by saying that the Poincar\'e-Sobolev inequality \eqref{intro:poin-sob} {\it ceases to be true} whenever $\Omega$ contains balls of radius arbitrarily large. The following particular cases deserve a special mention.
	\begin{itemize}
		\item If $q = p^*_s$, where $s\,p < N$, the power appearing in the denominator of the rightmost term of \eqref{monotonicity}  vanishes. The quantity $\lambda^s_{p,p^{*}_s}(\Omega)$ is {\it independent} from the open set $\Omega$, that is 
		\begin{equation} \label{costante-sobolev}
			\lambda^s_{p,p^{*}_s}(\Omega) = \lambda^s_{p, p^*_s}(\mathbb{R}^N), \quad \mbox{ for every } \Omega \subseteq \mathbb{R}^N \mbox{ open set}, 
		\end{equation}
		and it corresponds to the best constant in the {\it fractional Sobolev's inequality} 
		\[
		C\left(\int_{\Omega} |u|^{p^*_s}\,dx\right)^{\frac{p}{p^*_s}} \le \iint_{\mathbb{R}^N \times \mathbb{R}^N} \frac{|u(x) - u(y)|^p}{|x-y|^{N+sp}}\,dxdy, \qquad \mbox{ for } u \in C^\infty_0(\Omega).
		\]
		This observation can be made rigorous by arguing as in \cite[Chapter I, Section 4.5]{Struwe}. For later convenience, we introduce a distinguished notation for the inverse of $\lambda^s_{p, p^*_s}(\mathbb{R}^N)$
		\[
		\mathcal{S}_{N,p,s} := \sup_{u \in C^\infty_{0}(\Omega)} \left\{\left(\int_{\mathbb{R}^N} |u|^{\frac{N\,p}{N-s\,p}}\,dx\right)^{\frac{N-s\,p}{N}} \,:\, \iint_{\mathbb{R}^N \times \mathbb{R}^N} \frac{|u(x) - u(y)|^p}{|x-y|^{N + s\,p}}\,dxdy = 1\right\}.
		\] 
		More details on this constant can be found in \cite{BraMosSqu, CotTav, MazShap} and in the references therein contained. 
		\vskip.2cm
		\item If $p=q = 1$, by \cite[Theorem 5.7]{BLP}, the Poincar\'e constant $\lambda^s_{1}(\Omega)$ coincides with the {\it fractional Cheeger's constant} of $\Omega$ \begin{equation}\label{cheeger-constant}
			h_s(\Omega) :=\inf \left\{\frac{P_s( E)}{|E|}: E \Subset \Omega \text { open set with smooth boundary }\right\},
		\end{equation}
		first considered\footnote{Actually, the authors in \cite[formula (5.1)]{BLP} propose a slightly different definition of {\it fractional Cheeger's constant}. In particular, their quantity is always smaller than \eqref{cheeger-constant}, and the two notions coincide for open and bounded Lipschitz sets, see \cite[Theorem 5.8]{BLP} and Lemma \ref{Cheeger--type constant} below.} in \cite{BLP}. Here $P_s$ stands for the \emph{fractional $s-$perimeter}, which can be defined in terms of the Gagliardo-Slobodecki\u{\i} seminorm as \[
		P_s(A) := [1_{A}]_{W^{s,1}(\mathbb{R}^N)} = \iint_{\mathbb{R}^N \times \mathbb{R}^N} \frac{|1_A(x) - 1_A(y)|}{|x-y|^{N+s}}\,dxdy, \quad \mbox{ for every Borel } A \subseteq \mathbb{R}^N.
		\]
		Also for the $s-$perimeter, we do not claim to give an exhaustive list of references and we mention only the papers \cite{AmbDePMar, BLP, FFMMM, Lombardini}. 
		\vskip.2cm
		\item If $p = q = 2$, the operator $(- \Delta_p)^s$ coincides with the well-known {\it fractional Laplacian}, up to a normalization constant $c = c(N,s)$, and $\lambda^s_{2}(\Omega)$ corresponds to the {\it bottom of the spectrum of the fractional Laplacian} with nonlocal Dirichlet boundary conditions  on $\Omega$. If the infimum \eqref{defi:freq} is attained, then $\lambda_{2}^s(\Omega)$ is also called the {\it first (Dirichlet) eigenvalue} of the fractional Laplacian. We refer to Chapter 3 of the monograph \cite{BisRadServadei_book} and to the references therein contained for a detailed account on the eigenvalue problem for the fractional Laplacian. 
	\end{itemize}
	\vskip.2cm \noindent	
	The previous preamble was in order to settle the problem we aim to attack. Here, we are interested in characterizing the positivity of the sharp fractional Poincar\'e-Sobolev constants $\lambda_{p, q}^s$,  among the class of {\it all} the open subsets of $\mathbb{R}^N$. 
	\par 
	A first prototype result one could imagine might be to reverse the estimate \eqref{monotonicity}. Unfortunately, this cannot be achieved unless by adding quite restrictive assumptions on the open set $\Omega$ and/or restrictions on the parameter $s$, for instance see \cite[Theorem 1.3]{BiaBra22} for a counterexample and \cite[Introduction]{BiaBra24}. Positive results in this direction can be found in \cite[Corollary 2]{BanLatMH} and \cite[Corollary 5.1]{BC}. 
	\vskip.2cm \noindent
	For a general open set $\Omega$, the ultimate reason which prevents to infer the positivity of $\lambda_{p, q}^s(\Omega)$ in terms of the finiteness of its inradius, is that, differently from $r_\Omega$, the frequencies $\lambda_{p, q}^s(\Omega)$ are not affected by {\it capacitary perturbations} of $\Omega$ by compact sets $\Sigma \subseteq B_{R}(x_0)$ of null {\it homogeneous relative $(s,p)-$capacity}
	\begin{equation} \label{defi:rel-cap}
		\widetilde{{\rm cap}}_{s,p}\left(\Sigma; B_R(x_0)\right) = \inf_{\varphi \in C^\infty_0(B_R(x_0))}\left\{\iint_{\mathbb{R}^N \times \mathbb{R}^N} \frac{|\varphi(x) - \varphi(y)|^p}{|x-y|^{N+s\,p}}\,dxdy \,:\, \varphi \ge 1 \mbox{ on } \Sigma\right\},
	\end{equation}
	see Proposition \ref{prop:cap-null} (see also \cite[Remark 1.4]{BiaBra22}). 
	Typically, we will call $\widetilde{{\rm cap}}_{s,p}$ just {\it fractional capacity} or {\it (relative) $(s,p)-$capacity} and refer to Section \ref{subsec:cap} for some fundamentals on this notion. 
	\vskip.2cm \noindent
	The same obstruction occurs in the local setting. To encompass this difficulty, in \cite{MS} Maz'ya and Shubin introduced the {\it interior capacitary radius} and proved a two-sided estimate for $\lambda_{p}$ in terms of this quantity valid {\it for every} open set, in the case $p=2$. Here, we denote with
	\[
	\lambda_{p}(\Omega) = \inf_{C^\infty_0(\Omega)}\left\{\int_\Omega |\nabla u|^p\,dx : \|u\|_{L^p(\Omega)} = 1\right\}, 
	\]
	the sharp constant in the Poincar\'e's inequality 
	\[
	c_\Omega\,\int_\Omega |u|^p\,dx \le \int_\Omega |\nabla u|^p\,dx, \quad \mbox{ for } u \in C^\infty_0(\Omega).
	\]
	We also recall that, if $1 \le p < \infty$ and $E \subseteq \mathbb{R}^N$ is an open set, the {\it relative $p-$capacity} of a compact set $\Sigma \subseteq E$ is
	\begin{equation} \label{defi:rel-cap-local}
		{\rm cap}_p\left(\Sigma; E\right) = \inf_{\varphi \in C^\infty_0(E)}\left\{\int_E |\nabla \varphi|^p\,dx : \varphi \ge 1 \mbox{ on } \Sigma\right\}.
	\end{equation}
	Observe that the relative $(s,p)-$capacity \eqref{defi:rel-cap} is the natural nonlocal counterpart of \eqref{defi:rel-cap-local}.  
	\par 
	The interior capacitary radius is a {\it capacitary variant} of the inradius, in the sense that the competitors balls $B_r(x_0)$ in \eqref{inradius} are allowed to cross $\partial \Omega$ for a small portion, uniformly controlled in $2-$capacity by means of a precise power of their radius, $r^{N-2}$, and a parameter $\gamma$. We refer to \cite{BozBra2, BB_variation} for a wider discussion. 
	\par 
	In \cite{BozBra2} the authors sharpened and extended the results contained in \cite{MS}. More precisely, in \cite[Main Theorem]{BozBra2}, the first author and Brasco obtained the following equivalence\footnote{Actually, in \cite[Theorem 6.1]{BozBra2} a more general result  is stated, valid for Poincar\'e-Sobolev constants $\lambda_{p,q}$ with $p \le q < p^*$. } 
	\begin{equation} \label{bozbra2}
		\sigma_{N,p}\,\gamma\, \left(\frac{1}{R_{p,\gamma}(\Omega)}\right)^p \le \lambda_p(\Omega) \le C_{N,p,\gamma}\,\left(\frac{1}{R_{p,\gamma}(\Omega)}\right)^p, \quad \mbox{ for } \gamma \in (0, 1),
	\end{equation} 
	for every $1 \le p \le N$ and for every $\Omega \subseteq \mathbb{R}^N$ open set. The quantity $R_{p,\gamma}$ stands for {\it $(p, \gamma)-$capacitary inradius} given by 
	\begin{equation} \label{defi:capin-local}
		R_{p, \gamma}(\Omega) = \sup\Big\{r > 0 : \exists B_r(x_0) \mbox{ such that } \overline{B_r(x_0)}\setminus \Omega \mbox{ is } (p, \gamma)-\mbox{negligible} \Big\},
	\end{equation}
	where, by saying {\it $(p, \gamma)-$negligible}, we mean that 
	\[
	{\rm cap}_p\left(\overline{B_r(x_0)}\setminus \Omega; B_{2\,r}(x_0)\right) \le \gamma\,{\rm cap}_p\left(\overline{B_r(x_0)}; B_{2\,r}(x_0)\right).
	\]
	We kindly refer the reader always to \cite{BozBra2, BB_variation} for an in-depth account on the {\it capacitary inradius} and possible variations of it. Here we confine ourselves to saying that $R_{p, \gamma}$ is a variant both of Maz'ya and Shubin's interior capacitary inradius considered in \cite{MS} and of Maz'ya's {\it $(p, 1)-$inner (cubic) diameter} (see \cite[Definition 14.2.2]{Maz}).
	\par 
	The relevant feature of the $(p, \gamma)-$capacitary inradius and of the interior capacitary radius, or the $(p, 1)-$inner diameter, is that their finiteness {\it completely characterizes} the positivity of $\lambda_{p}$, and more in general of $\lambda_{p, q}$ with $q \in [p, p^*)$, among the class of all the open subsets of $\mathbb{R}^N$. This was first noticed by Maz'ya in the $70's$, by using the notion of $(p, 1)-$inner diameter, see \cite[Theorem 15.4.1]{Maz} and the historical comment in \cite[pag. 692]{Maz}. 
	\par 
	For this reason, in the present paper we introduce the terminology {\it Maz'ya-type bounds} to refer to two-sided estimates of the sharp Poincar\'e-Sobolev constants in terms of capacitary-based objects, like the capacitary inradius/inner diameter and their fractional variants. 
	\vskip.2cm \noindent 
	In this manuscript, we obtain Maz'ya-type bounds for the sharp constants in the {\it fractional} Poincar\'e-Sobolev inequalities $\lambda_{p, q}^s$. More precisely, our main goal is to formulate and transpose a two sided-estimate like the one provided by \eqref{bozbra2} into the nonlocal framework. In analogy with \eqref{defi:capin-local}, we are led to introduce the following 
	\begin{definition}\label{def:inradius}
		Let $1 \le p < \infty$ and $0<s<1$. For every $0<\gamma<1$, we say that a compact set $\Sigma \subseteq \overline{B_r(x_0)}$ is $(s, p, \gamma)-$\emph{negligible} if
		\begin{equation*}
			\widetilde{{\rm cap}}_{s,p}\left(\Sigma; B_{2\,r}(x_0)\right) \le \gamma\,\widetilde{{\rm cap}}_{s,p}\left(\overline{B_r(x_0)}; B_{2\,r}(x_0)\right).
		\end{equation*}
		Accordingly, for every $\Omega \subseteq \mathbb{R}^N$ open set its {\it $(p,\gamma)-$capacitary inradius of order $s$} is given by 
		\begin{equation}
			R^s_{p,\gamma}(\Omega) := \sup \Big\{ r > 0 :\, \exists B_r(x_0) \mbox{ such that } \overline{B_r(x_0)} \setminus \Omega \mbox{ is } (s,p,\gamma)-\mbox{negligible} \Big\}.
		\end{equation}
		Typically, we will refer to $R^s_{p, \gamma}$ just as {\it fractional capacitary inradius}.  
	\end{definition}
	
	\begin{remark}
		By definition
		\[
		r_\Omega \le R^s_{p, \gamma}(\Omega), \mbox{ for every } 0 < \gamma < 1, \quad \mbox{ and} \quad \gamma \mapsto R^s_{p, \gamma}(\Omega) \mbox{ is monotone non-decreasing.}
		\]
		For a comparison between $R_{p,\gamma}$ and $R^s_{p, \gamma}$, we refer to Proposition \ref{prop:capin-vs-capin}. 
	\end{remark}
	\subsection{Main Theorems} The following theorems are the main achievements of the paper. They can be interpreted as nonlocal variants of \cite[Main Theorem \& Theorem 6.1]{BozBra2}. 
	\par 
	The lower bound on $\lambda_{p, q}^s$ reads as follows 
	
	\begin{theorem}[Lower bound]
		\label{thm:lower-bound}
		Let  $1 \le p < \infty$ and $0<s<1$ be such that $s\,p \le N$ and let $p \le q < p^*_s$.
		Let $0 < \gamma < 1$ and let $\Omega \subseteq \mathbb{R}^N$ be an open set. There exists a constant $\sigma = \sigma\left(N, p, s, q\right) > 0$ such that  
		\begin{equation} \label{main:lower-bound}
			\gamma\,\sigma\,\left(\frac{1}{R^s_{p, \gamma}(\Omega)}\right)^{s\,p-N+N\,\frac{p}{q}} \le \lambda_{p, q}^s(\Omega).
		\end{equation}
		Moreover, for $1 \le p < \infty$ we have 
		\[
		\sigma(N, p, s, q) \sim \frac{1}{s} \quad \mbox{ as } s \searrow 0, 
		\]
		and for $1 \le p \le N$ we have 
		\[
		\sigma(N, p, s, q) \sim \frac{1}{1-s} \quad \mbox{ as } s \nearrow 1. 
		\]
	\end{theorem}

	As for the upper bound on $\lambda_{p, q}^s$, we have the following

	\begin{theorem}[Upper bound]
		\label{thm:upper-bound}
		Let $1 \le p < \infty$ and $0 < s < 1$ be such that $s\,p\leq N$ and let $p \le q < p^*_s$. Let $\Omega \subseteq \mathbb{R}^N$ be an open set. There exists $0 < \gamma_0 = \gamma_0(N,p,s) \le 1$ such that for every $0 < \gamma < \gamma_0$, we have   \begin{equation}\label{main:upper-bound p-q}
			\lambda_{p,q}^s(\Omega)\leq \mathcal{C}\,\left(\frac{1}{R^s_{p, \gamma}(\Omega)}\right)^{s\,p-N+N\,\frac{p}{q}},
		\end{equation}
		for some constant $\mathcal{C} = \mathcal{C}(N,p,s, q, \gamma) > 0$; if $p=1$ we can take $\gamma_0 =1.$ In particular, if $R^s_{p, \gamma}(\Omega) = +\infty$ then $\lambda^s_{p,q}(\Omega) = 0$.  Moreover, if $1 \le p < \infty$ we have 
		\[
		\mathcal{C}\left(N,p,s, q, \gamma\right) \sim \frac{1}{s} \quad\mbox{ as } s \searrow 0,
		\]
	for $0 < \gamma < \liminf_{s \to 0} \gamma_0\left(N,p,s\right)$, and if $1 \le p \le N$ we have 
		\[
		\mathcal{C}\left(N,p,s, q, \gamma\right) \sim \frac{1}{1- s} \quad\mbox{ as } s \nearrow 1,
		\]
		for $ 0 < \gamma < \liminf_{s \to 1} \gamma_0\left(N,p,s\right)$.
	\end{theorem}
	
	For a discussion on the negligibility threshold $\gamma_0$ we refer to Remark \ref{rmk:asymptotic-s-upper bound}. 
	
	\begin{remark}
		We stress that the conclusions of the Main Theorems are a novelty even for the bottom of the spectrum of the fractional Laplacian with nonlocal Dirichlet boundary conditions, corresponding to the case $p = q =2$.  
	\end{remark}
	
	As a byproduct of Main Theorems, we get the following equivalence.
	\begin{corollary} \label{cor:embedding-ext}
		Let $1 \le p < \infty$ and $0 < s < 1$ be such that $s\,p \leq N$ and let $p \le q < p^*_s$. For every $\Omega \subseteq \mathbb{R}^N$ open set, we have 
		\[
		\mathcal{D}^{s,p}_0(\Omega) \hookrightarrow L^q(\Omega) \quad \Longleftrightarrow \quad R^s_{p, \gamma}(\Omega) < \infty,
		\]  
		for some $0 < \gamma < \gamma_0(N,p,s)$, where $\gamma_0$ is the same of Theorem \ref{thm:upper-bound}. 
	\end{corollary} 
	
	Already contained in the Main Theorems, we have the following {\it optimal} characterization for the positivity of the fractional Cheeger's constant of an open set, which we emphasize with a separate statement. The optimality has to be intended in the sense that the full spectrum of the admissible values of $\gamma$ is allowed. 
	\begin{corollary}
		Let $0 < s <1$ and $0 < \gamma < 1$. For every $\Omega \subseteq \mathbb{R}^N$ open set, we have 
		\[
		\sigma\,\gamma\,\left(\frac{1}{R^s_{1, \gamma}(\Omega)}\right)^s \le h_s(\Omega) \le \mathcal{C}\,\left(\frac{1}{R^s_{1, \gamma}(\Omega)}\right)^s,  
		\] 
		where the constants $\sigma = \sigma\left(N,s\right)$ and $\mathcal{C} = \mathcal{C}\left(N,s,\gamma\right)$ are the same of the Main Theorems, with $\mathcal{C}\left(N,s,\gamma\right)$ which diverges to $+\infty$ as $\gamma \to 1$ and   
		\[
		0 < \lim_{s \to 0} s\,\mathcal{C}\left(N,s,\gamma\right) < \infty, \qquad 0 < \lim_{s \to 1} (1-s)\,\mathcal{C}\left(N,s,\gamma\right) < \infty.
		\]
		In particular, we have 
		\[
		h_s(\Omega) > 0 \quad \Longleftrightarrow \quad R^s_{1, \gamma}(\Omega) < \infty,
		\]
		and the last condition does not depend on $0 < \gamma < 1$. 
	\end{corollary}
	
	As a byproduct of Theorem \ref{thm:lower-bound}, in analogy with \cite[Corollary 2]{BozBra2}, we can also infer an upper bound on $\|w_{p, s, \Omega}\|_{L^\infty(\Omega)}$ in terms of $R^s_{p, \gamma}(\Omega)$, where $w_{p, s, \Omega}$ stands for the so-called {\it $(s,p)-$torsion function on $\Omega$}.  Loosely speaking, it can be seen as the unique solution of the following problem 
	\[
	\begin{cases}
		\begin{aligned}
			(-\Delta_p)^s u = 1, &\quad \mbox{ in } \Omega,\\
			\\
			u = 0, &\quad \mbox{ in } \mathbb{R}^N \setminus \Omega,
		\end{aligned}
	\end{cases}
	\]
	we refer to \cite[Section 3]{franzina-torsion} for the precise definition.	The importance to get $L^\infty$ bounds on $w_{p,s, \Omega}$ is encoded in \cite[Theorem 1.1]{franzina-torsion}. We have the following
	
	\begin{corollary}
		Let $1 < p < \infty$ and $0 < s < 1$. let $\Omega \subseteq \mathbb{R}^N$ be an open set, then we have 
		\[
		\|w_{p, s, \Omega}\|_{L^\infty(\Omega)} \le \left(\frac{{\bf C}_{N,p,s}}{\gamma\,\sigma}\right)^\frac{1}{p-1}\,\Big(R^s_{p,\gamma}(\Omega) \Big)^{s\frac{p}{p-1}}, \quad \mbox{ for every } 0 < \gamma < 1,
		\]
		where $\sigma$ is the same constant in Theorem \ref{thm:lower-bound} and ${\bf C}_{N,p,s}$ is the same constant of \cite[formula (5.7)]{franzina-torsion}.  
	\end{corollary}	
	
	
	\subsection{Strategy of the proofs} The guidelines we followed for the proofs of our Main Theorems are inspired by the proofs of \cite[Main Theorem \& Theorem 6.1]{BozBra2}. We refer to \cite[Introduction]{BozBra2} for an exhaustive discussion. Roughly speaking, we tried to obtain the nonlocal counterparts of the crucial results used to prove \cite[Main Theorem \& Theorem 6.1]{BozBra2} and tried to reproduce the argument. 
	\par 
	Nevertheless, technical difficulties arose in adapting the local technique to the nonlocal framework. This led us to modify parts of the proofs in a nontrivial way and gave rise to a number of results, which weren't available in the literature and which are of independent interest, as well. Below, we comment separately the proofs of Theorem \ref{thm:lower-bound} and Theorem \ref{thm:upper-bound}. 
	\vskip.2cm \noindent
	{\it Comments on the proof of Theorem \ref{thm:lower-bound}.}  \begin{itemize}
		\item The first step is the choice of a {\it tiling} of the whole $\mathbb{R}^N$ made of balls having equidistant centers, the same radius $r$ and an explicit control on the multiplicity of the covering. The choice of the radius $r$ is not by chance and made in such a way to exploit the {\it definition of $R^s_{p,\gamma}(\Omega)$}, Definition \ref{def:inradius}. More precisely, by taking $r > R^s_{p, \gamma}(\Omega)$ we exploit a uniform estimate from below on the fractional capacity of $\overline{B_r(x_0)} \setminus \Omega$ in terms of $\gamma$ and $r$
		\begin{equation} \label{eqn:gamma-fat}
			\widetilde{{\rm cap}}_{s,p}\left(\overline{B_r(x_0)}\setminus \Omega; B_{2\,r}(x_0)\right) \ge \gamma\,r^{N-s\,p}\,\widetilde{{\rm cap}}_{s,p}\left(\overline{B_1};B_2\right), \quad \mbox{ for every } x_0 \in \mathbb{R}^N. 
		\end{equation}
		\vskip.2cm \noindent
		\item The crucial step is now to use a fractional version of the {\it Maz'ya-Poincar\'e inequality}, see Lemma \ref{lm:mazya-poin2}, on each ball of the tiling introduced in the previous point and then adding all the contributions. For local formulations of this inequality, see for instance \cite[Theorem 14.1.2]{Maz}. One could be tempted to use fractional variants of the Maz'ya-Poincar\'e inequality already existing in the literature: for example \cite[Proposition 4.3]{BiaBra24} and its natural extension to the case $p \neq 2$. Unfortunately, this would not lead to sharp limiting behaviours with respect to the fractional parameter of differentiability $s$ in \eqref{main:lower-bound}, more precisely for the regime $s \searrow 0$. Let us explain a little more clearly this issue, assume for simplicity $q = p = 2$. By keeping in mind Example \ref{ex:finiteness-inr} below, we can take an open set $\Omega$ such that 
		\begin{equation}\label{finiteness-inradius}
			\limsup_{s \to 0} R^{s}_{2,\gamma}(\Omega) < r_{N,\gamma}, \quad \mbox{ for } 0 < \gamma < \gamma_0(N),
		\end{equation}  
		for some $r_{N, \gamma} > 0$, and denote by $\mathscr{B}$ the family of balls introduced in the previous step. Then by using \cite[Proposition 4.3]{BiaBra24} and \eqref{eqn:gamma-fat} on each $B \in \mathscr{B}$ we would obtain\footnote{Here, the symbol ``$\gtrsim$" indicates that we have the inequality ``$\ge$" up to a universal constant depending only on $N$ and $p$.} 
		\[
		\begin{split}
			[u]^2_{W^{s,2}(\mathbb{R}^N)} &\gtrsim \sum_{B \in \mathscr{B}} [u]^2_{W^{s,2}(B)} \\&\gtrsim \frac{\gamma}{r^{2\,s}}\,s\,\widetilde{{\rm cap}}_{s,2}\left(\overline{B_1}; B_2\right)\,\sum_{B \in \mathscr{B}} \|u\|^2_{L^2(B)} = \frac{\gamma}{r^{2\,s}}\,s\,\widetilde{{\rm cap}}_{s,2}\left(\overline{B_1};B_2\right)\,\|u\|^2_{L^2(\mathbb{R}^N)},
		\end{split}
		\]
		for every $u \in C^\infty_0(\Omega)$. This would entail the following estimate
		\[
		\lambda_{2}^s(\Omega) \gtrsim \gamma\,s\,\widetilde{{\rm cap}}_{s,2}\left(\overline{B_1};B_2\right)\,\left(\frac{1}{R^s_{2,\gamma}(\Omega)}\right)^{2\,s},
		\]
		which becomes trivial as $s \searrow 0$, in light of \eqref{finiteness-inradius}, \eqref{eqn:asym-sharp1} and Proposition \ref{prop:asym2}. The loss of sharpness as $s \searrow 0$ in the previous argument is eventually imputable to the fact that {\it for every} $1 \le p < \infty$ and $0 < s < 1$, the inequality 
		\[
		\frac{1}{C}\,[u]_{W^{s,p}(\mathbb{R}^N)} \le [u]_{W^{s,p}(B_r(x_0))}, \quad \mbox{ for every } u \in C^\infty_0(B_r(x_0)),
		\]      
		fails, see for instance \cite[Lemma 2.3]{BS}. For this reason, by taking inspiration from \cite[Proposition 3.1]{BiaBra22}, we introduce an {\it asymmetric seminorm} on ``strips" 
		\begin{equation}\label{asym-semi1}
			u \mapsto \left(\iint_{B_r(x_0) \times \mathbb{R}^N} \frac{|u(x) - u(y)|^{p}}{|x-y|^{N+s\,p}}\,dxdy\right)^{\frac{1}{p}}, \quad \mbox{ for } u \in C^\infty_0(\mathbb{R}^N),
		\end{equation}
		which well-behaves in the limits $s \searrow 0$ and $s \nearrow 1$ (see Remark \ref{rmk:seminorma-asimmetrica} below). A new fractional Maz'ya-Poincar\'e inequality for the seminorm \eqref{asym-semi1}, encoding the {\it sharp} asymptotic behaviours in $s$, is then established in Lemma \ref{lm:mazya-poin2}. By using this inequality on each ball of the tiling introduced in the previous point, we eventually get the lower bound \eqref{main:lower-bound}, equipped with the correct asymptotic behaviours with respect to $s$.   
	\end{itemize}
	\begin{remark} [Limiting cases] \label{rmk:seminorma-asimmetrica}
		Let $x_0 \in \mathbb{R}^N$ and $r > 0$. It is immediate to verify that the function 
		\begin{equation}\label{asym-semi}
			u \mapsto \left(\iint_{B_r(x_0) \times \mathbb{R}^N} \frac{|u(x) - u(y)|^{p}}{|x-y|^{N+s\,p}}\,dxdy\right)^{\frac{1}{p}}, \quad \mbox{ for } u \in C^\infty_0(\mathbb{R}^N),
		\end{equation}
		is a seminorm (actually it is a norm on $C^\infty_0(\mathbb{R}^N)$). Noteworthy, it inherits  from $[\,\cdot\,]_{W^{s,p}(\mathbb{R}^N)}$ its {\it interpolative nature}. In other words, for every $u \in C^\infty_0(\mathbb{R}^N)$ we have 
		\[
		\iint_{B_r(x_0) \times \mathbb{R}^N} \frac{|u(x) - u(y)|^{p}}{|x-y|^{N+s\,p}}\,dxdy \sim \frac{1}{s}\,\int_{\mathbb{R}^N} |u|^p\,dx, \quad \mbox{ for } s \searrow 0,
		\]
		and 
		\[
		\iint_{B_r(x_0) \times \mathbb{R}^N} \frac{|u(x) - u(y)|^{p}}{|x-y|^{N+s\,p}}\,dxdy \sim \frac{1}{1-s}\,\int_{\mathbb{R}^N} |\nabla u|^p\,dx, \quad \mbox{ for } s \nearrow 1, 
		\]
		in accordance with the celebrated convergence theorems for the Gagliardo-Slobodecki\u{\i} seminorm, for $s \searrow 0$ and $s \nearrow 1$, respectively contained in \cite{MazShap} and \cite{BBM} (see also \cite{Bre_constant}).  
		Indeed, by \cite[Remark 3.4]{BS} and Proposition \ref{k-functional} below we can infer 
		\[
		\frac{1}{C}\,\int_{0}^{\infty}\left(\frac{K(t,u, L^p(\mathbb{R}^N), \mathcal{D}^{1,p}_0(\mathbb{R}^N))}{t^s}\right)^p \frac{dt}{t} \le \iint_{B_r(x_0) \times \mathbb{R}^N} \frac{|u(x) - u(y)|^{p}}{|x-y|^{N+s\,p}}\,dxdy \le [u]^p_{W^{s,p}(\mathbb{R}^N)},
		\]
		where $K$ indicates the {\it K-functional}, conveniently introduced in Section \ref{sec:4}.
		In light of \cite[Proposition 4.1]{BS}, this entails that 
		\[
		\frac{1}{C}\,[u]^p_{W^{s,p}(\mathbb{R}^N)} \le \iint_{B_r(x_0) \times \mathbb{R}^N} \frac{|u(x) - u(y)|^{p}}{|x-y|^{N+s\,p}}\,dxdy \le [u]^p_{W^{s,p}(\mathbb{R}^N)},
		\]
		which yields the desired limiting behaviours by \cite[Theorem 3]{MazShap} and \cite[Proposition 2.8]{BPS}.
	\end{remark}
	
	\vskip.2cm \noindent
	{\it Comments on the proof of Theorem \ref{thm:upper-bound}.}
	\begin{itemize}
		\item As in the proof of \cite[Main Theorem]{BozBra2}, the first step is an {\it approximation argument} for the capacitary potential of the set $\overline{B_r(x_0)} \setminus \Omega$ by means of smooth functions $(\varphi_\delta)_{\delta > 0}$, where $B_r(x_0)$ is any $(s,p,\gamma)-$negligible ball in the sense of Definition \ref{def:inradius}. By using a suitable {\it cut-off function}, we then modify each approximating functions into a feasible competitor for $\lambda_{p, q}^s(\Omega)$. In order to verify the last condition, we use an expedient technical result, Lemma \ref{lm:brezis-type}, which is of independent interest.
		\vskip.2cm \noindent
		\item The second step is to obtain explicit constants (and possibly the sharp ones) in a {\it $L^1-W^{s,p}$ Poincar\'e inequality on balls}. In Lemma \ref{poincarè-Palle}, by means of a geometric argument, we are able to compute the explicit form of the sharp constant for $p=1$. The case $p \neq 1$ is more delicate: as in \cite{BozBra2}, we are only able to estimate the sharp constant of the corresponding $L^1-W^{s,p}$ Poincar\'e inequality. Nevertheless, we pay due attention to the constants appearing in the bound. This would be important to assure that the Poincar\'e inequality has the correct asymptotic behaviour as $s \searrow 0$ and $s \nearrow 1$, see Lemma \ref{lm:bound-l-infinito}. This difference between the cases $p=1$ and $p \neq 1$ entails an artificial restriction on the negligibility parameter $\gamma$, for $p \neq 1$. On the other hand, the result obtained for $p=1$ is optimal since we get the full range of the admissible values of the parameter $\gamma$, as in \cite{BozBra2}. 
	\end{itemize}
	
	\begin{open}
		In light of the previous discussion, we may expect that the negligibility threshold $\gamma_0$ of Theorem \ref{thm:upper-bound} is not optimal for $1 < p < \infty$. Prove or disprove that the conclusion of Theorem \ref{thm:upper-bound} holds by taking $\gamma_0 = 1$ even for $1 < p < \infty$. 
	\end{open}
	
	\subsection{Plan of the paper} In Section \ref{sec:2}, we settle the notation, the functional analytic framework and some basic results needed throughout the whole manuscript. In particular, we prove an expedient lemma for fractional Sobolev function, Lemma \ref{lm:brezis-type}, and some fundamentals on the notion of relative $(s,p)-$capacity. Among which, we emphasize a nonlocal variant of the relation between $1-$capacity and perimeter due to Maz'ya and Fleming, Proposition \ref{prop:cap-per}. In Section \ref{sec:3}, we derive some Poincar\'e-type estimates on balls, see Lemma \ref{poincarè-Palle} and Lemma \ref{lm:bound-l-infinito}, which will be crucially exploited in the proof of Theorem \ref{thm:upper-bound}. Section \ref{sec:4} is devoted to derive two new fractional Maz'ya-Poincar\'e-type inequalities involving an asymmetric seminorm, Lemma \ref{lm:mazya-poin} and Lemma \ref{lm:mazya-poin2}, which are the cornerstones of the proof of Theorem \ref{thm:lower-bound}. Incidentally, we also establish new fractional Poincar\'e-Wirtinger-type estimates, Lemma \ref{lm:poincare-wirtinger} and Lemma \ref{lm:poin-sob-wirtinger}, which are of independent interest. Noteworthy, all the inequalities previously mentioned display sharp dependence with respect to the fractional order of differentiability $s$. Section \ref{sec:5} contains the proofs of the Main Theorems, Theorem \ref{thm:lower-bound} and Theorem \ref{thm:upper-bound}. Eventually, in Appendix \ref{sec:app} we briefly discuss the qualitative asymptotic behaviours of the frequencies as $s \searrow 0$ and $s \nearrow 1$, and exhibit an expedient example.   
	
	\begin{ack}
		The authors wish to express their most sincere gratitude to Prof. Lorenzo Brasco for insightful discussions on the topic of this paper. F.\,B. and M.T. are members of the {\it Gruppo Nazionale per l'Analisi Matematica, la Probabilit\`a
			e le loro Applicazioni} (GNAMPA) of the Istituto Nazionale di Alta Matematica (INdAM) and partially supported by the ``INdAM - GNAMPA Project Ottimizzazione Spettrale, Geometrica e Funzionale", CUP E5324001950001. M.T. is partially supported by the PRIN project ``NO3–Nodal Optimization, NOnlinear elliptic equations, NOnlocal geometric problems, with a focus on regularity”, CUP J53D23003850006.
	\end{ack}
	\section{Preliminaries}
	\label{sec:2}
	\subsection{Notation and functional spaces}
	Unless otherwise specified, throughout the whole paper we will assume the dimension of our ambient space $\mathbb{R}^N$ to be $N \ge 1$. As usual, given a point $x_0 \in \mathbb{R}^N$ and a real number $R > 0$ the symbol $B_R(x_0)$ stands for the $N-$dimensional open ball, centered at $x_0$ and having radius $R > 0$, i.e.
	\[
	B_R(x_0) = \left\{x \in \mathbb{R}^N : |x-x_0| < R\right\}. 
	\] 
	Whenever an open ball with radius $R$ is centered at the origin, we will use the shortcut notation $B_R$. The volume of the $N-$dimensional ball with radius $1$ will be denoted by $\omega_N$. The $N-$dimensional open hypercube centered at $x_0$ and having radius $R$ is given by 
	\[
	Q_R(x_0) = \prod_{i=1}^{N} (x_0^i - R, x_0^i + R), \qquad \mbox{ where } x_0=(x_0^1, \ldots, x_0^N) \in \mathbb{R}^N.
	\] 
	The $N-$dimensional Lebesgue measure is denoted by $|\cdot|$ and often it will be addressed just as volume or measure. For $1 \le p \le \infty$ the Lebesgue spaces $L^p$, $L^p_{\rm loc}$ and the norms $\|\cdot\|_{L^p}$ are defined as usual in the literature. The conjugate exponent of $p$ will be denoted by $p' = p/(p-1)$. If $E \subseteq \mathbb{R}^N$ is a measurable set having finite and positive volume we set 
	\[
	{\rm av}(u; E) := \fint_{E} u\,dx = \frac{1}{|E|} \int_E u\,dx, \qquad \mbox{ for } u \in L^1_{\rm loc}(\mathbb{R}^N),
	\]
	for the $L^1-$mean of $u$ over $E$.  Let $E, \Omega$ be open subsets of $\mathbb{R}^N$, the symbol $E \Subset \Omega$ means that the closure $\overline{E}$ is a compact subset of $\Omega$. 
	\par 
	If $E\subseteq\mathbb{R}^N$ is a non-empty open set, the notation $C^\infty_0(E)$ stands for the space of infinitely differentiable functions whose support is a compact subset of $E$. For $1\le p \leq \infty$, we will denote by $W^{1,p}(E)$ 
	the standard Sobolev space
	\[
	W^{1,p}(E)=\Big\{u\in L^p(E)\, :\, \nabla u\in L^p(E;\mathbb{R}^N)\Big\},
	\]
	endowed with the norm\[
	\|u\|_{W^{1,p}(E)} = \|u\|_{L^p(E)} + \|\nabla u\|_{L^p(E)}, \qquad \mbox{ for } u \in W^{1,p}(E),
	\] 
	where we used the notation \[
	\|\Phi\|_{L^p(E)} = \left(\int_{E} |\Phi(x)|^p\,dx\right)^{\frac{1}{p}}, \qquad \mbox{ for } \Phi \in L^p(E; \mathbb{R}^N).
	\]
	For $1 \leq p < \infty$, the {\it homogeneous Sobolev space} denoted by the symbol $\mathcal{D}^{1,p}_0(E)$ is defined as the completion (in the sense of metric spaces) of $C^\infty_0(E)$ with respect to the norm \[
	\varphi \mapsto \|\nabla \varphi\|_{L^p(E)}.
	\]
	The space $W^{1,p}_0(E)$ is given by the closure of $C^{\infty}_0(E)$ in $W^{1,p}(E)$ with respect to $\|\cdot\|_{W^{1,p}(E)}$.
	\par 
	We now introduce basic definitions from the theory of fractional Sobolev spaces. The reader is referred to \cite{Demengel_book,  EE, Grisvard_book, Leoni_fractional, Stein_book, Triebel1} and to the references therein included for a systematic treatment of these spaces. Let $1 \le p < \infty$ and $0 < s < 1$ and let $E \subseteq \mathbb{R}^N$ be as before. The Gagliardo-Slobodecki\u{\i} seminorm is defined as
	\[
	[u]_{W^{s,p}(E)}:= \left(\iint_{E \times E} \frac{|u(x) - u(y)|^p}{|x-y|^{N+s\,p}}dx\,dy\right)^{\frac{1}{p}}, \qquad \mbox{ for } u \in L^1_{\rm loc}(E). 
	\]
	The space of functions given by 
	\[
	W^{s,p}(E) = \left\{u \in L^p(E) : [u]_{W^{s,p}(E)} < \infty \right\},
	\]
	endowed with the norm 
	\[
	\|u\|_{W^{s,p}(E)} := \|u\|_{L^p(E)} + [u]_{W^{s,p}(E)}
	\]
	is called {\it fractional Sobolev space}. The symbol $\widetilde{W}_0^{s,p}(E)$ stands for the {\it closure} of $C^\infty_0(E)$ in $W^{s,p}(\mathbb{R}^N)$ with respect to its norm. We recall that the {\it homogeneous fractional Sobolev spaces} $\mathcal{D}^{s,p}_0(\Omega)$ are Banach spaces defined as the completion of 
	\[
	\left(C^\infty_0(\Omega),\,\, [\,\cdot\,]_{W^{s,p}(\mathbb{R}^N)}\right),
	\]
	equipped with the natural norm produced by this procedure, that is 
	\[
	\|u\|_{\mathcal{D}^{s,p}_0(\Omega)} := [u]_{W^{s,p}(\mathbb{R}^N)}.
	\]
	We refer to \cite{BGV} for a comprehensive treatment of $\mathcal{D}^{s,p}_0$, the homogeneous Sobolev spaces of functions satisfying nonlocal Dirichlet conditions. See also \cite{BS} and \cite[Chapter 6]{Leoni_fractional} for further information, references and extensions. 
	\begin{remark} \label{rmk:riflessivita}
		For $1 < p < \infty$, since the spaces $\mathcal{D}^{s,p}_0(\Omega)$ are uniformly convex\footnote{The homogeneous Sobolev space $\mathcal{D}^{s,p}_0(\Omega)$ inherits the uniform convexity property from the uniform convexity of $L^p(\Omega \times \Omega)$, for instance see \cite[Chapter 6, Paragraphs 11 \& 12]{DiB}. Indeed, by definition, we have $$\|u\|_{\mathcal{D}^{s,p}_0(\Omega)} = \|\mathcal{J}(u)\|_{L^p(\Omega \times \Omega)}, \quad \mbox{ where } \mathcal{J}(u) := \frac{u(x) - u(y)}{|x-y|^{\frac{N}{p} +s}},$$ 
			for every $u \in \mathcal{D}^{s,p}_0(\Omega).$}
		Banach spaces, they are also reflexive, in light of Milman-Pettis Theorem (for instance see \cite[Theorem 3.31]{Bre}). 
	\end{remark}

	\begin{remark}
		In general, we only have that $\widetilde{W}^{s,p}_0(E) \subseteq \mathcal{D}^{s,p}_0(E)$ and the inclusion is possibly strict, see \cite[Remark 3.2]{BGV}. Nevertheless, if $E \subseteq \mathbb{R}^N$ is an open set supporting the $(s,p)-$Poincar\'e inequality, namely $\lambda_{p}^s(E) > 0$, then $\widetilde{W}^{s,p}_0(E) = \mathcal{D}^{s,p}_0(E)$. We also recall that in general we only have the inclusion
		\[
		\widetilde{W}_0^{s,p}(E) \subseteq \left\{u \in W^{s,p}(\mathbb{R}^N) : u = 0 \,\,\mbox{ a.e. on } \mathbb{R}^N \setminus E\right\} =: W^{s,p}_0(E),
		\]
		see for instance  \cite[Remark 7]{FSV}. However, if we assume that $E$ has {\it continuous boundary} the three functional spaces considered coincide, by the density of $C^\infty_0(E)$ in $W^{s,p}_0(E)$ as proved in \cite[Theorem 6]{FSV}, see also \cite[Theorem 1.4.2.2]{Grisvard_book} and \cite[Proposition B.1]{BPS} for a slightly less general result obtained with a different proof.
	\end{remark}
	
	\par 
	For $0 < s \le 1$, we say that a function $u : E \to \mathbb{R}$ is {\it $s-$H\"older continuous} if
	\[
	[u]_{C^{0,s}(E)} := \sup_{x,y \in E,\,x \neq y}\, \frac{|u(x) - u(y)|}{|x-y|^s} < \infty,
	\]
	in particular for $s = 1$, the function $u$ is also said to be Lipschitz continuous. The space of $s-$H\"older continuous functions on $E$ will be denoted by $C^{0,s}(E)$. We set  
	\[
	C^{0,s}_0(E) :=\left\{u \in C^{0, s}(E) : {\rm supp}(u) \mbox{ is a compact subset of } E\right\},
	\]
	and for $s = 1$ we will often use the distinguished notation ${\rm Lip}_0(E)$. Given two functions $u, v$, with the notation $``u \sim v \mbox{ for } x \to x_0"$ we mean that 
	\[
	0 < \liminf_{x \to x_0} \frac{u(x)}{v(x)} \le \limsup_{x \to x_0} \frac{u(x)}{v(x)} < \infty.  
	\] 
	
	\subsection{Two expedient lemmas}
	The following interpolation estimate was established in \cite[Lemma 2.6]{BBZ} for compactly supported $C^1-$regular functions. With the same proof therein given, it is possible to state a more general result which will be useful in the sequel.   
	\begin{lemma} \label{leibniz-holder}
		Let $1 \le p < \infty$ and $0< s \le 1$. For every $\varphi \in C^{0,s}_0(\mathbb{R}^N)$, we have
		\[
		\sup_{x \in \mathbb{R}^N} \int_{\mathbb{R}^N} \frac{|\varphi(x) - \varphi(y)|^p}{|x-y|^{N+s\,p}}\,dy \le \frac{c_{N,p}}{s\,(1-s)} \|\varphi\|^{(1-s)\,p}_{L^\infty(\mathbb{R}^N)}\,[\varphi]^{s\,p}_{C^{0,s}(\mathbb{R}^N)}, 
		\]
		where we can take $c_{N,p} = 2^p\,\frac{N\,\omega_N}{p}.$
	\end{lemma}
	The next lemma may be regarded as a nonlocal analog of \cite[Theorem 9.17]{Bre}. It reads as follows
	\begin{lemma} \label{lm:brezis-type}
		Let $1 \le p < \infty$ and $0 < s < 1$. Let $E \subseteq \mathbb{R}^N$ be an open and bounded set. For every $\varphi \in C^{0,s}_0(\mathbb{R}^N)$ such that $\varphi = 0$ on $\mathbb{R}^N \setminus E$, we have $\varphi \in \widetilde{W}^{s,p}_0(E)$. 
	\end{lemma}
	\begin{proof}
		Without loss of generality, we assume that $0 \in E$. We need to exhibit a sequence of functions $(\varphi_h)_{h \in \mathbb{N}} \subseteq C^\infty_0(E)$ such that 
		\[
		{\rm (i)}\,\, \lim_{h \to \infty}\|\varphi_h - \varphi\|_{L^p(\mathbb{R}^N)} =  0; \qquad {\rm (ii)} \,\, \lim_{h \to \infty} [\varphi_h - \varphi]_{W^{s,p}(\mathbb{R}^N)} = 0.
		\]
		We set 
		\[
		E_h := \left\{x \in E : {\rm dist}\left(x; \partial E\right) > \frac{1}{h}\right\}, \quad \mbox{ for } h \in \mathbb{N},
		\]
		this is an open set. For every $h \in \mathbb{N}$, we take $\eta_h \in C^\infty_0(E_{4h})$ satisfying 
		\begin{equation} \label{prop-cut-off-brezis-type}
			\eta_h \equiv 1 \quad \mbox{ on } E_h, \qquad 0 \le \eta_h \le 1, \qquad |\nabla\eta_h| \le A\,h,
		\end{equation} 
		and set 
		\begin{equation} \label{defi-lip}
			\varphi_h:= \eta_h\,\varphi \in C^{0,s}_0(E_{4h}).
		\end{equation}
		Point (i) is straightforward. We focus on point (ii). We split the seminorm on the set $E \setminus E_h$
		\[
		\begin{split}
			[\varphi_h - \varphi]^p_{W^{s,p}(\mathbb{R}^N)}
			&= \iint_{(E \setminus E_h) \times (E \setminus E_h)} \frac{|(\eta_h(x)-1)\,\varphi(x) - (\eta_h(y)-1)\,\varphi(y)|^p}{|x-y|^{N+s\,p}}\,dxdy \\&\qquad + 2\,\int_{\mathbb{R}^N \setminus (E \setminus E_h)} \left(\int_{E \setminus E_h} \frac{|(\eta_h(y) -1)\,\varphi(y)|^p}{|x-y|^{N+s\,p}}\,dy\right)dx\\
			&=: \mathcal{A}_1 + \mathcal{A}_2. 
		\end{split}
		\]
		For the first term, by adding and subtracting $\varphi(x)\,(1-\eta_h(y))$ and by using Minkowski's inequality, we get 
		\[
		\begin{split}
			\frac{\mathcal{A}_1}{2^{p-1}} &\le \int_{E \setminus E_h} |\varphi(x)|^p \left(\int_{E \setminus E_h} \frac{|\eta_h(x) - \eta_h(y)|^p}{|x-y|^{N+s\,p}}\,dy\right)dx \\& \qquad + \int_{E \setminus E_h} |1-\eta_h(y)|^p\left(\int_{E \setminus E_h} \frac{|\varphi(x) - \varphi(y)|^p}{|x-y|^{N+s\,p}}\,dx\right)dy \\
			&=: \mathcal{B}_1 + \mathcal{B}_2. 
		\end{split}
		\]
		By using \eqref{prop-cut-off-brezis-type} and Lemma \ref{leibniz-holder}, we have
		\[
		\begin{split}
			\mathcal{B}_2 \le \frac{c_{N,p}}{s\,(1-s)}\,[\varphi]^{s\,p}_{C^{0,s}(\mathbb{R}^N)}\,\|\varphi\|^{(1-s)\,p}_{L^\infty(\mathbb{R}^N)}\, |E \setminus E_h|.
		\end{split}
		\]
		In order to estimate $\mathcal{B}_1$, we need to use that $\varphi$ is $s-$H\"older and vanishes on $\mathbb{R}^N \setminus E$. For every $x \in E \setminus E_{h}$, we have
		\begin{equation} \label{per-favore-funziona}
			\begin{split}
				|\varphi(x)|^p = |\varphi(x) - \varphi(z)|^p \le [\varphi]^p_{C^{0,s}(\mathbb{R}^N)}\,|x-z|^{s\,p} \le [\varphi]^p_{C^{0,s}(\mathbb{R}^N)}\,\left(\frac{1}{h}\right)^{s\,p}, 
			\end{split}
		\end{equation}
		for some $ z \in \partial E$. 
		This yields
		\[
		\begin{split}
			\mathcal{B}_1 &\le  \frac{c_{N,p}}{s\,(1-s)}\,[\varphi]^p_{C^{0,s}(\mathbb{R}^N)}\,|E \setminus E_h|\,(A\,h)^{s\,p}\,\left(\frac{1}{h}\right)^{s\,p}.  
		\end{split}
		\]
		where we used \eqref{prop-cut-off-brezis-type} and Lemma \ref{leibniz-holder}. By collecting the previous estimates and by sending $h \to \infty$, we get that 
		\begin{equation} \label{limite-1}
			\lim_{h \to \infty} \mathcal{A}_1 = 0. 
		\end{equation} 
		We now turn to $\mathcal{A}_2$, we have 
		\[
		\begin{split}
			\frac{\mathcal{A}_2}{2} &= \int_{\mathbb{R}^N \setminus E}\left(\int_{E \setminus E_{h}} \frac{|(\eta_h(y) -1 )\,\varphi(y)|^p}{|x-y|^{N+s\,p}}\,dy\right)dx + \int_{E_h}\left(\int_{E \setminus E_{h}} \frac{|(\eta_h(y) - \eta_h(x) )\,\varphi(y)|^p}{|x-y|^{N+s\,p}}\,dy\right)dx\\
			&\le 2\,\int_{E \setminus E_h} |\varphi(y)|^p\,\left(\int_{\mathbb{R}^N \setminus E} \frac{1}{|x-y|^{N+s\,p}}\,dx\right)dy + \int_{E \setminus E_h} |\varphi(y)|^p\,\left(\int_{E_{h}} \frac{ |\eta_h(x) - \eta_h(y)|^p}{|x-y|^{N+s\,p}}\,dx\right)dy \\
			&:= \mathcal{C}_1 + \mathcal{C}_2.
		\end{split}
		\]
		Since $\varphi$ vanishes on $\mathbb{R}^N \setminus E$, in light of \cite[Theorem 6.97]{Leoni_fractional}, we have  
		\[
		\int_{E} |\varphi(y)|^p\,\left(\int_{\mathbb{R}^N \setminus E} \frac{1}{|x-y|^{N+s\,p}}\,dx\right)dy < \infty.
		\]
		This allows to use the Lebesgue's Dominated Convergence Theorem to infer that 
		\begin{equation} \label{limite-2}
			\lim_{h \to \infty} \mathcal{C}_1 = 0.
		\end{equation}
		For $\mathcal{C}_2$, by using Lemma \ref{leibniz-holder} and \eqref{per-favore-funziona} we have
		\[
		\begin{split}
			\mathcal{C}_2 \le \frac{c_{N,p}}{s\,(1-s)}\,[\varphi]^p_{C^{0,s}(\mathbb{R}^N)}\,\left(A\,h\right)^{s\,p}\, \left(\frac{1}{h}\right)^{s\,p}\,|E \setminus E_h|.
		\end{split}
		\]
		This entails that 
		$
		\lim_{h \to \infty} \mathcal{C}_2 = 0,
		$
	which in combination with \eqref{limite-2} yields 
	$
	\lim_{h \to \infty} \mathcal{A}_2 = 0.
	$
	By spending the last information and \eqref{limite-1}, we can infer that point (ii) holds for the sequence of functions given by \eqref{defi-lip}.  Eventually, by standard approximation arguments, we can modify each $\varphi_h$ so that for the sequence of modified functions (not relabeled) we have $(\varphi_h)_{h  \in \mathbb{N}} \subseteq C^\infty_0(E)$ and (i)-(ii) still hold.  
\end{proof}

\subsection{Basic properties of fractional capacity} \label{subsec:cap}
As anticipated in the Introduction, we will work with the following notion of fractional capacity.
\begin{definition} \label{def:frac-capacity}
	Let $1 \le p < \infty$ and $0 < s < 1$. For every $E \subseteq \mathbb{R}^N$ open set and for every $\Sigma \subseteq E$ compact set, we define the {\it (homogeneous) $(s,p)-$capacity of $\Sigma$ relative to $E$} as
	\[
	\widetilde{{\rm cap}}_{s,p}(\Sigma; E) = \inf_{\varphi \in C^\infty_0(E)} \left\{\iint_{\mathbb{R}^N \times \mathbb{R}^N} \frac{|\varphi(x) - \varphi(y)|^{p}}{|x-y|^{N+s\,p}}\,dxdy : \varphi \ge 1 \mbox{ on } \Sigma \right\}. 
	\] 
	In the sequel, we will often refer to $\widetilde{{\rm cap}}_{s,p}$ just as {\it fractional capacity}.
\end{definition}
Directly from the definition, we record the following scaling property.	
\begin{remark} \label{riscalamento}
	Let $1 \le p < \infty$ and $0 < s < 1$. By using the change of variable formula for multiple integrals in the Definition \ref{def:frac-capacity}, we can infer that 
	\[
	\widetilde{\mathrm{cap}}_{s,p}\left(\overline{B_r}; B_R\right) = r^{N-s\,p}\,\widetilde{\mathrm{cap}}_{s,p}\left(\overline{B_1}; B_{R/r}\right), \quad \mbox{for every } 0 < r < R.
	\]
\end{remark}

\begin{remark} \label{rmk:defi-cap} 
	In the Definition \ref{def:frac-capacity}, it is not restrictive to assume that every feasible competitor $\varphi$ is nonnegative. Indeed, if $\varphi \in C^\infty_0(E)$ is such that $\varphi \ge 1$ on $\Sigma$, by the triangle inequality we have
	\begin{equation*}\label{ok4}
		\left[|\varphi|\right]_{W^{s,p}(\mathbb{R}^N)} \le [\varphi]_{W^{s,p}(\mathbb{R}^N)}.
	\end{equation*}
	By taking  $\widetilde{\varphi_\varepsilon} := f_\varepsilon \circ \varphi$, where
	\[
	f_\varepsilon(t) := \frac{\left(t^2 + \varepsilon^2\right)^\frac{1}{2} - \varepsilon}{\left(1 + \varepsilon^2\right)^{\frac{1}{2}} - \varepsilon}, \qquad \mbox{ for } t \in \mathbb{R},\,\, \varepsilon > 0,
	\] 
	we can verify that $\widetilde{\varphi_{\varepsilon}}$ is an admissible nonnegative competitor in the Definition \ref{def:frac-capacity} and 
	\begin{equation*} \label{admissible-1}
		\lim_{\varepsilon \to 0} [\widetilde{\varphi_\varepsilon}]_{W^{s,p}(\mathbb{R}^N)} = [|\varphi|]_{W^{s,p}(\mathbb{R}^N)},
	\end{equation*}
	by the Lebesgue's Dominated Convergence Theorem.
\end{remark}

A useful equivalent characterization of fractional capacity is given below. For the proof, we share ideas from \cite[Remark 1.4.2]{tesi-francesca}. 

\begin{proposition} \label{prop:equivalent-cap}
	Let $1 \le p < \infty$ and $0 < s < 1$. For every $E \subseteq \mathbb{R}^N$ open set and for every $\Sigma \subseteq E$ compact set, we have
	\[
	\widetilde{{\rm cap}}_{s,p}(\Sigma; E) = \inf_{\varphi \in {\rm Lip}_0(E)}\left\{\iint_{\mathbb{R}^N \times \mathbb{R}^N} \frac{|\varphi(x) - \varphi(y)|^{p}}{|x-y|^{N+s\,p}}\,dxdy : 0 \le \varphi \le 1, \,\, \varphi=1 \mbox{ {\rm on }}\Sigma \right\}. 
	\]   
\end{proposition}
\begin{proof}
	We first claim that 
	\begin{equation} \label{claim-cap}
		\widetilde{{\rm cap}}_{s,p}(\Sigma; E) = \inf_{\varphi \in {\rm Lip}_0(E)}\left\{[\varphi]^p_{W^{s,p}(\mathbb{R}^N)}: \varphi \ge 0,\,\, \varphi \ge 1 \mbox{ on } \Sigma \right\}.
	\end{equation}
	By Remark \ref{rmk:defi-cap}, the inequality
	\begin{equation} \label{ok5}
		\widetilde{{\rm cap}}_{s,p}(\Sigma; E) \ge \inf_{\varphi \in {\rm Lip}_0(E)}\Big\{[\varphi]^p_{W^{s,p}(\mathbb{R}^N)} : \varphi \ge 0,\,\, \varphi \ge 1 \mbox{ on }\Sigma \Big\}.
	\end{equation}
	trivially follows by the inclusion of sets $C^\infty_0(E) \subseteq {\rm Lip}_0(E)$. On the other hand, let $\varphi \in {\rm Lip}_0(E)$ be a nonnegative function such that $\varphi \ge 1$ on $\Sigma$. If $(\rho_\varepsilon)_{\varepsilon > 0}$ is a family of standard mollifiers, we consider the convolution $\varphi_\varepsilon = \varphi \ast \rho_\varepsilon$ and set 
	$
	\widetilde{\varphi_\varepsilon} := \varphi_{\varepsilon}/ \min_{\Sigma} \varphi_\varepsilon, 
	$
	for $\varepsilon > 0.$
	The function $\widetilde{\varphi_\varepsilon}$ belongs to $C^\infty_0(E)$, as long as $\varepsilon > 0$ is sufficiently small, and by construction $\widetilde{\varphi_\varepsilon} \ge 1$ on $\Sigma.$
	Since $\varphi_{\varepsilon}$ converges uniformly  to $\varphi$ 
	and by \cite[Lemma 11]{FSV}, we get 
	\[
	\widetilde{{\rm cap}}_{s,p}(\Sigma; E) \le \lim_{\varepsilon \to 0} [\widetilde{\varphi_\varepsilon}]^p_{W^{s,p}(\mathbb{R}^N)} = \dfrac{[\varphi]^p_{W^{s,p}(\mathbb{R}^N)}}{ \displaystyle \min_{\Sigma} \varphi} \le [\varphi]^p_{W^{s,p}(\mathbb{R}^N)}.
	\]
	Thus, by the arbitrariness of $\varphi$, we get the reverse inequality of \eqref{ok5} and so also \eqref{claim-cap}. To conclude the proof, we observe that by truncating a Sobolev function we can lower the seminorm $[\,\cdot\,]_{W^{s,p}}$, see \cite[Lemma 2.7]{Warma} for instance. Thus, we can restrict the class of admissible competitors function $\varphi$ on the right-hand side to the ones that are $\varphi \le 1$ on $\mathbb{R}^N$.
\end{proof}	

When $p=1$, we have a further characterization for the $(s, 1)-$capacity in terms of the $s-$perimeter. The following result can be seen as a nonlocal analog of \cite[Lemma 2.2.5]{Maz}.  

\begin{proposition}  \label{prop:cap-per}
	Let $0 < s < 1$ and $E \subseteq \mathbb{R}^N$ be an open set. For every $\Sigma \subseteq E$ compact set, we have 
	\[
	\widetilde{{\rm cap}}_{s,1}\left(\Sigma; E\right) = \inf\left\{P_s(\mathcal{O})\,:\, \begin{array}{ll}
		\mathcal{O} \Subset E \mbox{ open set with  }  \\ \mbox{ smooth boundary such that } \Sigma \subseteq \mathcal{O}
	\end{array}\right\}. 
	\]
	In particular, if $\Sigma$ is the closure of a bounded open convex set then 
	\[
	\widetilde{{\rm cap}}_{s,1}\left(\Sigma; E\right) = P_s(\Sigma). 
	\]
\end{proposition}
\begin{proof}
	For simplicity, we say that an open set $\mathcal{O} \Subset E$ is {\it admissible} if it has smooth boundary and $\Sigma \subseteq \mathcal{O}.$ Let $\mathcal{O}$ be admissible, we claim that 
	\begin{equation} \label{claim1-cap-per}
		\widetilde{{\rm cap}}_{s,1}\left(\Sigma; E\right)  \ge \inf\left\{P_s(\mathcal{O})\,:\, \mathcal{O} \mbox{ is admissible } \right\}.
	\end{equation}
	Let $\varphi \in C^\infty_0(E)$ be a nonnegative function such that $\varphi \ge 1$ on $\Sigma$. By using the nonlocal Coarea-type Formula for $[\,\cdot\,]_{W^{s,1}(\mathbb{R}^N)}$ established in \cite{Visintin} (see also \cite[Lemma 10]{AmbDePMar}), we get  
	\[
	\begin{split}
		[\varphi]_{W^{s,1}(\mathbb{R}^N)} = \int_{0}^{\infty} P_s \left(\left\{\varphi > t\right\}\right)\,dt \ge \int_0^1 P_s\left(\left\{\varphi > t\right\}\right)\,dt \ge  \inf\left\{P_s(\mathcal{O})\,:\, \mathcal{O} \mbox{ is admissible } \right\},
	\end{split}
	\]
	where the last inequality follows since for a.e. $t \in (0,1)$ the open set $\left\{\varphi > t\right\}$ has smooth boundary and so it is admissible, in light of Sard's Lemma. By the arbitrariness of $\varphi$, we get our claim \eqref{claim1-cap-per}.  To prove the reverse inequality, take any $\mathcal{O}$ admissible set. We set 
	$
	\delta:= {\rm dist}\left(\Sigma; \partial\mathcal{O}\right)/2 > 0, 
	$ 
	and define $\mathcal{O}' := \left\{x \in \mathcal{O}\,:\, {\rm dist}\left(x; \partial\mathcal{O}\right) > \delta\right\}.$
	By construction $\Sigma \Subset \mathcal{O}'$. Let $(\rho_{\varepsilon})_{\varepsilon > 0}$ be a family of standard mollifiers and set $\varphi_{\varepsilon} := {1}_{\mathcal{O}} \ast \rho_{\varepsilon}$.
	For $\varepsilon > 0$ small enough, we have 
	\begin{equation} \label{prop:ammissibile-generale}
		\varphi_{\varepsilon} \in C^\infty_0(E), \qquad 0 \le \varphi_{\varepsilon} \le 1 \quad \mbox{ on } \mathbb{R}^N, \qquad \mbox{ and } \qquad \varphi_{\varepsilon} =1 \quad \mbox{ on } \mathcal{O}'.
	\end{equation}
	In particular, the last property follows since if $x \in \mathcal{O}'$ then $|B_\varepsilon(x) \setminus \mathcal{O}| = 0$ for $\varepsilon$ small enough, thus we have $\varphi_\varepsilon(x) = 1$, by standard properties of the mollifiers. Then, by Fubini's theorem
	\begin{equation*}\label{ok1}
		\begin{split}
			\widetilde{{\rm cap}}_{s,1}\left(\Sigma; E\right) 
			&\le  \iint_{\mathbb{R}^N \times \mathbb{R}^N}\left(\int_{\mathbb{R}^n} \frac{|1_{\mathcal{O}}(x-z) - 1_{\mathcal{O}}(y-z)|}{|x-y|^{N+s}}\,\rho_\varepsilon(z)\,dz\right) dxdy \\
			&\le \|\rho_{\varepsilon} \|_{L^1(\mathbb{R}^N) } \,\iint_{\mathbb{R}^N \times \mathbb{R}^N}  \frac{|1_{\mathcal{O}}(x) - 1_{\mathcal{O}}(y)|}{|x-y|^{N+s}}\,dxdy = P_s(\mathcal{O}),
		\end{split}
	\end{equation*}
	which entails the reverse inequality of \eqref{claim1-cap-per}, by the arbitrariness of $\mathcal{O}$. 
	\vskip.2cm \noindent
	To prove the last sentence of the statement, let $\Sigma \subseteq E$ be the closure of a bounded open convex set and assume that $0 \in {\rm int}(\Sigma)$. By \cite[Lemma B.2]{FFMMM}, for every $\mathcal{O}$ admissible we have  
	\begin{equation}\label{ok3}
		P_s(\Sigma) = P_s(\Sigma \cap \mathcal{O}) \le P_s(\mathcal{O}).
	\end{equation}
	On the other hand, for $\lambda > 0$ we set 
	\[
	\mathcal{O}_\lambda := (1 + \lambda)\,{\rm int}(\Sigma) = \Big\{(1+ \lambda)\,x\,:\, x \in {\rm int}(\Sigma) \Big\}. 
	\]
	By convexity $\Sigma \subseteq \mathcal{O}_\lambda$ and $\mathcal{O}_\lambda \Subset E$, for $\lambda$ small enough, thus $\mathcal{O}_\lambda$ is admissible. Moreover, by the scaling properties of the $s-$perimeter 
	\[
	P_s(\mathcal{O}_\lambda) = \left(1 + \lambda\right)^{N-s}\,P_s(\Sigma),
	\]
	which converges to $P_s(\Sigma)$, as $\lambda \to 0$. By recalling \eqref{ok3}, this entails that 
	\[
	P_s(\Sigma) = \inf\left\{P_s(\mathcal{O})\,:\, \mathcal{O} \mbox{ is admissible }\right\},
	\]
	and so the desired conclusion follows by the first part of the proof. 
\end{proof}
\newpage
For $p=2$, the following proposition is already contained in \cite[Proposition 2.6 \& Corollary 2.7]{AbFelNor}. 

\begin{proposition} \label{prop:cap-null}
	Let $1 \le p < \infty$,  $0 < s< 1$ and take an exponent $q$ satisfying \eqref{hp:subcritical}. Let $E\subseteq\mathbb{R}^N$ be an open and bounded set, and let $F \subseteq E$ be compact set such that $F \Subset B_R(x_0)$ and 
	\begin{equation} \label{cap-null}
		\widetilde{{\rm cap}}_{s,p}\left(F; B_R(x_0)\right) = 0, 
	\end{equation}
	for some $x_0 \in \mathbb{R}^N$ and $R > 0$. Then we have 
	\[
	\widetilde{W}^{s,p}_0(E \setminus F) = \widetilde{W}^{s,p}_0(E).
	\]
	In particular 
	\[
	\lambda^s_{p, q}(E) = \lambda^s_{p, q}(E \setminus F).
	\]
\end{proposition}   

\begin{proof}
	Without loss of generality, we assume $x_0 = 0$. Since $C^\infty_0(E \setminus F) \subseteq C^\infty_0(E)$, we have 
	\begin{equation} \label{ok7}
		\lambda_{p,q}^s(E) \le \lambda_{p,q}^s(E \setminus F),
	\end{equation}
	by definition. Observe that $\lambda_{p, q}^s(E \setminus F) > 0$, as a consequence of \cite[Lemma 2.4]{BLP} and H\"older's inequality. Then, in order to prove the reverse inequality of \eqref{ok7}, given any $u \in \widetilde{W}_0^{s,p}(E)$ we need to find a sequence of functions in $\widetilde{W}^{s,p}_0(E \setminus F)$ converging to $u$ with respect $[\,\cdot\,]_{W^{s,p}(\mathbb{R}^N)}$. Let $(u_n)_{n \in \mathbb{N}} \subseteq C^\infty_0(E)$ be such that  
	\begin{equation} \label{inizio}
		\lim_{n \to \infty} \|u_n - u\|_{W^{s,p}(\mathbb{R}^N)} = 0. 
	\end{equation} 
	By Proposition \ref{prop:equivalent-cap}, there exists a sequence $(\varphi_m)_{m\in \mathbb{N}} \subseteq {\rm Lip}_0(B_R)$ satisfying
	\begin{equation} \label{prop-app-cap}
		0 \le \varphi_m \le 1\,\, \mbox{ on } \mathbb{R}^N, \qquad \varphi_m = 1\,\, \mbox{ on } F \qquad \mbox{ and } \qquad \lim_{m \to \infty} [\varphi_m]_{W^{s,p}(\mathbb{R}^N)} = 0, 
	\end{equation}  
	for every $m \in \mathbb{N}$. We set 
	\begin{equation*}
		v^{(n)}_m := \Big((1-\varphi_m)\,u_n\Big)_{+}, \qquad \mbox{ for every } m, n \in \mathbb{N}.
	\end{equation*}
	Since $F \Subset E$, we have  $\mathbb{R}^N \setminus \left(E \setminus F\right) = F \cup \left(\mathbb{R}^N \setminus E\right)$. By construction, $v^{(n)}_m \in W^{s,p}(\mathbb{R}^N)$ and $v^{(n)}_m = 0$ on $F \cup \left(\mathbb{R}^N \setminus E\right)$, thus by Lemma \ref{lm:brezis-type} we get that  
	$
	v^{(n)}_m \in \widetilde{W}^{s,p}_0(E \setminus F).
	$
	By Minkowski's inequality and Lemma \ref{leibniz-holder}, we obtain
	\begin{equation*} 
		\begin{split}
			&[v^{(n)}_m - u_n]_{W^{s,p}(\mathbb{R}^N)} \le [\varphi_m\,u_n]_{W^{s,p}(\mathbb{R}^N)} \\ &\le [\varphi_m]_{W^{s,p}(\mathbb{R}^N)}\,\|u_n\|_{L^\infty(\mathbb{R}^N)} + \left(\iint_{\mathbb{R}^N \times \mathbb{R}^N} |\varphi_m(y)|^p\,\frac{|u_n(x) - u_n(y)|^p}{|x-y|^{N+s\,p}}\,dxdy\right)^{\frac{1}{p}} \\
			&\le  [\varphi_m]_{W^{s,p}(\mathbb{R}^N)}\,\|u_n\|_{L^\infty(\mathbb{R}^N)} + \left(\frac{c_{N,p}}{s\,(1-s)}\right)^{\frac{1}{p}}\,\|\varphi_m\|_{L^p(\mathbb{R}^N)}\,\|\nabla u_n\|^s_{L^\infty(\mathbb{R}^N)}\,\|u_n\|^{1-s}_{L^\infty(\mathbb{R}^N)} \\
			&\le [\varphi_m]_{W^{s,p}(\mathbb{R}^N)}\,\left(\|u_n\|_{L^\infty(\mathbb{R}^N)} + \left(\frac{c_{N,p}}{\lambda_{p}^s(B_R)\,s\,(1-s)}\right)^{\frac{1}{p}}\,\|\nabla u_n\|^s_{L^\infty(\mathbb{R}^N)}\,\|u_n\|^{1-s}_{L^\infty(\mathbb{R}^N)}\right),
		\end{split}
	\end{equation*}
	where in the last line we also used Poincar\'e's inequality on balls. From \eqref{prop-app-cap}, we then obtain 
	\begin{equation*}  \label{seminorma-converge}
		\lim_{m \to \infty} [v^{(n)}_m - u_n]_{W^{s,p}(\mathbb{R}^N)} = 0,  \qquad \mbox{ for every } n\in \mathbb{N},
	\end{equation*}
	and so also
	$
	u_n \in \widetilde{W}^{s,p}_0(E \setminus F),
	$
	for every $n \in \mathbb{N}$.
	By \eqref{inizio}, we infer that $u \in \widetilde{W}^{s,p}_0(E \setminus F)$, as wanted. 
\end{proof}

In the next simple proposition, we derive basic estimates between the fractional capacity and its local counterpart/Lebesgue measure. 
\begin{proposition}   \label{prop:cap-s-1}
	Let $1 \le p < \infty$, $0 < s< 1$ and let $E\subseteq\mathbb{R}^N$ be an open set. For every compact set $\Sigma\subseteq E$, we have  
	\begin{equation} \label{cap-cap}
		\widetilde{\operatorname{cap}}_{s,p}(\Sigma;E)\leq \frac{c_{N,p}}{s\,(1-s)} \lambda_{p}(E)^{s-1}\,\operatorname{cap}_{p}(\Sigma;E),
	\end{equation}
	and 
	\begin{equation} \label{cap-vol}
		|\Sigma|\,\lambda_{p}^s(E) \le \widetilde{\operatorname{cap}}_{s,p}(\Sigma; E).
	\end{equation}
\end{proposition}
\begin{proof}
	We assume $\lambda_{p}(E) > 0$, otherwise there is nothing to prove. Let $\varphi\in C^{\infty}_0(E)$ be such that $\varphi\geq1$ on $\Sigma.$ 
	\par 
	We first prove \eqref{cap-cap}. By using \cite[Corollary 2.2]{BPS} and \cite[Proposition 4.2]{BLP}, for the case $p>1$ and for the case $p=1$ respectively, we get  
	\[
	\widetilde{\mathrm{cap}}_{s,p}(\Sigma; E) \le [\varphi]_{W^{s,p}(\mathbb{R}^N)}^p\leq\frac{c_{N,p}}{s\,(1-s)}\|\varphi\|^{(1-s)\,p}_{L^p(\mathbb{R}^N)}\|\nabla\varphi\|^{s\,p}_{L^p(\mathbb{R}^N)} \le \frac{c_{N,p}}{s\,(1-s)} \lambda_{p}(E)^{s-1} \|\nabla \varphi\|^p_{L^p(E)}.
	\]
	By the arbitrariness of $\varphi$, we obtain the conclusion.
	For \eqref{cap-vol},  we have 
	\[
	\lambda_{p}^s(E)\,|\Sigma|\leq \lambda_{p}^s(E)\, \|\varphi\|^p_{L^p(E)} \le [\varphi]_{W^{s,p}(\mathbb{R}^N)}^p,
	\]
	and we conclude, as before.
\end{proof}
\begin{remark} \label{rmk:asym-cap}
	Let $1 \le p < \infty$ and let $E \subseteq \Omega$ be an open set. From the previous result and by Propositions \ref{prop:asym-s-1}-\ref{prop:asym2}, we can infer that
	\begin{equation} \label{eqn:asym-sharp1}
		\widetilde{\operatorname{cap}}_{s,p}\left(\Sigma; E\right) \sim \frac{1}{s},\,\, \mbox{ for } s \searrow 0,  \qquad	\widetilde{\operatorname{cap}}_{s,p}\left(\Sigma; E\right) \sim \frac{1}{1-s}, \,\, \mbox{ for } s \nearrow 1,
	\end{equation}
	for every $\Sigma \subseteq E$ compact set. 
\end{remark}
A first comparison result between local and nonlocal variant of capacitary inradius is derived below.
\begin{proposition} \label{prop:capin-vs-capin}
	Let $1 \le p < \infty$ and $0 < s< 1$. There exists a constant $\beta = \beta(N,p,s) > 0$ such that for every $\Omega\subseteq\mathbb{R}^N$ open set we have 
	\[
	R_{p, \beta \cdot \gamma}(\Omega)\leq R^s_{p, \gamma}(\Omega), \qquad \mbox{ for every }  0<\gamma<\min\left\{1,\,\, \frac{1}{\beta}\right\}.
	\]
\end{proposition}
\begin{proof}
	Let $r>R^s_{p, \gamma}(\Omega)$ then, by using Proposition \ref{prop:cap-s-1} and Remark \ref{riscalamento}, for every $x_0 \in \mathbb{R}^N$ we have 
	\[
	\begin{split}
		\mathrm{cap}_p\left(\overline{B_r(x_0)}\setminus\Omega; B_{2r}(x_0)\right) &\ge c_{N,p}\, s\,(1-s)\,\lambda_{p}\left(B_{2r}\right)^{1-s}\,\widetilde{\operatorname{cap}}_{s,p}\left(\overline{B_r(x_0)}\setminus\Omega; B_{2r}(x_0)\right) \\
		&>  c_{N,p}\, s\,(1-s)\, \lambda_{p}(B_{2})^{1-s}\,\gamma\,\widetilde{\operatorname{cap}}_{s,p}\left(\overline{B_r(x_0)};B_{2r}(x_0)\right) \\
		&= \beta\,\gamma\,\mathrm{cap}_p\left(\overline{B_r(x_0)}; B_{2r}(x_0)\right),
	\end{split}
	\]
	where we set 
	\[
	\beta = \beta(N,p, s) =\frac{ c_{N,p}\,\lambda_{p}(B_{2})^{1-s}}{\mathrm{cap}_p(\overline{B_1}; B_2)}\,s\,(1-s)\, \widetilde{\operatorname{cap}}_{s,p}(\overline{B_1};B_{2}).
	\]
	This entails that $r > R_{p, \beta\cdot \gamma}(\Omega),$ for every $\gamma < \min\left\{1,\,\, 1/\beta\right\},$ so we get the desired inequality, by the arbitrariness of $r$.
\end{proof}

The following technical result will be useful in the sequel. 
\begin{proposition}[Capacity with respect to balls] \label{prop:cap-wrt-balls}
	Let $1 \le p < \infty$ and $0 < s < 1$. Let $\Sigma \subseteq B_r(x_0)$ be a compact set. For every $R > r$, there exists a constant $\mathscr{C} = \mathscr{C}(N, p, s, R/d) > 0$ such that  
	\begin{equation} \label{capacity-wrt-balls}
		\widetilde{\mathrm{cap}}_{s,p}\left(\Sigma; B_R(x_0)\right) \leq \widetilde{\mathrm{cap}}_{s,p}\left(\Sigma; B_r(x_0)\right) \leq \mathscr{C}\,\widetilde{\mathrm{cap}}_{s,p}\left(\Sigma; B_R(x_0)\right),
	\end{equation}
	where $d := \mathrm{dist}(\Sigma; \partial B_r(x_0)) > 0$. Moreover, the function $\tau \mapsto \mathscr{C}(N, p, s, \tau),$ as $\tau > 0$, is increasing. 
\end{proposition} 
\begin{proof}
	The leftmost inequality is trivial. Without loss of generality  we can assume  that $x_0 = 0$. Let $u \in C^\infty_0(B_R)$ be such that $u \geq 1$ on $\Sigma$. We fix $\delta = d/2$ and define a cut-off function $\eta \in \mathrm{Lip}_0(B_{r})$ given by \[
	\eta(x):=\min\left\{\left(\dfrac{(r- \delta) - |x|}{(r-\delta) - (r-d)}\right)_{+}, 1\right\}.
	\]
	In particular 
	\begin{equation} \label{prop-cutoff-cap}
		\eta \equiv 1 \mbox{ on } B_{r-d}, \qquad \eta \equiv 0 \mbox{ on } \mathbb{R}^N \setminus B_{r-\delta}, \qquad \|\nabla \eta\|_{L^\infty(\mathbb{R}^N)} = \dfrac{1}{d -\delta}.
	\end{equation}
	By construction, the function $\psi := \eta\,u$ is admissible to test the definition of relative $p-$capacity. Then, by Minkowski's inequality we have 
	\begin{align*} \nonumber
		\Big(\widetilde{\mathrm{cap}}_{s,p}(\Sigma; B_r)\Big)^{\frac{1}{p}} &\leq \left(\int_{B_r}\int_{B_r} \dfrac{|\psi(x) - \psi(y)|^p}{|x-y|^{N+sp}}dxdy\right)^{\frac{1}{p}} \\&\qquad + 2^{\frac{1}{p}}\left(\int_{B_r} |\psi(x)|^p\left(\int_{\mathbb{R}^N \setminus B_r} \dfrac{dy}{|x-y|^{N+sp}}\right) dx\right)^{\frac{1}{p}} \\
		&:= \mathcal{A}_1 + \mathcal{A}_2.
	\end{align*} 
 We estimate the last two terms separately. For $\mathcal{A}_1$, we add and subtract $\eta(x)u(y)$ and use Minkowski's inequality obtaining
	\begin{align*}
		\mathcal{A}_1 &\leq \left(\int_{B_r} |\eta(x)|^p \int_{B_r}\dfrac{|u(x) - u(y)|^p}{|x-y|^{N+s\,p}} dxdy\right)^{\frac{1}{p}} + \left(\int_{B_r} |u(y)|^p \int_{B_r} \dfrac{|\eta(x) - \eta(y)|^p}{|x-y|^{N+s\,p}} dx dy\right)^{\frac{1}{p}} \\ &\leq
		[u]_{W^{s,p}(B_R)} + \left(\dfrac{c_{N,p}}{s\,(1-s)}\right)^{\frac{1}{p}}\,\|\nabla \eta\|^s_{L^\infty(B_r)}\,\|\eta\|^{1-s}_{L^\infty(B_r)}\,\|u\|_{L^p(B_r)},
	\end{align*}
	where in the last line we also used \eqref{prop-cutoff-cap} and Lemma \ref{leibniz-holder}. In turn, by using the definitions of $\lambda_{p}^s(B_R)$ and $\delta$, we get \begin{align*} 
		\mathcal{A}_1 \le \left(1 + \left(\dfrac{c_{N,p}}{s\,(1-s)}\right)^{\frac{1}{p}} \left(\frac{2\,R}{d}\right)^s \left(\frac{1}{\lambda^s_{p}(B_1)}\right)^{\frac{1}{p}}\right)[u]_{W^{s,p}(\mathbb{R}^N)}. 
	\end{align*}
	We focus now on $\mathcal{A}_2$. By recalling \eqref{prop-cutoff-cap}, we have 
	\begin{align*} 
		\mathcal{A}_2 \le 2^{\frac{1}{p}} \left(\int_{B_{r - \delta}} |u(x)|^p \left(\int_{\mathbb{R}^N \setminus B_r} \dfrac{dy}{|x-y|^{N+s\,p}}\right)\,dx \right)^{\frac{1}{p}} \le  \left(\dfrac{N \omega_N}{s\,p}\right)^{\frac{1}{p}}\,R^{\frac{N}{p}} \left(\dfrac{1}{\delta}\right)^{\frac{N+s\,p}{p}} \,\|u\|_{L^p(B_R)}.
	\end{align*}
	Observe that, in the last inequality we used the following estimate
	\[
	\begin{split}
		\int_{\mathbb{R}^N \setminus B_r} \dfrac{dy}{|x-y|^{N+s\,p}} \leq \left(\dfrac{r}{\delta}\right)^{N+s\,p} \int_{\mathbb{R}^N \setminus B_r} \dfrac{dy}{|y|^{N+s\,p}} = \frac{N\,\omega_N}{s\,p}\,\left(\dfrac{r}{\delta}\right)^{N+ s\,p}\,\dfrac{1}{r^{s\,p}},
	\end{split}
	\]
	which holds since we have
	\[
	|x - y| \geq |y| - |x| \geq |y| - (r-\delta) \geq |y| - \dfrac{r-\delta}{r} |y| = |y| \dfrac{\delta}{r}, \qquad \mbox{ for } x \in B_{r-\delta}, \, y \notin B_{r}.
	\]
	Then, by using the $(s,p)-$Poincar\'e inequality on $B_R$ and by recalling the definition of $\delta$, we get 
	\begin{equation*} 
		\mathcal{A}_2 \le \left(\dfrac{N \omega_N}{s\,p} \frac{1}{\lambda^s_{p}(B_1)}\right)^{\frac{1}{p}} \, \left(\dfrac{2\,R}{d}\right)^{\frac{N+s\,p}{p}} \,[u]_{W^{s,p}(\mathbb{R}^N)}.
	\end{equation*}
	By collecting the previous estimates, we can infer the second inequality in \eqref{capacity-wrt-balls} by taking 
	\begin{equation} \label{costante-cap-wrt-balls}
		\begin{split}
			&\mathscr{C}\left(N, p, s, \frac{R}{d}\right) = \\ 
			&\left(1 +  \left(\dfrac{c_{N,p}}{s(1-s)} \frac{1}{\lambda^s_{p}(B_1)}\right)^{\frac{1}{p}}\left(\frac{2R}{d}\right)^s +  (1-s)^{\frac{1}{p}}\left(\dfrac{N \omega_N}{s(1-s)} \frac{1}{\lambda^s_{p}(B_1)}\right)^{\frac{1}{p}} \left(\dfrac{2R}{d}\right)^{\frac{N+s\,p}{p}}\right)^p,
		\end{split}
	\end{equation}
the proof is thereby complete. 
\end{proof}

\section{Poincar\'e-type estimates on balls}
\label{sec:3}

In the proof of Theorem \ref{thm:upper-bound}, we will crucially exploit two $L^1-W^{s,p}$ Poincar\'e-type estimates on balls, derived in Lemma \ref{poincarè-Palle} and Lemma \ref{lm:bound-l-infinito}  below. We provide such estimates with the correct asymptotic behaviours with respect to the parameter of fractional differentiability. We treat separately the cases $p= 1$ and $1 < p < \infty$, starting from the first one.   

\subsection{A Cheeger-type constant}
Let  $\Omega \subseteq \mathbb{R}^N$ be an open set, we introduce the following {\it Cheeger-type problem}
\[
h_s\left(E; \Omega\right) := \inf \left\{\dfrac{P_s(\mathcal{O})}{\mid \mathcal{O} \cap E|}: \begin{array}{ll}
	\mathcal{O} \Subset \Omega \mbox{ open with } \\ \mbox{ smooth boundary }
\end{array}  \right\},
\]
for every $E \subseteq \Omega$. Observe that, by taking $E = \Omega$, we have $h_s(\Omega; \Omega) = h_s(\Omega)$. 
\begin{proposition}[Poincar\'e vs Cheeger] \label{Cheeger--type constant}
	Let  $\Omega \subseteq \mathbb{R}^N$ be an open set. For every $E \subseteq \Omega$, we have
	\begin{equation*}
		\inf_{\varphi \in C^\infty_0(\Omega)}\left\{\iint_{\mathbb{R}^N \times \mathbb{R}^N} \frac{|\varphi(x) - \varphi(y)|}{|x-y|^{N + s}}\,dxdy : \int_{E} |\varphi|\,dx = 1\right\} = h_s\left(E; \Omega\right).
	\end{equation*}
\end{proposition}
\begin{proof}
	Let $\mathcal{O} \Subset \Omega$ be an open set with smooth boundary. We set $\varphi_{\varepsilon} = 1_{\mathcal{O}} \ast \rho_\varepsilon/\|1_{\mathcal{O}} \ast \rho_{\varepsilon}\|_{L^1(E)},$ where $(\rho_\varepsilon)_{\varepsilon > 0}$ is a family of standard mollifiers. Then, by usual properties of mollifiers 
	\[
	\inf_{\varphi \in C^\infty_0(\Omega)}\Big\{ [\varphi]_{W^{s,1}(\mathbb{R}^N)} : \|\varphi\|_{L^1(E)} = 1 \Big\}  \le \lim_{\varepsilon \rightarrow 0}  \dfrac{[1_{\mathcal{O}} \ast \rho_{\varepsilon}]_{W^{s,1}(\mathbb{R}^N)}}{\|1_{\mathcal{O}} \ast \rho_{\varepsilon}\|_{L^1(E)}} = \dfrac{P_s(\mathcal{O})}{|\mathcal{O} \cap E |}.
	\]
	By the arbitrariness of $\mathcal{O}$, we get a first inequality. On the other hand, let $\varphi \in C^\infty_0(\Omega)$ be a nonnegative function. We can easily observe that this is not restrictive. Then, by the Coarea Formula for $[\,\cdot\,]_{W^{s,1}(\mathbb{R}^N)}$ and Sard's Lemma
	\[
	\begin{split}
		[\varphi]_{W^{s,1}(\mathbb{R}^N)} = \int_0^\infty P_s\left(\left\{\varphi > t\right\}\right)\,dt \ge h_s\left(E; \Omega\right)\,\int_0^\infty |\left\{\varphi > t\right\} \cap E|\,dt = h_s\left(E; \Omega\right)\,\|\varphi\|_{L^1(E)},
	\end{split}
	\]
	where the last equality follows by Cavalieri's formula. This entails the reverse inequality as wanted. 
\end{proof}
\newpage
As a byproduct of the foregoing result, we infer the following $L^1-W^{s,1}$ Poincar\'e-type inequality 

\begin{lemma} \label{poincarè-Palle} 
	Let $0 < r \le R < \infty$, we have \[
	\inf_{\varphi \in C^\infty_0(B_R)}\left\{\iint_{\mathbb{R}^N \times \mathbb{R}^N} \frac{|\varphi(x) - \varphi(y)|}{|x-y|^{N + s}}\,dxdy : \int_{B_r} |\varphi|\,dx = 1\right\} = \frac{P_s(B_r)}{|B_r|},
	\]
	that is, $E = B_r$ is the unique optimal set for the Cheeger-type problem defining $h_s\left(B_r; B_R\right)$. In particular, we have the following sharp estimate
	\begin{equation}
		\fint_{B_r} |\varphi|\,dx \le \frac{1}{P_s(B_r)}\,\iint_{\mathbb{R}^N \times \mathbb{R}^N} \frac{|\varphi(x) - \varphi(y)|}{|x-y|^{N + s}}\,dxdy, \quad \mbox{ for } \varphi \in C^\infty_0(B_R). 
	\end{equation}
\end{lemma}
\begin{proof}
	In light of Proposition \ref{Cheeger--type constant}, we only need to prove that $h_s\left(B_r; B_R\right) = P_s(B_r)/|B_r|.$
	Let $\mathcal{O} \Subset B_R$ be an open set with smooth boundary.  By the {\it strict Polya-Sz\"ego-type inequality} \cite[Theorem A.1]{FS}, we have \[ 
	P_s(\mathcal{O}^*)\le P_s(\mathcal{O}),
	\] 
	where $\mathcal{O}^*$ is the ball centered at the origin with $|\mathcal{O}^*|=|\mathcal{O}|$, and the equality holds if and only if $\mathcal{O} = \mathcal{O}^*$, up to translations. Let $\rho^*$ be the radius of $\mathcal{O}^*$, if $\rho^* \leq r$ we have $\mathcal{O}^* \subseteq B_r$ and so \[
	|\mathcal{O}^* \cap B_r| = |\mathcal{O}^*| = |\mathcal{O}| \geq |\mathcal{O} \cap B_r|.
	\]
	On the other hand, if $\rho^* > r$ we have that $B_r \subseteq \mathcal{O}^*$ and so again \[
	|\mathcal{O}^* \cap B_r| = |B_r| \geq |\mathcal{O} \cap B_r|.
	\]
	In particular \[
	\dfrac{P_s(\mathcal{O})}{|\mathcal{O} \cap B_r|} \geq \dfrac{P_s(\mathcal{O}^*)}{|\mathcal{O}^* \cap B_r|},
	\]
	with equality if and only if $\mathcal{O} = \mathcal{O}^*$, up to translations. 
	Thus, we can restrict the class of competitors for $h_s\left(B_r; B_R\right)$ to the balls centered at $0$ and with radius $0 < \rho < R$. Let $B_\rho$ one of these balls, by the scaling properties of $P_s$ we have 
	\[
	\dfrac{P_s\left({B}_\rho\right)}{\left|B_\rho \cap B_{r}\right|}=
	\begin{cases}
		\dfrac{\rho^{N-s}\,P_s\left(B_1\right)}{r^N\,\omega_N}, \quad &\mbox{ if } r \leq \rho < R, \\ 
		\\
		\dfrac{\rho^{N-s}\,P_s\left(B_1\right)}{\rho^N\,\omega_N}, \quad &\mbox{ if } 0 < \rho < r.
	\end{cases}
	\]
	It is immediate to see that the infimum is attained when $\rho=r$, from which we can also infer the uniqueness of the optimal set, as claimed.
\end{proof}	
\subsection{A Poincar\'e-type constant}
If $1 < p < \infty$, we take advantage of the forthcoming nonlocal torsion-like PDE.  
\begin{proposition}\label{cor:Linfty-bound}
	Let $1 < p< \infty$ and $0 < s < 1$ be such that $s\,p \le N$. For every $0 < r \le R$, there exists a unique $V \in \widetilde{W}^{s,p}_0(B_R)$ weak solution of the following equation
	\begin{equation*} \label{euler-lagrange-torsion}
		\begin{cases}
			\begin{aligned}
				(-\Delta_p)^s V &= 1_{B_r}, \qquad &&\mbox{ in } B_R,\\
				\\
				V &=0, \qquad &&\mbox{ in } \mathbb{R}^N \setminus B_R,
			\end{aligned}
		\end{cases}
	\end{equation*}
		that is $V$ satisfies
	\begin{equation} \label{eqn-statement}
		\iint_{\mathbb{R}^N \times \mathbb{R}^N} \frac{|V(x) - V(y)|^{p-2}(V(x) - V(y))(\varphi(x) - \varphi(y))}{|x-y|^{N+s\,p}}dxdy = \int_{B_r} \varphi dx, \mbox{ for } \varphi \in C^\infty_0(B_R).
	\end{equation}
	Moreover, $V \ge 0$ and we have the following estimates
	\begin{itemize}
\item if $s\,p < N$,
\begin{equation} \label{torsione}
	[V]_{W^{s, p}(\mathbb{R}^N)} \le \left(|B_r|^{\frac{s\,p}{N}-1+p} \,\mathcal{S}_{N, p, s}\right)^{\frac{1}{p\,(p-1)}};
\end{equation}
\vskip.2cm \noindent
\item if $s\,p=N$,
\begin{equation} \label{torsione-conforme}
	[V]_{W^{s, \frac{N}{s}}(\mathbb{R}^N)} \le \left(\frac{|B_R|^{\frac{N}{s}}}{\omega_N^{\frac{N}{s}}\,\lambda^s_{\frac{N}{s}, 1}(B_1)}\right)^{\frac{1}{\frac{N}{s}\,(\frac{N}{s}-1)}}.
\end{equation}
\end{itemize}
\end{proposition}
\begin{proof}
	The energy functional associated to the equation is 
	\[
	\mathfrak{F}(u) := \frac{1}{p}\,\iint_{\mathbb{R}^N \times \mathbb{R}^N} \frac{|u(x) - u(y)|^p}{|x-y|^{N+s\,p}} \,dxdy - \int_{B_r} u\,dx, \quad \mbox{ for } u \in \widetilde{{W}}^{s,p}_0(B_R),
	\]
	and every function $u$ satisfying \eqref{eqn-statement} is equivalently characterized as a minimum of $\mathfrak{F}$. 
\par 
	{\it Existence \& Uniqueness} of a nonnegative function $V \in \widetilde{{W}}^{s,p}_0(B_R)$ attaining the infimum of $\mathfrak{F}$ are direct consequences of the {\it Direct Method in the Calculus of Variations} and of the strict convexity of $\mathfrak{F}$, the details are left to the reader.
	\par 
	If $s\,p < N$, since $V$ is solution we can use \eqref{eqn-statement} with $\varphi = V$, by density, and by using also the fractional Sobolev's inequality we get
	\[
	\begin{aligned}
		[V]_{W^{s, p}(\mathbb{R}^N)}^p=\int_{B_r} V\,d x \leq |B_r|^{\frac{1}{(p^*_s)'}}\,\|V\|_{L^{p^*_s}(B_R)} & \leq |B_r|^{\frac{1}{(p^*_s)'}}\,\mathcal{S}_{N, p, s}^{\frac{1}{p}}\,[V]_{W^{s,p}(\mathbb{R}^N)},
	\end{aligned}
	\]
	from which we infer \eqref{torsione}. If $s\,p= N$, since $|B_R| < \infty$ we can infer that $\lambda_{\frac{N}{s},1}^s(B_R) > 0$, by \cite[Lemma 2.4]{BLP} and H\"older's inquality. Thus, by reasoning as before but using the $\left(\frac{N}{s}, 1\right)-$Poincar\'e inequality in place of the fractional Sobolev's inequality, we can infer the claimed estimate \eqref{torsione-conforme}.
\end{proof}

As a byproduct, we infer the following $L^1-W^{s,p}$ Poincar\'e-type estimate.
\begin{lemma} \label{lm:bound-l-infinito}
	Let $1 < p < \infty$ and $0 < s < 1$ be such that $s\,p \le N$. For every $0 < r \le R$, we have
	\vskip.2cm \noindent
	\begin{itemize} 
		\item if $s\,p<N$, 
		\begin{equation} \label{ineq:cap-potential}
			\left(\fint_{B_r} |\varphi|\,dx\right)^p \le \mathcal{S}_{N,p,s}\,|B_r|^{\frac{s\,p}{N} - 1}\iint_{\mathbb{R}^N \times \mathbb{R}^N}\frac{|\varphi(x) - \varphi(y)|^p}{|x-y|^{N+s\,p}}\,dxdy,\quad \mbox{ for } \varphi \in C^\infty_0(B_R),
		\end{equation}
		where $\mathcal{S}_{N,p,s}$ is the sharp constant in the fractional Sobolev's inequality;
		\vskip.2cm 
		\item if $s\,p=N$, 
		\begin{equation} \label{ineq:cap-potential-sp=N}
			\left(\fint_{B_r} |\varphi|\,dx\right)^{\frac{N}{s}} \le \frac{1}{K_{N,s}}\left(\frac{|B_R|}{|B_r|}\right)^{\frac{N}{s}}\iint_{\mathbb{R}^N \times \mathbb{R}^N}\,\frac{|\varphi(x) - \varphi(y)|^{\frac{N}{s}}}{|x-y|^{2\,N}}\,dxdy,\quad \mbox{ for }  \varphi \in C^\infty_0(B_R),
		\end{equation}
		where $K_{N,s} = \omega_N^{\frac{N}{s}}\, \lambda^s_{\frac{N}{s}, 1}(B_1)$. 
	\end{itemize}
\end{lemma}
\begin{proof}
	Let $V$ be as in Proposition \ref{cor:Linfty-bound}. We claim that 
	\begin{equation} \label{claim-torsione}
		\sup_{\varphi \in C^\infty_0(B_R)} \left\{\left(\displaystyle \int_{B_r} |\varphi|\,dx\right)^p: [\varphi]_{W^{s,p}(\mathbb{R}^N)} = 1 \right\} = [V]_{W^{s,p}(\mathbb{R}^N)}^{p\,(p-1)}.
	\end{equation}
	Indeed, by density, we can use the function $\varphi = V$ in \eqref{eqn-statement} obtaining
	\begin{equation} \label{eqn:from-el}
		\int_{B_r} V\,dx = [V]^p_{W^{s,p}(\mathbb{R}^N)},
	\end{equation} 
	which yields
	\begin{equation} \label{ineq:torsion}
		\sup_{\varphi \in C^\infty_0(B_R)} \left\{\left(\displaystyle \int_{B_r} |\varphi|\,dx\right)^p: [\varphi]_{W^{s,p}(\mathbb{R}^N)} = 1 \right\} \ge [V]_{W^{s,p}(\mathbb{R}^N)}^{p\,(p-1)}.
	\end{equation}
	On the other hand, by using again \eqref{eqn-statement} and H\"older's inequality  
	\[
	\begin{split}
		\int_{B_r} |\varphi|\,dx &=  \iint_{\mathbb{R}^N \times \mathbb{R}^N} \frac{|V(x) - V(y)|^{p-2}(V(x) - V(y))\,\left(|\varphi(x)| - |\varphi(y)|\right)}{|x-y|^{N+s\,p}}\,dxdy \\
		&= \iint_{\mathbb{R}^N \times \mathbb{R}^N} \frac{|V(x) - V(y)|^{p-2}(V(x) - V(y))}{|x-y|^{N\,\frac{p-1}{p} + s\,(p-1)}}\,\frac{|\varphi(x)| - |\varphi(y)|}{|x-y|^{\frac{N}{p}+s}}\,dxdy \\
		&\le [V]_{W^{s,p}(\mathbb{R}^N)}^{p-1}\,[\varphi]_{W^{s,p}(\mathbb{R}^N)}, 
	\end{split}
	\]
	for every $\varphi \in C^\infty_0(B_R)$. This entails the reverse inequality of \eqref{ineq:torsion} and thus \eqref{claim-torsione}. The claimed estimates then follow by spending \eqref{torsione} and \eqref{torsione-conforme} in \eqref{claim-torsione}.
\end{proof}

\section{Poincar\'e inequalities} \label{sec:4}
In this section we derive new {\it fractional Maz'ya-Poincar\'e-type inequalities}, see Lemma \ref{lm:mazya-poin} and Lemma \ref{lm:mazya-poin2} below. As anticipated in the Introduction, they involve  the presence of an {\it asymmetric seminorm} on strips (see Remark \ref{rmk:seminorma-asimmetrica}) on the right-hand side and display sharp limiting behaviours for $s \searrow 0$ and $s \nearrow 1$. 
\vskip.2cm \noindent  
All the inequalities of this section will be stated for smooth functions with compact support, nevertheless they can be extended to every function in $W^{s,p}(\mathbb{R}^N)$ by density (see for instance \cite[Theorem 6.66]{Leoni_fractional}). \par 
For the forthcoming technincal result, we  take advantage of the $K-$method of the {\it Real Interpolation Theory} applied to fractional Sobolev spaces, for which we refer to \cite{BS}, \cite[Chapter 16]{Leoni}, \cite[Chapter 12]{Leoni_fractional} and to the references therein contained.      
\begin{proposition} \label{k-functional}
	Let $1 \le p < \infty$ and $0 < s < 1$. Let $x_0 \in \mathbb{R}^N$ and $r > 0$.
	For every $u \in C^\infty_0(\mathbb{R}^N)$ we have 
	\begin{equation}  \label{eqn:clam-1-dir-der}
		\int_{0}^{\infty}\left(\frac{K(t,u)}{t^s}\right)^p \frac{dt}{t} \le A_{N,p}\,\iint_{B_r(x_0) \times \mathbb{R}^N} \frac{|u(x) - u(y)|^p}{|x-y|^{N+s\,p}}\,dxdy,
	\end{equation}	
	where $K(t,u)$ is the $K-${\it functional}\footnote{Since in this context there is no ambiguity regarding the choice of the interpolation spaces, with an abuse of notation, we will denote the $K-$functional 	$K\left(t,u, L^p(B_r(x_0)\right), \mathcal{D}^{1,p}_0(B_r(x_0)))$ only with $K(t,u).$} given by
	\begin{equation*} \label{eqn:k-functional}
		K\left(t,u, L^p(B_r(x_0)\right), \mathcal{D}^{1,p}_0(B_r(x_0))) := \min_{v \in W^{1,p}(B_r(x_0))} \Big[\|u-v\|_{L^p(B_r(x_0))} + t\|\nabla v\|_{L^p(B_r(x_0))}\Big], \quad t \in [0, \infty).
	\end{equation*}
\end{proposition}
\begin{proof}
	Without loss of generality, we can assume $x_0=0$. We prove the claimed inequality \eqref{eqn:clam-1-dir-der} along the lines of \cite[Proposition 4.5]{BS}. 
	Let $u \in C^\infty_0(\mathbb{R}^N)$, we set 
	\[
	U(h) = \left(\int_{B_{r}}|u(x+h) - u(x)|^p dx\right)^\frac{1}{p}, \qquad \mbox{ for } h \in \mathbb{R}^N.
	\]
	The rightmost term in \eqref{eqn:clam-1-dir-der} is then given by
	\[
	\iint_{B_r \times \mathbb{R}^N} \frac{|u(x) -u(x+h)|^p}{|h|^{N+s\,p}}\,dx\,dh = \int_{\mathbb{R}^N}\frac{U(h)^p}{|h|^{N+s\,p}}\,dh.
	\]
	We also set 
	\[
	\overline{U}(\varrho) := \fint_{\partial B_{\varrho}} Ud\mathcal{H}^{N-1}, \qquad \mbox{ for } \varrho > 0.
	\]
	By Jensen's inequality 
	we get
	\begin{equation} \label{ineq:jensen-dir-der}
		\begin{split}
			\int_{0}^{\infty}\overline{U}^p\frac{d\varrho}{\varrho^{1+s\,p}} &\le \frac{1}{N\omega_N}\int_{0}^{\infty}\left(\int_{\partial B_\varrho} U^p\, d\mathcal{H}^{N-1}\right)\,\frac{d\varrho}{\varrho^{N+s\,p}} \\ 
			&= \frac{1}{N\omega_N} \int_{\mathbb{R}^N}\frac{U(h)^p}{|h|^{N+s\,p}}\,dh \\
			&= \frac{1}{N\omega_N} \iint_{B_r \times \mathbb{R}^N} \frac{|u(x) - u(x+h)|^p}{|h|^{N+s\,p}}\,dxdh.
		\end{split}
	\end{equation}
	We introduce the function \[
	\psi(x) = \frac{N+1}{\omega_N}\big(1-|x|\big)_{+}, \qquad \mbox{ for } x \in \mathbb{R}^N,
	\]
	which is Lipschitz with compact support and with $L^1-$norm normalized to $1$. We set
	\[
	\psi_t(x) = \frac{1}{t^N}\,\psi\left(\frac{x}{t}\right), \qquad \mbox{ for } t > 0,
	\]
	this function is supported on $\overline{B_t}$ and still have unitary $L^1-$norm. By standard properties of convolutions,
	we have that $\psi_t \ast u \in W^{1,p}(B_r)$, thus 
	\[
	K(t,u) \le \|u - \psi_t \ast u\|_{L^p(B_r)} + t\,\|\nabla \left( \psi_t \ast u\right)\|_{W^{1,p}(B_r)}, \qquad \mbox{ for } t > 0.
	\]
	We estimate separately the last two terms. For the first one, we use the integral Minkowski's inequality, see \cite[Theorem 2.4]{LL} for instance, with data
	\[
	f(x,y):= u(x) - u(x-y), \quad \nu(dy):= \psi_t(y)\,dy, \quad \mu(dx):= dx, \quad \Omega:= B_r \quad \Gamma:= B_t,
	\]
	by using the notation therein contained. Along with Fubini's Theorem, we get 
	\[
	\begin{split}
		\| u - \psi_t \ast u\|_{L^p(B_r)} &= \left(\int_{B_r}\left(\int_{B_t} (u(x) - u(x-y))\,\psi_t(y)\,dy\right)^p dx\right)^{\frac{1}{p}} \\
		&\le \int_{B_t}\left(\int_{B_r}\left(u(x) - u(x-y)\right)^pdx\right)^{\frac{1}{p}}\,\psi_t(y)\,dy \\
		&= \int_{B_t} U(-y)\,\psi_t(y)\,dy \le \frac{N(N+1)}{t^N} \int_{0}^{t} \overline{U}\varrho^{N-1}\,d\varrho \\
		&\le \frac{N(N+1)}{t}\int_{0}^{t} \overline{U}\,d\varrho.
	\end{split}
	\]
	For the second term, by the Divergence Theorem  
	\[
	\int_{B_t} \nabla \psi_t(y)\,dy = 0,
	\]
	thus
	\[
	\nabla \left(\psi_t \ast u\right) = (\nabla \psi_t) \ast u = \int_{B_t} \nabla \psi_t(y)\left(u(x-y) - u(x)\right)dy.
	\]
	Then, by using the integral Minkowski's inequality, we get \[
	\begin{split}
		\|\nabla(\psi_t \ast u)\|_{L^p(B_r)}&= \left(\int_{B_r}\left(\int_{B_t} \nabla \psi_t(y)\left(u(x-y) - u(x)\right)dy\right)^{p} dx\right)^{\frac{1}{p}} \\
		&\le \int_{B_t} \left(\int_{B_r} |u(x-y) - u(x)|^pdx\right)^{\frac{1}{p}} \left|\nabla \psi_t(y)\right|dy \\
		&\le \frac{N+1}{\omega_N\,t^{N+1}} \int_{B_t} U(-y)\,dy \le \frac{N\,(N+1)}{t^2} \int_{0}^{t} \overline{U}\,d\varrho.
	\end{split}
	\]
	By collecting the previous estimates, we infer that 
	\begin{equation*}
		K(t, u) \le \|u - \psi_t \ast u\|_{L^p(B_r)} + t\,\|\nabla \left( \psi_t \ast u\right)\|_{W^{1,p}(B_r)} \le  \frac{2N(N+1)}{t} \int_{0}^t \overline{U}d\varrho, \mbox{ for } t > 0,
	\end{equation*}
	yielding 
	\[
	\int_{0}^{\infty} \left(\frac{K(t,u)}{t^s}\right)^p \frac{dt}{t} \le \left(2N(N+1)\right)^p\int_{0}^{\infty} \left(\frac{1}{t}\int_{0}^{t} \overline{U}d\varrho\right)^p\frac{dt}{t^{1+s\,p}}.
	\]
	To estimate the last term, we use the one-dimensional Hardy's inequality (see for instance \cite[Theorem C.41]{Leoni}) obtaining
	\[
	\begin{split}
		\int_{0}^{\infty} \left(\frac{K(t,u)}{t^s}\right)^p \frac{dt}{t} &\le \left(\frac{2N(N+1)}{s+1}\right)^p\int_0^{\infty} \overline{U}^p\frac{dt}{t^{1+s\,p}}\\
		&\le  \frac{\left(N+1\right)^p\,N^{p-1}}{\omega_N}\iint_{B_r \times \mathbb{R}^N} \frac{|u(x) - u(x+h)|^p}{|h|^{N+s\,p}}dxdh,
	\end{split}
	\]
	where in the last line we used \eqref{ineq:jensen-dir-der}. The proof is thereby complete. 
\end{proof}

We also need the following Poincar\'e-Wirtinger-type estimate, which is of independent interest. 

\begin{lemma}[Fractional Poincar\'e-Wirtinger inequality] \label{lm:poincare-wirtinger}
	Let $1 \le p < \infty$ and $0 < s < 1$. Let $x_0 \in \mathbb{R}^N$ and $r > 0$.
	For every $u \in C^\infty_0(\mathbb{R}^N)$ we have 
	\begin{equation*}\label{ineq:poin-wirtinger}
		\|u - \mathrm{av}(u; B_r(x_0))\|^{p}_{L^p(B_r(x_0))} \le \mathcal{W}_{N,p}\,s\,(1-s)\,r^{s\,p}\,\iint_{B_r(x_0) \times \mathbb{R}^N} \frac{|u(x) - u(y)|^p}{|x-y|^{N+s\,p}}\,dxdy,
	\end{equation*}
	for some $\mathcal{W}_{N, p} > 0$.
\end{lemma}
\begin{proof} 
	Without loss of generality, we can assume $x_0 = 0$. We take $v \in W^{1,p}(B_r)$ and $t > 0$. By the triangle inequality and Jensen's inequality 
	\[
	\begin{split}
		\|u - {\rm av}(u; B_r)\|_{L^p(B_r)} &\le \|u-v\|_{L^p(B_r)} + \|v - {\rm av}(v; B_r)\|_{L^p(B_r)} + \|{\rm av}(u; B_r) - {\rm av}(v; B_r)\|_{L^p(B_r)}\\
		&\le 2\,\|u - v\|_{L^p(B_r)} + \|v - {\rm av}(v; B_r)\|_{L^p(B_r)} \\
		&\le \max\left\{2, \mu_{N,p} \right\} \left( \|u - v\|_{L^p(B_r)}	+ r\,\|\nabla v\|_{L^p(B_r)} \right),
	\end{split}
	\]
	for some $\mu_{N,p} > 0$, where we used the classical Poincar\'e-Wirtinger inequality on balls. Thus
	\[
	\|u - {\rm av}(u; B_r)\|^p_{L^p(B_r)} \le c_{N,p}\,K(r, u)^p \le c_{N,p}\,K(t, u)^p, \quad \mbox{ for } t> r,
	\]
	where we set $c_{N,p} := \left(\max\{2, \mu_{N,p}\}\right)^p$.
	This entails that 
	\begin{equation} \label{dall-infinito}
		\begin{split}
			\|u - {\rm av}(u; B_r)\|^p_{L^p(B_r)}\,\frac{1}{s\,p} \left(\frac{1}{r^{s\,p}}\right) &= \|u - {\rm av}(u; B_r)\|^p_{L^p(B_r)}\,\int_{r}^{\infty}\frac{dt}{t^{1+s\,p}}\\
			&\le c_{N,p}\,\int_{r}^{\infty} \left(\frac{K(t,u)}{t^s}\right)^p \frac{dt}{t}.
		\end{split}
	\end{equation}
	On the other hand, by arguing as in the beginning of the proof, for every $0 < t \le r$ we have 	
	\[
	\begin{split}
		t\,\|u - \mathrm{av}(u; B_r)\|_{L^p(B_r)} &\le 4r\,\|u - v\|_{L^p(B_r)} + t\,\|v-\mathrm{av}(v;B_r)\|_{L^p(B_r)} \\
		&\le  d_{N,p}\,r\left(\|u-v\|_{L^p(B_r)} + t\,\|\nabla v\|_{W^{1,p}(B_r)}\right),
	\end{split}
	\]
	where we set $d_{N,p} := \max\{4, \mu_{N,p}\}$, and we used again the classical Poincar\'e-Wirtinger inequality on balls. This yields 
	\[
	t^p\,\|u - \mathrm{av}(u; B_r)\|^p_{L^p(B_r)} \le d_{N,p}^p\,r^{p}\,K(t,u)^p, \qquad \mbox{ for } 0 < t \le r, 
	\]
	and so 
	\begin{equation} \label{al-finito}
		\begin{split}
			\|u - \mathrm{av}(u; B_r)\|^p_{L^p(B_r)} \frac{r^{(1-s)\,p}}{(1-s)\,p} &= \|u - \mathrm{av}(u; B_r)\|^p_{L^p(B_r)}\,\int_{0}^{r} \frac{t^p}{t^{1+s\,p}}dt \\
			&\le d_{N,p}^p\,r^p\,\int_{0}^r \left(\frac{K(t, u)}{t^s}\right)^p \frac{dt}{t}. 
		\end{split}
	\end{equation}
	By \eqref{dall-infinito} and \eqref{al-finito}, we get
	\[
	\|u - {\rm av}(u; B_r)\|^p_{L^p(B_r)} \left(\frac{1}{(1-s)\,p} + \frac{1}{s\,p}\right) \left(\frac{1}{r}\right)^{s\,p} \le a_{N,p} \int_{0}^{\infty} \left(\frac{K(t, u)}{t^s}\right)^p \frac{dt}{t},
	\]  
	where we set $a_{N,p}:= \max\left\{c_{N,p}^p, d_{N,p}^p\right\}$. By spending \eqref{eqn:clam-1-dir-der}, we eventually conclude the proof. 
\end{proof}

\begin{remark}[Asymptotic sharpness]
	By Remark \ref{rmk:seminorma-asimmetrica}, the estimate established in Lemma \ref{lm:poincare-wirtinger} displays sharp limiting behaviours with respect to the fractional differentiability parameter $s$. In other words, by taking the $\liminf$ for $s \nearrow 1$, we get the non-trivial Poincar\'e-Wirtinger inequality 
	\[
	\|u - {\rm av}\left(u; B_r(x_0)\right)\|_{L^p(B_r(x_0))} \le C_{N,p}\,r\,\|\nabla u\|_{L^p(\mathbb{R}^N)}, \quad \mbox{ for } u \in C^\infty_0(\mathbb{R}^N),
	\] 
	while taking the $\liminf$ for $s \searrow 0$, we get
	\[
	\|u - {\rm av}\left(u; B_r(x_0)\right)\|_{L^p(B_r(x_0))} \le C_{N,p}\,\|u\|_{L^p(\mathbb{R}^N)}, \quad \mbox{ for } u \in C^\infty_0(\mathbb{R}^N).
	\] 
\end{remark}

We are now able to prove a new fractional Maz'ya-Poincar\'e-type inequality. 
Its special form has been kindly suggested to the authors by Lorenzo Brasco.
\begin{lemma}[Maz'ya-Poincar\'e] \label{lm:mazya-poin}
	Let $1 \le p < \infty$ and $0 < s < 1$. Let $x_0 \in \mathbb{R}^N$ and $0 < r < R$ and let $\Sigma \subseteq  \overline{B_{r}(x_0)}$ be a compact set. For every $u \in C^\infty_0(\mathbb{R}^N)$ such that $u \equiv 0$ on $\Sigma$, we have 
	\begin{equation*}
		\frac{M}{r^{\frac{N}{p}}} \left(\widetilde{\rm cap}_{s,p}\left(\Sigma; B_{R}(x_0)\right)\right)^{\frac{1}{p}}  \|u\|_{L^p(B_r(x_0))} \le \left(\iint_{B_{R}(x_0) \times \mathbb{R}^N} \frac{|u(x) - u(y)|^p}{|x-y|^{N+s\,p}}dxdy\right)^{\frac{1}{p}},
	\end{equation*}
	where $M=M\left(N,p,R/r\right) > 0$. Moreover
\[
0 < \displaystyle \lim_{R/r \to 1} \frac{M\left(N,p, R/r\right)}{\left(R/r - 1\right)^{\frac{N}{p} + 1}} < \infty, \quad \mbox{ and } \quad 	0 < \lim_{R/r \to \infty}\left(\frac{R}{r}\right)^{\frac{N}{p}}\,M\left(N,p, \frac{R}{r}\right) < \infty.
\]
\end{lemma}
\begin{proof}
	Without loss of generality, we can assume $x_0=0$ and 
	\begin{equation} \label{hp:normalizzazione}
		\fint_{B_r} |u|^p dx = 1.
	\end{equation} 
	We set 
	$
	\delta:= (R-r)/2.
	$
	We take $\eta \in {\rm Lip}_0(B_{R})$ such that 
	\begin{equation} \label{eqn:cut-off-prop}
		0 \leq \eta \leq 1, \quad \eta \equiv 1 \mbox{ on } \overline{B_{r}}, \quad \eta \equiv 0 \mbox{ on } B_{R} \setminus B_{R-\delta}, \quad |\nabla \eta| \leq \frac{1}{(R- \delta)-r}, 
	\end{equation}
	and we set 
	\[
	\psi:= \Big((1-u)\,\eta\Big)_{+} \in {\rm Lip}_0(B_{R}). 
	\]
	By Proposition \ref{prop:equivalent-cap}, this function is a feasible competitor for the definition of $\widetilde{\rm cap}_{s,p}$, thus 
	\[
	\begin{split}
		\left(\widetilde{\rm cap}_{s,p}\left(\Sigma; B_{R}\right)\right)^{\frac{1}{p}} &\le \left(\iint_{B_{R} \times \mathbb{R}^N} \frac{|\psi(x) - \psi(y)|^p}{|x-y|^{N+s\,p}}dxdy\right)^{\frac{1}{p}} 
		+ \left(\iint_{\left(\mathbb{R}^N \setminus B_{R}\right) \times \mathbb{R}^N} \frac{|\psi(x) - \psi(y)|^{p}}{|x-y|^{N+s\,p}}dxdy\right)^{\frac{1}{p}} \\ &:= \mathcal{A}_1 + \mathcal{A}_2,
	\end{split}
	\] 
	where we used the subadditivity of the function $\tau \mapsto \tau^{\frac{1}{p}}$. To estimate the first term, we use \cite[Lemma 2.6(b)]{Warma}, we add and subtract $\eta(x)(1-u(y))$ and apply Minkowski's inequality obtaining 
	\[
	\begin{split}
		\mathcal{A}_1 &\le \left(\iint_{B_{R} \times \mathbb{R}^N} \frac{|u(x) - u(y)|^p}{|x-y|^{N+s\,p}}dxdy \right)^{\frac{1}{p}} + \left(\int_{B_{R}} |1-u(y)|^p\left(\int_{\mathbb{R}^N} \frac{|\eta(x) - \eta(y)|^p}{|x-y|^{N+s\,p}}dx\right)dy\right)^{\frac{1}{p}} \\
		&\le \left(\iint_{B_{R} \times \mathbb{R}^N} \frac{|u(x) - u(y)|^p}{|x-y|^{N+s\,p}}dxdy \right)^{\frac{1}{p}}  + \left(\frac{c_{N,p}}{s\,(1-s)}\right)^{\frac{1}{p}} \left(\frac{1}{(R-\delta)-r}\right)^s \|1- u\|_{L^p(B_{R})},
	\end{split}
	\]
	where in the last line we also used Lemma \ref{leibniz-holder}.
	\par 
	To estimate $\mathcal{A}_2$, we observe that for $x \in B_{R- \delta}$ and $y \in \mathbb{R}^N \setminus B_{R}$ we have
	\[
	|x - y| \ge |y| - |x| \ge |y| - \left(R-\delta\right) \ge |y| - \frac{R-\delta}{R}|y| = \frac{\delta}{R}|y|,
	\]
	thus a direct computation leads to 
	\[
	\int_{\mathbb{R}^N \setminus B_{R}}\frac{dy}{|x-y|^{N+s\,p}} \le \frac{N\,\omega_N}{s\,p} \left(\frac{R}{\delta}\right)^{N+s\,p}\frac{1}{R^{s\,p}}, \qquad \mbox{ for } x \in B_{R-\delta}.
	\]
	This entails that  
	\[
	\begin{split}
		\mathcal{A}_2 = \left(\iint_{\left(\mathbb{R}^N \setminus B_{R}\right) \times \mathbb{R}^N}\frac{|\psi(x)|^p}{|x-y|^{N+s\,p}}dxdy \right)^{\frac{1}{p}} &\le \left(\int_{B_{R-\delta}} |1-u(x)|^p\,\left( \int_{\mathbb{R}^N \setminus B_{R}} \frac{dy}{|x-y|^{N+s\,p}}\right)dx\right)^{\frac{1}{p}}\\
		&\le \left(\frac{N\,\omega_N}{s\,p}\right)^{\frac{1}{p}} \left(\frac{R}{\delta}\right)^{\frac{N}{p}+s}\,\frac{1}{R^{s}} \|1 - u\|_{L^p(B_{R})}. 
	\end{split}
	\] 
	By collecting the previous estimates, we obtain 
	\begin{equation} \label{quasi-cap}
		\begin{split}
			\left(\widetilde{\rm cap}_{s,p}\left(\Sigma; B_{R}\right)\right)^{\frac{1}{p}} \le \left(\iint_{B_{R} \times \mathbb{R}^N} \frac{|u(x) - u(y)|^p}{|x-y|^{N+s\,p}}dxdy\right)^{\frac{1}{p}} 
			+ \left(\frac{1}{s\,(1-s)}\right)^{\frac{1}{p}} \frac{D}{R^s}\,\|1-u\|_{L^p(B_{R})},
		\end{split}
	\end{equation}
	where we set 
	\begin{equation*} \label{costante-D}
		D\left(N, p, s, R, r, \delta\right) := \left(\frac{R}{\delta}\right)^s\left[c_{N,p}^{\frac{1}{p}}\left(\frac{\delta}{(R-\delta)-r}\right)^s + \left(\frac{(1-s)\,N\,\omega_N}{p}\right)^{\frac{1}{p}} \left(\frac{R}{\delta}\right)^{\frac{N}{p}}\right].
	\end{equation*}
	In turn, by recalling the definition of $\delta$, we can majorize the last constant with another one which has a simpler expression 
	\begin{equation} \label{costante-E}
		D \le \frac{2R}{R-r} \left[c_{N,p}^{\frac{1}{p}} + \left(\frac{N \omega_N}{p}\right)^{\frac{1}{p}} \left(\frac{2R}{R-r}\right)^{\frac{N}{p}}\right] =: E\left(N,p, \frac{R}{r}\right).
	\end{equation}
	To estimate the $L^p-$norm appearing in the second term of \eqref{quasi-cap}, by the triangle inequality and by recalling \eqref{hp:normalizzazione} we get
	\[
	\begin{split}
		\|1 - u\|_{L^p(B_{R})} &\le \|1-{\rm av}(u; B_{R})\|_{L^p(B_{R})} + \|{\rm av}(u; B_{R}) - u\|_{L^p(B_{R})} \\
		&\le \frac{|B_{R}|^{\frac{1}{p}}}{|B_r|^{\frac{1}{p}}}\Big|\|u\|_{L^p(B_r)} - \|{\rm av}(u; B_{R})\|_{L^p(B_r)}\Big| + \|{\rm av}(u; B_{R}) - u\|_{L^p(B_{R})} \\
		&\le \left(\left(\frac{R}{r}\right)^{\frac{N}{p}} + 1 \right)\|u - {\rm av}(u; B_{R})\|_{L^p(B_{R})}.
	\end{split} 
	\]
	By Lemma \ref{lm:poincare-wirtinger}, we then obtain
	\[
	\begin{split}
		\|1 - u\|_{L^p(B_{R})} &\le  \left(\left(\frac{R}{r}\right)^{\frac{N}{p}} + 1 \right)\,\left(\mathcal{W}_{N,p}\,s\,(1-s)\right)^{\frac{1}{p}}\,R^{s}\,\left(\iint_{B_{R} \times \mathbb{R}^N} \frac{|u(x) - u(y)|^p}{|x-y|^{N+s\,p}}dxdy\right)^{\frac{1}{p}}.
	\end{split}
	\]
	By spending this information in \eqref{quasi-cap} and using \eqref{costante-E}, we get 
	\[
	\left(\widetilde{\rm cap}_{s,p}\left(\Sigma; B_{R}\right)\right)^{\frac{1}{p}}  \le \left[1 + \mathcal{W}_{N,p}^{\frac{1}{p}}\,\left(\left(\frac{R}{r}\right)^{\frac{N}{p}} + 1 \right)E \right]\left(\iint_{B_{R} \times \mathbb{R}^N} \frac{|u(x) - u(y)|^p}{|x-y|^{N+s\,p}}dxdy\right)^{\frac{1}{p}}.
	\]
	By recalling \eqref{hp:normalizzazione}, we eventually obtain 
	\[
	\frac{M}{r^{\frac{N}{p}}} \left(\widetilde{\rm cap}_{s,p}\left(\Sigma; B_{R}\right)\right)^{\frac{1}{p}}  \|u\|_{L^p(B_r)} \le \left(\iint_{B_{R} \times \mathbb{R}^N} \frac{|u(x) - u(y)|^p}{|x-y|^{N+s\,p}}dxdy\right)^{\frac{1}{p}},
	\]
	where we set 
	\[
	M\left(N,p, \frac{R}{r}\right):= \omega_N^{-\frac{1}{p}} \left[1 + \mathcal{W}_{N,p}^{\frac{1}{p}}\,\left(\left(\frac{R}{r}\right)^{\frac{N}{p}} + 1 \right)E \right]^{-1},
	\]
	where $E$ is given by \eqref{costante-E}, as desired. 
\end{proof}
\begin{remark}[Quality of the constants] \label{rmk:asym-maz-poin}
	For future purposes (see the proof of Lemma \ref{lm:mazya-poin2} below), by using the notation of the proof above, we record the following limiting behaviours  
	\[
	0 < \lim_{R/r \to 1} \left(\frac{R}{r} - 1\right)^{\frac{N}{p} + 1}\,E\left(N,p, \frac{R}{r}\right) < \infty, \quad \lim_{R/r \to \infty}E\left(N,p, \frac{R}{r}\right)= 2\left[C_{N,p}^{\frac{1}{p}} + \left(\frac{N \omega_N}{p}\right)^{\frac{1}{p}} 2^{\frac{N}{p}}\right].
	\]
\end{remark}

\begin{remark}[Asymptotic sharpness] \label{rmk:asym-sharp1}
	 By keeping in mind Remark \ref{rmk:seminorma-asimmetrica} and Remark \ref{rmk:asym-cap}, we infer that the estimate established in Lemma \ref{lm:mazya-poin} displays sharp limiting behaviours for $s \nearrow 1$ and $s \searrow 0$. 
\end{remark}

\subsection{Poincar\'e-Sobolev inequalities}
We now derive {\it Poincar\'e-Sobolev-type inequalities}. In other words, in place of the $L^p-$norm appearing on the left-hand side of the estimates established in Lemma \ref{lm:poincare-wirtinger} and Lemma \ref{lm:mazya-poin}, there will be the $L^q-$norm. The exponent $q$ is required to be {\it superhomogeneous} and {\it subcritical}, i.e.   
\begin{equation} \label{cond:subcritical}
	p\le q \begin{cases}
		\le p^*_s, \quad &\mbox{ if } s\,p<N,\\
		<\infty, \quad &\mbox{ if } s\,p=N,\\
		\le\infty, \quad &\mbox{ if } s\,p>N. 
	\end{cases}
\end{equation}
We state the forthcoming inequalities separately from that of Lemma \ref{lm:poincare-wirtinger} and Lemma \ref{lm:mazya-poin}, since they shows a slight difference in the asymptotic degeneracy of the constants involved.   

\begin{proposition}[Poincar\'e-Sobolev] \label{prop:poin-sobolev}
	Let $1\le p<\infty$ and $0<s<1$. We take an exponent $q$ satisfying \eqref{cond:subcritical}.  
	Let $x_0 \in \mathbb{R}^N$ and $r > 0$. For every $u\in C^\infty_0(B_r(x_0))$, we have
	\[
	\|u\|^p_{L^q(B_r(x_0))}\le \frac{2}{s\,(1-s)\,\lambda_{p,q}^s(B_1)}\,r^{s\,p-N + N\,\frac{p}{q}}\,s\,(1-s)\iint_{B_r(x_0) \times \mathbb{R}^N}\frac{|u(x) - u(y)|^{p}}{|x-y|^{N+s\,p}}dxdy. 
	\]
\end{proposition}
\begin{proof}
	Without loss of generality, we assume $x_0=0$. Since $\lambda^s_{p,q}(B_r) > 0$, we have
	\[
	\begin{split}
		\|u\|^p_{L^q(B_r)}&\le \frac{r^{s\,p -N + N\,\frac{p}{q}} }{\lambda_{p,q}^s(B_1)}\,[u]^p_{W^{s,p}(\mathbb{R}^N)} \\
		&= \frac{r^{s\,p - N + N\,\frac{p}{q}} }{\lambda_{p,q}^s(B_1)} \left(\iint_{B_r \times \mathbb{R}^N}\frac{|u(x) - u(y)|^p}{|x-y|^{N+s\,p}}dxdy + \iint_{\left(\mathbb{R}^N \setminus B_r \right) \times \mathbb{R}^N}\frac{|u(x) - u(y)|^p}{|x-y|^{N+s\,p}}dxdy\right).
	\end{split}
	\]
	We focus on the second term. Since $u = 0$ on $\mathbb{R}^N \setminus B_r$, we have 
	\[
	\begin{split}
		\iint_{\left(\mathbb{R}^N \setminus B_r \right) \times \mathbb{R}^N}\frac{|u(x) - u(y)|^p}{|x-y|^{N+s\,p}}dxdy &= \iint_{\left(\mathbb{R}^N \setminus B_r\right) \times B_r } \frac{|u(y)|^p}{|x-y|^{N+s\,p}}dxdy \\
		&= \iint_{\left(\mathbb{R}^N \setminus B_r\right) \times B_r } \frac{|u(x) - u(y)|^p}{|x-y|^{N+s\,p}}dxdy \\
		&\le  \iint_{\mathbb{R}^N \times B_r } \frac{|u(x) - u(y)|^p}{|x-y|^{N+s\,p}}dxdy. 
	\end{split}
	\]
	The conclusion then follows by Fubini's theorem.
\end{proof}

\begin{remark}[Asymptotic sharpness] \label{rmk:asym-sharp2}
	In light of Remark \ref{rmk:seminorma-asimmetrica}, Proposition \ref{prop:asym-s-1} and Proposition \ref{prop:asym2}, we can infer that the estimate provided by Proposition \ref{prop:poin-sobolev} displays sharp limiting behaviours with respect to $s$. 
\end{remark}

As a consequence of the previous proposition and Lemma \ref{lm:poincare-wirtinger}, we get the following. 

\begin{lemma}[Poincar\'e-Sobolev-Wirtinger] \label{lm:poin-sob-wirtinger}
	Let $1\le p<\infty$ and $0<s<1$. We take an exponent $q$ satisfying \eqref{cond:subcritical}. 
	Let $x_0 \in \mathbb{R}^N$, $r > 0$ and $R > \sqrt{N}r$. For every $u \in C^\infty_0(\mathbb{R}^N)$ we have 
	\begin{equation*}\label{ineq:poin-wirtinger}
		\|u - \mathrm{av}(u; B_r(x_0))\|^{p}_{L^q(B_r(x_0))} \le W\,\frac{R^{s\,p-N + N\,\frac{p}{q} }}{s\,(1-s)\,\lambda_{p,q}^s(B_1)}\,s\,(1-s)\iint_{B_R(x_0) \times \mathbb{R}^N} \frac{|u(x) - u(y)|^p}{|x-y|^{N+s\,p}}dxdy,
	\end{equation*}
	for a constant $W = W\left(N, p, s, q, R/r\right) > 0$. More precisely, we have 
	\[
	0 < \lim_{s \to 0} W < \infty, \qquad 0< \lim_{s \to 1} W < \infty,
	\]
	and 
	\[
	0 < \lim_{R/r \to \sqrt{N}} \left(\frac{R}{r}-\sqrt{N}\right)^{s\,p}\,W\left(N, p, s, q, \frac{R}{r}\right) < \infty, \qquad	0 < \lim_{R/r \to \infty} W\left(N, p, s, q, \frac{R}{r}\right) < \infty.
	\]
\end{lemma}
\begin{proof}
	Without loss of generality, we assume $x_0=0$.
	We take $0 < \delta < R - \sqrt{N}r.$ Let $\eta \in {\rm Lip}_0(B_{R})$ be such that 
	\begin{equation*} \label{wirtinger:cut-off-prop}
		0 \leq \eta \leq 1, \quad \eta \equiv 1 \mbox{ on } \overline{B_{\sqrt{N}r}}, \quad \eta \equiv 0 \mbox{ on } B_{R} \setminus B_{R-\delta}, \quad |\nabla \eta| \leq \frac{1}{\left(R- \delta\right) - \sqrt{N}r}. 
	\end{equation*}
	We set 
	\[
	\psi:= \eta\,\left(u - m \right)\in {\rm Lip}_0(B_{R}), \quad \mbox{ with }\,\, m := {\rm av}(u; B_{R}).
	\]
	We have 
	\[
	\begin{split}
		\|u - {\rm av}(u; B_r)\|_{L^q(B_r)}^p &\le 	\|u - {\rm av}(u; B_r)\|_{L^q(Q_r)}^p \\
		&\le \|u - {\rm av}(u; Q_r)\|^p_{L^q(Q_r)} + \|{\rm av}(u; B_r) - {\rm av}(u; Q_r)\|^p_{L^q(Q_r)} := \mathcal{A}_1 + \mathcal{A}_2.
	\end{split}
	\]
	The second term can be estimated in terms of the first one. Indeed, a routine majorization yields
	\[
	|{\rm av}(u; B_r) - {\rm av}(u; Q_r)|^q \le \Bigg|\fint_{B_r} \left(u - {\rm av}(u; Q_r)\right)dx\Bigg|^q \le \fint_{B_r}|u - {\rm av}(u; Q_r)|^q\,dx.
	\]
	By integrating over $Q_r$ and raising to the power $p/q$, we get 
	\[
	\mathcal{A}_2 \le \left(\frac{|Q_r|}{|B_r|}\right)^{\frac{p}{q}} \|u - {\rm av}(u; Q_r)\|^p_{L^q(Q_r)}. 
	\]
	Thus
	\begin{equation} \label{fesso2}
		\begin{split}
			\|u - {\rm av}(u; B_r)\|_{L^q(B_r)}^p &\le \left(1 + \left(\frac{|Q_r|}{|B_r|}\right)^{\frac{p}{q}} \right) \|u - {\rm av}(u; Q_r)\|^p_{L^q(Q_r)}\\ &=  \left(1 + \left(\frac{|Q_r|}{|B_r|}\right)^{\frac{p}{q}} \right) \mathcal{A}_1. 
		\end{split}
	\end{equation}
	By \cite[Remark 2.1]{Gi}\footnote{More precisely, take $\xi = m$ in the last formula of \cite[page 48]{Gi}.}, we have 
	\[
	\begin{split}
		\mathcal{A}_1 = \|u - {\rm av}(u; Q_r)\|_{L^q(Q_r)}^p &\le 2^{(q+1)\,\frac{p}{q}}\|u - {\rm av}(u; B_{\sqrt{N}r})\|^{p}_{L^q(B_{\sqrt{N}\,r})}\\
		&\le 2^{(q+1)\frac{p}{q}}\|\psi\|_{L^{q}(B_{R})}^{p} \\
		&\le \frac{2^{(q+1)\,\frac{p}{q}+1}}{\lambda_{p,q}^s(B_1)} R^{s\,p - N + N\,\frac{p}{q}} \iint_{B_{R} \times \mathbb{R}^N} \frac{|\psi(x) - \psi(y)|^p}{|x - y|^{N+s\,p}}dxdy
	\end{split}
	\]
	where in the last line we used  Proposition \ref{prop:poin-sobolev}. We need to estimate the last integral. By adding and subtracting $\eta(x)\,u(y)$ and by using Minkowski's inequality, we get 
	\[
	\begin{split}
		&\left(\iint_{B_{R} \times \mathbb{R}^N} \frac{|\psi(x) - \psi(y)|^p}{|x-y|^{N+s\,p}} \,dxdy\right)^{\frac{1}{p}} \\
		&\le \left(\iint_{B_{R} \times \mathbb{R}^N}\frac{|u(x) - u(y)|^p}{|x-y|^{N+s\,p}}dxdy \right)^{\frac{1}{p}} + \left(\int_{B_{R}}|u(y) - m|^p\left(\int_{\mathbb{R}^N} \frac{|\eta(x) - \eta(y)|^{p}}{|x-y|^{N+s\,p}}dx\right) dy\right)^{\frac{1}{p}}\\
		&\le  \left(\iint_{B_{R} \times \mathbb{R}^N}\frac{|u(x) - u(y)|^p}{|x-y|^{N+s\,p}}dxdy \right)^{\frac{1}{p}} + \left(\frac{c_{N,p}}{s\,(1-s)}\right)^{\frac{1}{p}} \left(\frac{1}{(R - \delta)-\sqrt{N}\,r}\right)^s \|u - m\|_{L^p(B_{R})},
	\end{split}
	\]
	where in the last line we also used Lemma \ref{leibniz-holder}. To estimate the second term in the last inequality we use Lemma \ref{lm:poincare-wirtinger}, this leads to 
	\[
	\begin{split}
		&\left(\iint_{B_{R} \times \mathbb{R}^N} \frac{|\psi(x) - \psi(y)|^p}{|x-y|^{N+s\,p}} \,dxdy\right)^{\frac{1}{p}} \\
		&\le \left[1+\left(C_{N,p}\,\mathcal{W}_{N,p}\right)^{\frac{1}{p}} \left(\frac{R}{(R- \delta)-\sqrt{N}\,r}\right)^s\right] \left(\,\iint_{B_{R} \times \mathbb{R}^N} \frac{|u(x) - u(y)|^p}{|x-y|^{N+s\,p}}dxdy\right)^{\frac{1}{p}}
	\end{split}
	\] 
	By collecting the previous estimates, we get 
	\[
	\mathcal{A}_1 \le \frac{C}{\lambda_{p,q}^s(B_1)}\,R^{\frac{p}{q}\,N + s\,p - N}\,\iint_{B_{R} \times \mathbb{R}^N} \frac{|u(x) - u(y)|^p}{|x-y|^{N+s\,p}}dxdy
	\] 
	where we set 
	\[
	C = C\left(N,p,s, q, R, r, \delta\right) :=  2^{(q+1)\,\frac{p}{q} + 1} \left[1+\left(c_{N,p}\,\mathcal{W}_{N,p}\right)^{\frac{1}{p}} \left(\frac{R}{(R - \delta) - \sqrt{N}r}\right)^s\right]^p.
	\]
	We fix now 
	$
	\delta = (R - \sqrt{N}r)/2.
	$ 
	By recalling \eqref{fesso2}, we eventually get 
	\[
	\|u - {\rm av}(u; B_r)\|^p_{L^q(B_r)} \le \frac{W}{\lambda_{p,q}^s(B_1)}\,R^{s\,p - N + N\,\frac{p}{q}}\,\iint_{B_{R} \times \mathbb{R}^N} \frac{|u(x) - u(y)|^p}{|x-y|^{N+s\,p}}dxdy,
	\]
	where we set 
	\begin{equation} \label{esplicita0}
		W = W\left(N, p, s, q, \frac{R}{r}\right):=  \left(1+\omega_N^{-\frac{p}{q}}\right)  2^{(q+1)\frac{p}{q} + 1} \left[1+\left(C_{N,p}\mathcal{W}_{N,p}\right)^{\frac{1}{p}} \left(\frac{2R}{R - \sqrt{N}r}\right)^s\right]^p,
	\end{equation}
	as desired.
\end{proof}

\begin{remark}[Asymptotic sharpness]
	By recalling \eqref{esplicita0} and by arguing as in Remark \ref{rmk:asym-sharp2}, we infer that the estimate provided by Lemma \ref{lm:poin-sob-wirtinger} displays sharp limiting behaviours in $s$. 
\end{remark}

We are now able to prove the following important inequality, which is the main result of this section and one of the cornerstone in the proof of Theorem \ref{thm:lower-bound}. For a comment on the relevant constant, we refer to Remark \ref{asym-costante-mps}.  

\begin{lemma}[Maz'ya-Poincar\'e-Sobolev] \label{lm:mazya-poin2}
	Let $1 \le p < \infty$ and $0 < s < 1$. We take an exponent $q$ satisfying \eqref{cond:subcritical}. Let $x_0 \in \mathbb{R}^N$, $r > 0$ and $R > \sqrt{N}r$. Let $\Sigma \subseteq  \overline{B_{r}(x_0)}$ be a compact set. For every $u \in C^\infty_0(\mathbb{R}^N)$ such that $u \equiv 0$ on $\Sigma$, we have 
	\begin{equation*}
		\frac{M}{r^{\frac{N}{q}}}\,\left(\widetilde{\rm cap}_{s,p}\left(\Sigma; B_{R}(x_0)\right)\right)^{\frac{1}{p}} \,\|u\|_{L^q(B_r(x_0))} \le \left(\iint_{B_{R}(x_0) \times \mathbb{R}^N} \frac{|u(x) - u(y)|^p}{|x-y|^{N+s\,p}}dxdy\right)^{\frac{1}{p}},
	\end{equation*}
	where $M=M\left(N, p, s, q, R/r\right) > 0$. Moreover
	\[
	M\left(N, p, s, q, \frac{R}{r}\right) \sim 1, \quad \mbox{ as } s \searrow 0 \mbox{ and } s \nearrow 1,
	\]
	and 
	\[
	0 < \lim_{R/r \to \sqrt{N}} \frac{M\left(N, p, s, q, R/r\right)}{\left(R/r - \sqrt{N}\right)^s} < \infty, \qquad 0 < \lim_{R/r \to \infty}\left(\frac{R}{r}\right)^{\frac{N}{q}}\,M\left(N, p, s, q, \frac{R}{r}\right) < \infty.
	\] 
\end{lemma}
\begin{proof}
	Without loss of generality, we can assume $x_0=0$ and 
	\begin{equation} \label{hp:normalizzazione2}
		\fint_{B_r} |u|^q dx = 1.
	\end{equation} 
	We take $R > \sqrt{N}r$. By arguing as in the beginning of the proof of Lemma \ref{lm:mazya-poin}, we are led to 
	\begin{equation} \label{quasi-cap2}
		\begin{split}
			\left(\widetilde{\rm cap}_{s,p}\left(\Sigma; B_{R}\right)\right)^{\frac{1}{p}} \le& \left(\iint_{B_{R} \times \mathbb{R}^N} \frac{|u(x) - u(y)|^p}{|x-y|^{N+s\,p}}dxdy\right)^{\frac{1}{p}} 
			+ \left(\frac{1}{s\,(1-s)}\right)^{\frac{1}{p}} \frac{E}{R^s} \|1-u\|_{L^p(B_{R})},
		\end{split}
	\end{equation}
	where $E$ is given by \eqref{costante-E}. In order to estimate the $L^p-$norm appearing in the second term, we use this time H\"older's inequality and the triangle inequality obtaining 
	\[
	\begin{split}
		\|1- u\|_{L^p(B_R)} &\le |B_R|^{\frac{1}{p} - \frac{1}{q}} \|1-u\|_{L^q(B_R)} \\
		&\le |B_R|^{\frac{1}{p} - \frac{1}{q}} \left(\|1-{\rm av}(u; B_R)\|_{L^q(B_R)} + \|{\rm av}(u; B_R) - u\|_{L^q(B_R)}\right).
	\end{split} 
	\]
	The first term in the rightmost term can be estimated in terms of the second one. Indeed, by recalling \eqref{hp:normalizzazione2} we have 
	\[
	\begin{split}
		\|1-{\rm av}(u; B_R)\|_{L^q(B_R)} &= |B_R|^{\frac{1}{q}} |1-{\rm av}(u; B_R)| = \left(\frac{|B_R|}{|B_r|}\right)^{\frac{1}{q}} \Big|\|u\|_{L^q(B_r)}  - \|{\rm av}(u; B_R)\|_{L^q(B_r)}\Big| \\
		&\le \left(\frac{|B_R|}{|B_r|}\right)^{\frac{1}{q}} \|u - {\rm av}(u; B_R)\|_{L^q(B_R)}.
	\end{split}
	\]
	This leads to 
	\[
	\begin{split}
		\|1- u\|_{L^p(B_R)} &\le \omega_N^{\frac{1}{p} - \frac{1}{q}}R^{\frac{N}{p} - \frac{N}{q}} \left(1+ \left(\frac{R}{r}\right)^{\frac{N}{q}} \right) \|u-{\rm av}(u; B_R)\|_{L^q(B_R)}\\
		&\le \omega_N^{\frac{1}{p} - \frac{1}{q}} \left(1+ \left(\frac{R}{r}\right)^{\frac{N}{q}} \right)\,\left(\frac{W}{\lambda_{p,q}^s(B_1)}\right)^{\frac{1}{p}}\,R^{ s}\,\left(\iint_{B_R \times \mathbb{R}^N} \frac{|u(x) - u(y)|^p}{|x-y|^{N+s\,p}}dxdy\right)^{\frac{1}{p}},
	\end{split}
	\]
	where we used Lemma \ref{lm:poin-sob-wirtinger}. By spending this information in \eqref{quasi-cap2}, we get 
	\[
	\begin{split}
		\left(\widetilde{\rm cap}_{s,p}\left(\Sigma; B_{R}\right)\right)^{\frac{1}{p}} &\le \left[1+ \omega_N^{\frac{1}{p} - \frac{1}{q}}  \left(1+ \left(\frac{R}{r}\right)^{\frac{N}{q}} \right)\, \left(\frac{W}{s\,(1-s)\,\lambda_{p,q}^s(B_1)}\right)^{\frac{1}{p}} E\right] \times \\
		&\quad \times \left(\iint_{B_R \times \mathbb{R}^N} \frac{|u(x) - u(y)|^p}{|x-y|^{N+s\,p}}dxdy\right)^{\frac{1}{p}}.
	\end{split}
	\]
	We set
	\begin{equation} \label{esplicita}
		M = M\left(N, p, s, q, \frac{R}{r}\right) := \omega_N^{-\frac{1}{q}}\left[1+ \omega_N^{\frac{1}{p} - \frac{1}{q}}  \left(1+ \left(\frac{R}{r}\right)^{\frac{N}{q}} \right)\, \left(\frac{W}{s\,(1-s)\,\lambda_{p,q}^s(B_1)}\right)^{\frac{1}{p}} E\right]^{-1},
	\end{equation}
	and, by recalling \eqref{hp:normalizzazione2}, we get the desired estimate. By Propositions \ref{prop:asym-s-1}-\ref{prop:asym2}, we infer the claimed asymptotic behaviours of $M$ with respect to $s$. Eventually, by Lemma \ref{lm:poin-sob-wirtinger} and Remark \ref{rmk:asym-maz-poin}, we also infer the limiting behaviours of $M$ as $R/r \to \sqrt{N}$ and $R/r \to \infty$. 
\end{proof}

\begin{remark}[Asymptotic sharpness] \label{asym-costante-mps}
 By Remark \ref{rmk:seminorma-asimmetrica} and Remark \ref{rmk:asym-cap}, we infer that the estimate provided by Lemma \ref{lm:mazya-poin2} displays sharp limiting behaviours as $s \searrow 0$ and $s \nearrow 1$. 
\end{remark}

\section{Proof of the Main Theorems} \label{sec:5}
Armed with the results of the previous sections, we can now prove the Main Theorems of this paper: Theorem \ref{thm:lower-bound} and Theorem \ref{thm:upper-bound}. 
\subsection{Proof of the lower bound}
\begin{proof}[Proof of Theorem \ref{thm:lower-bound}]
	We can assume $R^s_{p, \gamma}(\Omega)< +\infty$, otherwise there is nothing to prove. Let $r > R^s_{p, \gamma}(\Omega)$, for every $x_0 \in \mathbb{R}^N$ we have \begin{equation} \label{ipotesi-gamma-fat}
		\widetilde{\mathrm{cap}}_{s,p}\left(\overline{B_r(x_0)} \setminus \Omega; B_{2\,r}(x_0)\right)  > \gamma\,\widetilde{\mathrm{cap}}_{s,p}\left(\overline{B_1}; B_{2}\right)\,r^{N-s\,p}.
	\end{equation}
	We introduce now two covering families for the whole $\mathbb{R}^N$ made of balls with equidistant centers. 
	Define	\[
	D_{N}:= \left\{{\bf i} = (i_1,...,i_N)\in \mathbb{Z}^N :\, i_k \mbox{ is odd for } k = 1, \ldots, N\right\},  
	\]
	and
	\[
	Z_r := \left\{{\bf p} = (p_1, \ldots, p_N) \in \mathbb{R}^N :\, {\bf p} = \frac{r}{\sqrt{N}}\,{\bf i} \mbox{ with }  {\bf i} \in D_N\right\}.
	\]
	We set \[
	B_{\bf p} := B_{r} + {\bf p}, \qquad  \widetilde{B_{\bf p}} := B_{2\,\sqrt{N}\,r} + {\bf p}, \qquad  \mbox{ for } {\bf p} \in Z_r.
	\]
	Both the families of balls $\{B_{\bf p}\}_{{\bf p} \in Z_r}$ and $\{\widetilde{B_{\bf p}}\}_{{\bf p} \in Z_r}$ cover the whole $\mathbb{R}^N$. We claim that the multiplicity of $\{\widetilde{B_{\bf p}}\}_{{\bf p} \in Z_r}$ is at most $(4N+1)^N$. To prove this, for ${\bf p} \in Z_r$ we set 
	\[
	\quad Q_{\bf p} := Q_{\frac{r}{\sqrt{N}}} + {\bf p}, \quad \widetilde{Q_{\bf p}} := Q_{\varrho} + {\bf p}, \quad \mbox{ where } \varrho := \frac{4N+1}{\sqrt{N}}\,r.
	\]
	Observe that $\widetilde{B_{{\bf p}}}$ intersects at most $(4N+1)^N$ balls of $\{\widetilde{ B_{\bf q}}\}_{{\bf q} \in Z_r}.$ Indeed, if
	\[
	x = (x_1, \ldots, x_N) \in \widetilde{B_{\bf p}} \cap \widetilde{B_{\bf q}}, \quad \mbox{ with } {\bf p} = \frac{r}{\sqrt{N}}{\bf i}\,\, \mbox{ and }\,\, {\bf q} = \frac{r}{\sqrt{N}}{\bf j},
	\]
	for some different ${\bf i}, {\bf j} \in D_N$, we have 
	\[
	|q_k - p_k| \le \left|{\bf p} - {\bf q}\right| \le \left|{\bf p} - x \right| + \left| x -  {\bf q}\right| < 4\sqrt{N}r, \qquad \mbox{ for } k=1, \ldots N. 
	\]
	This entails that 
	$
	Q_{\bf q} \subseteq \widetilde{Q_{\bf p}},
	$
	since for every $y = (y_1, \ldots, y_N) \in Q_{\bf q}$ we have 
	\[
	|y_k - p_k| \le |y_k - q_k| + |q_k - p_k| < \frac{r}{\sqrt{N}} + 4\sqrt{N}r = \varrho, \quad \mbox{ for } k=1, \ldots, N. 
	\]
	Our claim then follows by observing that \[
	{\rm card}\Big(Z_r \cap \widetilde{Q_{\bf p}}\Big) \le (4N+1)^N.
	\] 
	We can now conclude the proof. Let $u \in C^\infty_0(\Omega)$ be extended by zero on the whole $\mathbb{R}^N$. For every ${\bf p} \in Z_r$, we have $u \equiv 0$ on $\overline{B_{\bf p}} \setminus \Omega$. By using Lemma \ref{lm:mazya-poin2} with $R = 2\,\sqrt{N}\,r$, we get
	\begin{equation} \label{quasi-ci-siamo}
		\begin{split}
			(4N+1)^N \iint_{\mathbb{R}^N \times \mathbb{R}^N} \frac{|u(x) - u(y)|^p}{|x-y|^{N+s\,p}}\,dxdy &\ge \sum_{{\bf p}\in Z_r} \iint_{\widetilde{B_{\bf p}} \times \mathbb{R}^N} \frac{|u(x) - u(y)|^p}{|x-y|^{N+s\,p}}dxdy \\
			&\ge \frac{M^p}{r^{N\,\frac{p}{q}}} \sum_{{\bf p}\in Z_r} \widetilde{\rm cap}_{s,p}\left(\overline{B_{\bf p}} \setminus \Omega; \widetilde{B_{\bf p}}\right) \|u\|^p_{L^q(B_{{\bf p}})},
		\end{split}
	\end{equation}
	where $M = M\left(N,p, s, q\right)$ is given by Lemma \ref{lm:mazya-poin2}. By using Proposition \ref{prop:cap-wrt-balls}, for every ${\bf p} \in Z_r$, there exists a constant $\mathscr{C} = \mathscr{C}\left(N,p, s\right) > 0$, independent from ${\bf p}$, such that 
	\[
	\widetilde{\rm cap}_{s,p}\left(\overline{B_{\bf p}} \setminus \Omega; \widetilde{B_{\bf p}}\right) \ge \frac{1}{\mathscr{C}}\,\widetilde{\rm cap}_{s,p}\left(\overline{B_{\bf p}} \setminus \Omega; B^*_{\bf p}\right), \quad \mbox{ where } B^*_{\bf p}:= B_{2\,r} + {\bf p}.
	\]
	By spending this information in \eqref{quasi-ci-siamo}, we obtain that 
	\[
	\begin{split}
		(4\,N+1)^N\,\iint_{\mathbb{R}^N \times \mathbb{R}^N} \frac{|u(x) - u(y)|^p}{|x-y|^{N+s\,p}}\,dxdy &\ge \frac{1}{r^{N\,\frac{p}{q}}}\,\frac{M^p}{\mathscr{C}}\, \sum_{{\bf p}\in Z_r} \widetilde{\rm cap}_{s,p}\left(\overline{B_{\bf p}} \setminus \Omega; B^*_{\bf p}\right) \|u\|^p_{L^q(B_{{\bf p}})} \\
		&\ge \frac{\gamma}{r^{s\,p - N + N\,\frac{p}{q}}}\,\frac{M^p}{\mathscr{C}}\,\widetilde{{\rm cap}}_{s,p}\left(\overline{B_1}; B_2\right)\,\sum_{{\bf p}\in Z_r} \|u\|^p_{L^q(B_{{\bf p}})},
	\end{split}
	\]
	where in the last line we used \eqref{ipotesi-gamma-fat}.
	Since $q \ge p$, the function $\tau \mapsto \tau^{p/q}$ is subadditive, so
	\[
	\sum_{{\bf p}\in Z_r} \|u\|^p_{L^q(B_{{\bf p}})} \ge \left(\sum_{{\bf p}\in Z_r} \|u\|^q_{L^q(B_{\bf p})}\right)^{\frac{p}{q}} \ge \|u\|^{p}_{L^q(\Omega)}.
	\]
	By the last two inequalities and by the arbitrariness of $u$ and $r$, we get the desired result with
	\[
	\sigma_{N, p, s, q} = \dfrac{M^{{p}}}{\mathscr{C}}\,\frac{\widetilde{\mathrm{cap}}_{s,p}\left(\overline{B_1};B_2\right)}{(4N+1)^N},
	\]
	where $M$ and $\mathscr{C}$ are as before. Eventually, by recalling \eqref{costante-cap-wrt-balls} and by Remark \ref{rmk:asym-cap}, we infer the claimed limiting behaviours of $\sigma_{N, p, s, q}$, as $s \searrow 0$ and $s \nearrow 1$.
\end{proof}

\subsection{Proof of the upper bound} \label{sec:proof-ub}

\begin{proof}[Proof of Theorem \ref{thm:upper-bound}]
	Let $0 < \gamma_0 \le 1$ to be chosen later, we take $0<\gamma<\gamma_0$. Let $B_r(x_0)$ be a ball such that 
	\begin{equation}\label{gamma_cap p>1}
		\widetilde{\mathrm{cap}}_{s,p}\left(\overline{B_r\left(x_0\right)} \backslash \Omega ; B_{2 r}\left(x_0\right)\right) \leq \gamma\,\widetilde{\mathrm{cap}}_{s,p}\left(\overline{B_r\left(x_0\right)} ; B_{2 r}\left(x_0\right)\right).
	\end{equation}
	We look for a constant $\mathcal{C} = \mathcal{C}(N,p, s,\gamma) > 0$ such that
	\begin{equation} \label{claim_upper_bound_1 p>1}
		\lambda^s_{p, q}(\Omega) \leq \frac{\mathcal{C}}{r^{s\,p - N + N\,\frac{p}{q}}}. 
	\end{equation}
	The claimed result will eventually follow by taking the supremum over all the admissible $r$. In particular, if $R^s_{p, \gamma}(\Omega) = +\infty$ the last inequality entails that $\lambda^s_{p, q}(\Omega) = 0$. 
	\par
	Without loss of generality, we assume $x_0=0$. For simplicity, we set $F=\overline{B_r} \setminus \Omega$. For every $\delta>0$, we take a function $\varphi_\delta\in \operatorname{Lip}_0\left(B_{2 r}\right)$ such that 
	\begin{equation} \label{cap-approximation p>1}
		0\leq \varphi_\delta \leq 1, \qquad \varphi_\delta=1 \text { on } F, \qquad [\varphi_\delta]^p_{W^{s,p}(\mathbb{R}^N)}\leq \delta\,\widetilde{\operatorname{cap}}_{s,p}\left(\overline{B_r};B_{2r}\right)+\widetilde{\operatorname{cap}}_{s, p}\left(F ; B_{2 r}\right).
	\end{equation} 
	Such a function exists, in light of Proposition \ref{prop:equivalent-cap}. Fixed $0 < \varepsilon < 1/2$, we take the cut-off function $\eta \in \mathrm{Lip}_0(B_r)$ given by 
	\[
	\eta(x) = \min\left\{\left(\dfrac{(1-\varepsilon)\,r - |x|}{(1 - \varepsilon)\,r - (1- 2\,\varepsilon)\,r}\right)_{+},\,\, 1\right\}, \qquad \mbox{ for } x \in \mathbb{R}^N.
	\]
	In particular
	\begin{equation} \label{cut-off-upper-bound}
		0 \leq \eta \leq 1, \quad \eta \equiv 1 \mbox{ on } B_{(1-2\varepsilon)\,r}, \quad  \eta \equiv 0 \mbox{ on } \mathbb{R}^N \setminus B_{(1-\varepsilon)\,r} \quad \mbox{ and } \quad \|\nabla \eta\|_{L^\infty(\mathbb{R}^N)} = \frac{1}{\varepsilon\,r}.
	\end{equation}
	We use $\psi := (1-\varphi_\delta)\,\eta/ \|(1-\varphi_\delta)\,\eta\|_{L^p(\Omega)}$ as a test function in the definition of $\lambda^s_{p, q}(\Omega)$. This is an admissible function. Indeed, by construction $\psi \in \mathrm{Lip}_0(\mathbb{R}^N)$ and \[
	1 - \varphi_\delta \equiv 0 \quad \mbox{ on } F = \overline{B_r} \setminus \Omega, \qquad \mbox{ and } \qquad \eta \equiv 0 \quad \mbox{ on } \mathbb{R}^N \setminus B_r.
	\]
	This entails that
	$
	\psi \equiv 0 \mbox{ on } \mathbb{R}^N \setminus \Omega
	$
	and so $\psi \in \widetilde{W}^{s,p}_0(\Omega) \cap L^q(\Omega)$, by Lemma \ref{lm:brezis-type}.
	By Minkowski's inequality
	\begin{equation} 
		\label{test+leibniz p>1}
		\begin{split}
			\lambda^s_{p, q}(\Omega)  \left(\int_{B_{(1-2\varepsilon)\,r}}(1-\varphi_\delta)^q\,dx \right)^{\frac{p}{q}} &\le 2^{p-1} \left([\eta]^p_{W^{s,p}(\mathbb{R}^N)} + [\varphi_{\delta}]^p_{W^{s,p}(\mathbb{R}^N)}  \right)\\
			&\le  2^{p-1}\left([\eta]^p_{W^{s,p}(\mathbb{R}^N)} + (\delta+\gamma)\,\widetilde{\operatorname{cap}}_{s,p}\left(\overline{B_r};B_{2r}\right)\right),
		\end{split}
	\end{equation}
	where we used \eqref{gamma_cap p>1} and \eqref{cap-approximation p>1}. 
	By \cite[Corollary 2.2]{BPS}, \cite[Proposition 4.2]{BLP} and by \eqref{cut-off-upper-bound}, we have
	\begin{equation*} \label{norm-split p>1}
		\begin{split} 
			[\eta]_{W^{s, p}(\mathbb{R}^N)}^p &\leq \frac{C_{N,p}}{s\,(1-s)}\,\|\eta\|_{L^p(\mathbb{R}^N)}^{(1-s)\,p}\,\|\nabla \eta\|_{L^p(\mathbb{R}^N)}^{s\,p}\\
			&\leq \frac{C_{N,p}}{s\,(1-s)}\,\left(\frac{1}{\varepsilon\,r}\right)^{s\,p}\,|B_r|^{(1-s)}\,|B_{r(1-\varepsilon)}\setminus B_{r(1-2\varepsilon)}|^{s},
		\end{split}
	\end{equation*}
	for some $C_{N,p} > 0$. By Bernoulli's inequality, we further have 
	\begin{equation} \label{bernoulli}
		(1-\varepsilon)^N - (1-2\,\varepsilon)^N \le 1 - \left(1-2\,\varepsilon\right)^N \le 2\,\varepsilon\,N, \qquad \mbox{ for } 0 < \varepsilon < \frac{1}{2},
	\end{equation}
	thus by spending this information, we get
	\[
	[\eta]^p_{W^{s,p}(\mathbb{R}^N)} \leq  \omega_N\, \frac{(2N)^s}{\varepsilon^{s\,(p-1)}}\,\frac{C_{N, p}}{s\,(1-s)}\,r^{N-s\,p}.
	\]
	Then, we can majorize \eqref{test+leibniz p>1} obtaining
	\begin{equation} \label{ineq:stima-dallalto p>1} 
		\lambda^s_{p, q}(\Omega) \left(\fint_{B_{(1-2 \varepsilon)r}}\left(1-\varphi_\delta\right)^q\,dx \right)^{\frac{p}{q}} \leq \frac{\mathscr{A}}{r^{s\,p - N + N\,\frac{p}{q}}},
	\end{equation}
	where we set 
	\begin{equation} \label{costante-ausiliaria}
		\begin{split}
			\mathscr{A} &= \mathscr{A}(N,p,s, q, \gamma, \varepsilon, \delta) \\
			&= \frac{2^{p-1}\,\widetilde{\operatorname{cap}}_{s,p}\left(\overline{B_1};B_{2}\right)}{|B_{1-2\,\varepsilon}|^{\frac{p}{q}}}\left(\frac{C_{N, p}}{s\,(1-s)\,\widetilde{\operatorname{cap}}_{s,p}\left(\overline{B_1};B_{2}\right)} \frac{\omega_N\,(2N)^s}{\varepsilon^{s\,(p-1)}} +(\delta + \gamma)\right). 
		\end{split}
	\end{equation}
	By Jensen's inequality, we can bound from below the leftmost term in the last inequality obtaining 
	\begin{equation} \label{quasi-quasi}
		\begin{split}
			\lambda^s_{p, q}(\Omega) \left(1-\fint_{B_{(1-2 \varepsilon)r}}\varphi_\delta\, d x\right)^p \leq \frac{\mathscr{A}}{r^{s\,p - N + N\,\frac{p}{q}}}. 
		\end{split}
	\end{equation}
	In order to conclude, we need to prove that there exists $0 < \varepsilon_0 < 1/2$ depending on $\gamma, N, p$ and $s$ such that  
	\begin{equation}\label{1-media}
		\left(1-\fint_{B_{(1-2 \varepsilon)r}}\varphi_\delta\, d x \right)^p \geq\frac{1}{C},
	\end{equation}
	for every $0 < \varepsilon \le \varepsilon_0$. In the sequel, we discuss the cases $p > 1$ and $p=1$ separately.
	\vskip.2cm \noindent 
	{\it Case $\boxed{p > 1}$.} For $s\,p < N$, we use \eqref{ineq:cap-potential}. In combination with \eqref{gamma_cap p>1} and \eqref{cap-approximation p>1}, this yields 
	\begin{align*}\label{ineq-media_below} 
		\nonumber\left(1 - \fint_{B_{(1-2\varepsilon)r}} \varphi_\delta\,dx\right)^p &\ge \left( 1- |B_{(1-2\varepsilon)\,r}|^{\frac{s}{N}  - \frac{1}{p}} \,\mathcal{S}_{N,p,s}^{\frac{1}{p}}\,[\varphi_\delta]_{W^{s,p}(\mathbb{R}^N)}\right)^p \\
		&\ge \left(1- |B_{(1-2\varepsilon)\,r}|^{\frac{s}{N}  - \frac{1}{p}}\,\left(\mathcal{S}_{N,p,s}\,\widetilde{{\rm cap}}_{s,p}\left(\overline{B_r}; B_{2r}\right)\,(\delta + \gamma)\right)^{\frac{1}{p}} \right)^p.
	\end{align*}
	By spending the last inequality in \eqref{quasi-quasi} and by passing to the limit as $\delta \searrow 0$, we get
	\begin{equation}\label{ineq-final_upperbound}
		\begin{split}
			\frac{\mathscr{A}(N,p,s, q, \gamma, \varepsilon, 0)}{r^{s\,p - N + N\,\frac{p}{q}}} \ge \lambda_{p, q}^s(\Omega)\left(1- \frac{1}{(1-2\,\varepsilon)^{\frac{N}{p} -s }}\left(\omega_N^{\frac{p}{N}-1}\,\mathcal{S}_{N,p,s}\,\widetilde{{\rm cap}}_{s,p}\left(\overline{B_1}; B_{2}\right)\gamma\right)^{\frac{1}{p}}\right)^p,
		\end{split}
	\end{equation}
	with $\mathscr{A}$ given by \eqref{costante-ausiliaria}. 
	As announced in the beginning of the proof, we now spend the choice $\gamma_0$ by setting 
	\begin{equation} \label{scelta_gamma0}
		\gamma_0 = \gamma_0\left(N,p,s\right) :=  \min\Big\{\left(\omega_N^{\frac{p}{N}-1}\,\mathcal{S}_{N,p,s}\,\widetilde{{\rm cap}}_{s,p}\left(\overline{B_1}; B_{2}\right)\right)^{-1}, 1 \Big\}.
	\end{equation}
	and we also set 
	\begin{equation} \label{epsilon}
		\varepsilon_0 = \varepsilon_0(N, p, s, \gamma) := \frac{1}{4}\,\left(1 - \left(\frac{\gamma}{\gamma_0}\right)^{\frac{1}{N-s\,p}}\right),
	\end{equation}
	for every $0 < \gamma < \gamma_0$.
	Thus for every $ 0 < \gamma < \gamma_0$ and $0 < \varepsilon \le \varepsilon_0$, by considering the second factor in the rightmost term of \eqref{ineq-final_upperbound}, we have   
	\[
	\begin{split}
		\left(1- \frac{1}{(1-2\,\varepsilon)^{\frac{N}{p} -s }}\,\left(\omega_N^{\frac{p}{N}-s}\mathcal{S}_{N,p,s}\,\widetilde{{\rm cap}}_{s,p}\left(\overline{B_1}; B_{2}\right)\,\gamma\right)^{\frac{1}{p}}\right)^p &\ge \left(1 - \frac{1}{(1-2\,\varepsilon)^{\frac{N}{p} - s}}\,\left(\frac{\gamma}{\gamma_0}\right)^{\frac{1}{p}}\right)^p \\ &\ge  \left(1 - \frac{1}{(1-2\,\varepsilon_0)^{\frac{N}{p} - s}}\,\left(\frac{\gamma}{\gamma_0}\right)^{\frac{1}{p}}\right)^p, 
	\end{split}
	\]
 and the last quantity is {\it positive} in light of our choice \eqref{epsilon}. 
	By inserting \eqref{scelta_gamma0} and \eqref{epsilon} in \eqref{ineq-final_upperbound}, we then infer the claimed inequality \eqref{claim_upper_bound_1 p>1} with
	\begin{equation}\label{costante upper bound}
		\begin{split}
			\mathcal{C}(N, p, s, q, \gamma) = &\left(1 - \frac{1}{(1-2\,\varepsilon_0)^{\frac{N}{p} - s}}\,\left(\frac{\gamma}{\gamma_0}\right)^{\frac{1}{p}}\right)^{-p}\,\mathscr{A}\left(N,p,s, q, \gamma, \varepsilon_0, 0\right),
		\end{split}
	\end{equation}
	with $\varepsilon_0$ and $\mathscr{A}$ respectively given by \eqref{epsilon} and \eqref{costante-ausiliaria}. 
	\vskip.2cm \noindent 
	For $s\,p=N$, we need to use \eqref{ineq:cap-potential-sp=N}, in place of \eqref{ineq:cap-potential}, to bound from below the leftmost term in \eqref{1-media}. We get \[
	\left(1 - \fint_{B_{(1-2\varepsilon)r}} \varphi_\delta\,dx\right)^{\frac{N}{s}} \ge \left(1- \frac{|B_{2\,r}|}{|B_{(1-2\,\varepsilon)\,r}|} \left(\frac{1}{K_{N,s}}\,\widetilde{{\rm cap}}_{s,\frac{N}{s}}\left(\overline{B_1}; B_{2}\right)\,(\delta + \gamma)\right)^{\frac{s}{N}} \right)^{\frac{N}{s}}.
	\]
	By spending the last inequality in \eqref{quasi-quasi} and by passing to the limit as $\delta \searrow 0$, we get
	\begin{equation}\label{ineq-final_upperbound2}
		\begin{split}
			\frac{\mathscr{A}(N, N/s ,s, q, \gamma, \varepsilon, 0)}{r^{\frac{N^2}{s\,q}}} &\ge \lambda_{{\frac{N}{s}}, q}^s(\Omega)\,\left(1- \frac{1}{1-2\,\varepsilon}\,\left(\frac{\gamma}{\gamma_0}\right)^{\frac{s}{N}}\right)^\frac{N}{s} 
			\\ &\ge  \lambda_{{\frac{N}{s}}, q}^s(\Omega)\,\left(1- \frac{1}{1-2\,\varepsilon_0}\,\left(\frac{\gamma}{\gamma_0}\right)^{\frac{s}{N}}\right)^\frac{N}{s},
		\end{split}
	\end{equation}
	for every $0 < \gamma < \gamma_0$ and $0 < \varepsilon \le \varepsilon_0$, with $\mathscr{A}$ always given by \eqref{costante-ausiliaria} and where this time 
	\begin{equation}\label{gamma_sp=N}
		\gamma_0 = \gamma_0(N,s) :=  \min\left\{2^{-N} \left(\frac{1}{K_{N,s}}\,\widetilde{{\rm cap}}_{s,\frac{N}{s}}(\overline{B_1}; B_{2})\right)^{-1} , 1 \right\},
	\end{equation}
	and
	\begin{equation} \label{scelta-eps-conforme}
		\varepsilon_0 := \frac{1}{4}\,\left(1 - \left(\frac{\gamma}{\gamma_0}\right)^{\frac{1}{N}}\right). 
	\end{equation}
	This entails the announced inequality \eqref{claim_upper_bound_1 p>1}, where we can take 
	\begin{equation} \label{costante-upper-bound-conforme}
		\begin{split}
			\mathcal{C}(N,s, q, \gamma) = &\left(1- \frac{1}{1-2\,\varepsilon_0}\,\left(\frac{\gamma}{\gamma_0}\right)^{\frac{s}{N}}\right)^{-\frac{N}{s}}\,\mathscr{A}(N, N/s, s, q, \gamma, \varepsilon_0, 0).
		\end{split}
	\end{equation}
	\vskip.2cm \noindent
	{\it Case $\boxed{p = 1}$.} This time, we use the {\it sharp estimate} in Lemma \ref{poincarè-Palle} to minorize the integral in the leftmost term of \eqref{quasi-quasi}. By \eqref{gamma_cap p>1} and \eqref{cap-approximation p>1}, this entails
	\begin{equation}
		\begin{split}\label{ineq:lower_bound_media}
			\fint_{B_{(1-2\,\varepsilon)\,r}}(1-\varphi_\delta)\,dx \ge 1- \dfrac{P_s(B_r)}{P_s(B_{(1-2\,\varepsilon)\,r})}\,(\delta + \gamma)
			= 1- \frac{\delta + \gamma}{(1-2\,\varepsilon)^{N-s}},
		\end{split}
	\end{equation}
	for every $0 < \varepsilon < 1/2.$ By \eqref{quasi-quasi} and by sending $\delta \to 0,$ we get
	\begin{equation*}\label{ineq.stima_finale_p=1}
		\lambda^s_{1, q}(\Omega)\,\left(1-\frac{\gamma}{(1-2\,\varepsilon)^{N-s}}\right) \le \frac{\mathscr{A}(N, 1,s, q, \gamma, \varepsilon, 0) }{r^{s- N + \frac{N}{q}}},
	\end{equation*}
	for every $0 < \varepsilon < 1/2$, where $\mathscr{A}$ is given by \eqref{costante-ausiliaria}.
	We now take $ \varepsilon_0 = \varepsilon_0(\gamma) \in (0, 1/2)$ so that 
	\begin{equation} \label{scelta-eps-p=1}
		1-\frac{\gamma}{(1-2\varepsilon_0)^{N-s}} = \dfrac{1-\gamma}{2} \iff \varepsilon_0 = \frac{1}{2} \left(1 - \left(\frac{2\gamma}{1+\gamma}\right)^{\frac{1}{N-s}}\right).
	\end{equation}
	Then, from the last inequality, we can infer \eqref{claim_upper_bound_1 p>1} for a constant $\mathcal{C}$ given by  
	\begin{equation} \label{eqn:costante-upper-bound}
		\mathcal{C}(N, s, q, \gamma)= \frac{2}{1-\gamma}\,\mathscr{A}\left(N,s, q, \gamma,\varepsilon_0,0\right),
	\end{equation} 
	as desired. Eventually, for the claimed asymptotic behaviours of $\mathcal{C}\left(N,p,s, q, \gamma\right)$ we refer to Remark \ref{rmk:asym-C-upper-bound} below. 
\end{proof}
\begin{remark}\label{rmk:asymptotic-s-upper bound}
	{\it Quality of $\gamma_0$.} Let $1 < p < \infty$, by Remark \ref{rmk:asym-cap} and by \cite[Theorem 1]{MazShap} we infer 
	\[
	\gamma_0(N,p,s)\sim 1,\quad \mbox{ as } s \searrow 0,
	\]
	where $\gamma_0$ is given\footnote{Observe that we have $s\,p < N$ eventually, as $s \searrow 0$.} by \eqref{scelta_gamma0}. Furthermore, if $1 < p \le N$ we also have 
	\[
	\gamma_0(N, p, s) \sim 1,\quad \mbox{ as } s \nearrow 1.
	\] 
\end{remark}
\begin{remark} \label{rmk:asym-C-upper-bound}
	{\it Quality of $\mathcal{C}$}. We discuss the qualitative limiting behaviours of the constant $\mathcal{C} = \mathcal{C}(N,p,s, q,\gamma)$ of Theorem \ref{thm:upper-bound}. 
	\begin{itemize}
		\item For $1 < p< \infty$ and $0 < s < 1$ such that $s\,p \le N$, by \eqref{epsilon} and \eqref{scelta-eps-conforme}, we have
		\[
		0 < \lim_{\gamma \to \gamma_0} \frac{\varepsilon_0}{\gamma_0 -\gamma} < \infty,
		\]
		and so \[
		0 < \lim_{\gamma \to \gamma_0} \left(\gamma_0-\gamma\right)^{p + s\,(p-1)}\,\mathcal{C}(N,p, s, q, \gamma) < \infty,
		\]
		by recalling \eqref{costante upper bound} and \eqref{costante-upper-bound-conforme}. Moreover, for $1 < p < \infty$, by Remark \ref{rmk:asym-cap} and Remark \ref{rmk:asymptotic-s-upper bound} we get 
		\[
		\mathcal{C}(N, p, s, q, \gamma) \sim \frac{1}{s}, \quad \mbox{ as } s \searrow 0,
		\]
		and, for $1 < p \le N$, we also have that 
		\[
		\mathcal{C}(N,p, s, q, \gamma) \sim \frac{1}{1- s}, \quad \mbox{ as } s \nearrow 1.
		\]
		\vskip.2cm \noindent
		\item For $p=1$, we have $\gamma_0 = 1$. By recalling \eqref{scelta-eps-p=1}, we get  \[
		0 < \lim_{\gamma \to 1} \frac{\varepsilon_0}{1-\gamma} < \infty, \quad \mbox{ and } \quad 	0 < \lim_{\gamma \to 1} (1-\gamma)\,\mathcal{C}(N,s,q,\gamma) <\infty. 
		\]
		Moreover, by Remark \ref{rmk:asym-cap},  from \eqref{eqn:costante-upper-bound} we get \[
		0 < \lim_{s \to 0} s\,\mathcal{C}(N,s,q,\gamma) <\infty, \qquad \mbox{ and } \qquad 0 < \lim_{s \to 1} (1-s)\,\mathcal{C}(N,s,q,\gamma) < \infty.
		\]
	\end{itemize}
\end{remark}

\appendix 

\section{Qualitative asymptotics}
\label{sec:app}
For $p=2$, the following proposition is already contained in \cite[Lemma A.1]{BCV}. Its proof relies on the celebrated Bourgain-Brezis-Mironescu's Theorem, see \cite{BBM} and also \cite[Corollary 4]{Bre_constant}.  Up to some minor differences, it is immediate to extend the proof of \cite[Lemma A.1]{BCV} to every frequency $\lambda^s_{p,q}$. 
\begin{proposition}  \label{prop:asym-s-1}
	Let $1\leq p<\infty$ and $1\leq q<p^*$. For every $\Omega \subseteq \mathbb{R}^N$ open bounded set with Lipschitz boundary, we have 
	\begin{equation*}\label{eq:equality for lipschitz}
		\lim _{s \to 1}(1-s)\,\lambda_{p, q}^s(\Omega) = K_{N,p}\,\lambda_{p,q}(\Omega),
	\end{equation*}
	where $K_{N,p} > 0$ is defined in \cite[pag. 6]{BBM}. 
\end{proposition}

The next  result is direct consequence of \cite[Theorem 1]{MazShap} and H\"older's inequality.  
\begin{proposition}\label{prop:positività-autovalori}
	Let $1\le p <\infty$ and $0<s<1$ be such that $s\,p<N$. We take an exponent $1 \le q \le p^*_s$. For every $\Omega \subseteq \mathbb{R}^N$ open set with finite measure, we have \begin{equation*}\label{eigenvalue_subcritical_lowerbound}
		s\,(1-s)\,\lambda_{p, q}^s\left(\Omega\right) \geq c_{N,p}\,\left(N-s\,p\right)^{p-1}\,|\Omega|^{\frac{p}{p^*_s}-\frac{p}{q}}, 
	\end{equation*}
	for some $c_{N,p} > 0$.
\end{proposition}
For our purposes the next proposition, which is a plain consequence of Proposition \ref{prop:positività-autovalori} and \cite[Theorem 3]{MazShap}, will be sufficient. The details are left to the reader.
\begin{proposition} \label{prop:asym2}
	Let $1 \le p < \infty$ and $0 < s <1$ be such that $s\,p < N$. For every $\Omega\subseteq\mathbb{R}^N$ open set with finite measure, we have\footnote{By writing $``(s,q) \to (0, p)^{+}"$, we mean that $s \to 0^{+}$ and $q \to p^{+}$.} 
	\[
	0 < \liminf_{(s,q) \to (0, p)^{+}} s\,\lambda_{p, q}^s(\Omega) \le \limsup_{(s,q) \to (0, p)^{+}} s\,\lambda_{p, q}^s(\Omega) < \infty.
	\]
\end{proposition}
As a byproduct, we can infer the forthcoming result 
\begin{example}\label{ex:finiteness-inr}
	Let $1 \le p < \infty$. There exists an open set $\Omega \subseteq \mathbb{R}^N$ and $0 < \gamma_0 = \gamma_0\left(N,p\right) \le 1$ such that 
	\begin{equation} \label{eqn:ex-tesi}
		\limsup_{s \to 0} R^s_{p, \gamma}(\Omega) \le r_{N, p, \gamma}, \quad \mbox{ for } 0 < \gamma < \gamma_0,
	\end{equation}
	for some $r_{N, p, \gamma} > 0$. Indeed, let $\Omega = \mathbb{R}^{N-1} \times (-1,1)$ and pick a ball $B_r(x_0)$ of radius $r > 1$ such that
	\begin{equation} \label{eqn:ex}
		\widetilde{{\rm cap}}_{s,p}\left(\overline{B_r(x_0)} \setminus \Omega; B_{2\,r}(x_0)\right) \le \gamma\,\widetilde{{\rm cap}}_{s,p}\left(\overline{B_r(x_0)}; B_{2\,r}(x_0)\right). 
	\end{equation} 
	Since $\Omega$ is invariant by translations, we can assume without loss of generality that $x_0 = t\,{\bf e}_N.$ By \eqref{cap-vol} and \eqref{eqn:ex}, we get 
	\begin{equation}
		\begin{split}
		\left|\overline{B_{r}(t\,{\bf e}_N)}\setminus \Omega\right|\,\lambda^s_p(B_{2\,r}(t\,{\bf e}_N)) &\le \widetilde{{\rm cap}}_{s,p}\left(\overline{B_r(t\,{\bf e}_N)} \setminus \Omega; B_{2\,r}(t\,{\bf e}_N)\right) \\
		&\le \gamma\,r^{N-s\,p}\,\widetilde{{\rm cap}}_{s,p} \left(\overline{B_1}; B_2\right)
		\\
		&\le \gamma\,r^{N-s\,p}\,\frac{c_{N, p}}{s\,(1-s)}\,\lambda_{p}(B_2)^{s-1}\,{\rm cap}_p\left(\overline{B_1}; B_2\right),
	\end{split}
	\end{equation}
	where in the last line we also used \eqref{cap-cap}. 
	Furthermore, by arguing as in \cite[Example A.2]{BozBra2}, we can bound from below the leftmost term in the first line of the previous inequality by $r^{N}\,\varphi_N(r),$ where
	\[
 \varphi_N(r) := 2\,\omega_{N-1}\,\int_{\arcsin \frac{1}{r}}^{\frac{\pi}{2}} \cos^N t\,dt.
	\] 
	This entails that 
	\begin{equation} \label{eqn:ex-2}
		\varphi_N(r) \le \gamma\,c_{N,p}\frac{\lambda_{p}(B_2)^{s-1}\,{\rm cap}_p\left(\overline{B_1}; B_2\right)}{s\,(1-s)\,\lambda_{p}^s(B_2)}. 
	\end{equation}
	The function $\varphi_N(r)$, for $r \in (1, \infty)$, is continuous and monotone increasing and for its range of attainable values we have  ${\rm im}\,\varphi_N = (0, \omega_N).$ Thus, $\varphi_N$ admits a continuous monotone increasing inverse, say $\varphi_N^{-1}$, defined on the interval $(0, \omega_N)$. Then, by \eqref{eqn:ex-2} and
 by the arbitrariness of $r$ 
	\[
	R^s_{p,\gamma}(\Omega) \le \varphi_N^{-1}\left(\min\left\{\gamma\,c_{N,p}\,\frac{\lambda_{p}(B_2)^{s-1}\,{\rm cap}_p\left(\overline{B_1}; B_2\right)}{s\,(1-s)\,\lambda^s_p(B_2)},\,\, \omega_N\right\}\right).
	\]
	 By passing to the $\limsup$ as $s \searrow 0$ and by taking 
	 \begin{equation*}
	 	\gamma_0(N,p) := \min\left\{\frac{\omega_N\,\lambda_p(B_2)}{c_{N,p}\,{\rm cap}_p\left(\overline{B_1}; B_2\right)}\,\liminf_{s \to 0} s\,\lambda^s_p(B_2),\,\, 1\right\},
	 \end{equation*}
	 we eventually get \eqref{eqn:ex-tesi}, in light of the continuity of $\varphi_N^{-1}$ and Proposition \ref{prop:asym2}. 
\end{example}


\begin{thebibliography}{100}
	
	\bibitem{AbFelNor} L. Abatangelo, V. Felli and B. Noris, On simple eigenvalues of the fractional Laplacian under removal of small fractional capacity sets, Commun. Contemp. Math. {\bf 22} (2020), no.~8, 1950071, 32 pp.;
	
	
	\bibitem{AmbDePMar} L. Ambrosio, G. De~Philippis and L. Martinazzi, Gamma-convergence of nonlocal perimeter functionals, Manuscripta Math. {\bf 134} (2011), no.~3-4, 377--403; 
	
	\bibitem{BanLatMH} R. Ba\~nuelos, R. Lata\l a{} and P.~J. M\'endez-Hern\'andez, A Brascamp-Lieb-Luttinger-type inequality and applications to symmetric stable processes, Proc. Amer. Math. Soc. {\bf 129} (2001), no.~10, 2997--3008;
	
	\bibitem{tesi-francesca} F. Bianchi, {\it Some geometric estimates for fractional Poincar\'e inequalities}, Ph.D. Thesis, Universit\`a di Parma (2024), available at {\tt https://hdl.handle.net/1889/5647};
	
	\bibitem{BiaBra22} F. Bianchi and L. Brasco, The fractional Makai-Hayman inequality, Ann. Mat. Pura Appl. (4) {\bf 201} (2022), no.~5, 2471--2504;
	
	\bibitem{BiaBra24} F. Bianchi, L. Brasco, An optimal lower bound in fractional spectral geometry for planar sets with topological constraints, J. London Math. Soc., {\bf 109} (2024), e12814.
	
	\bibitem{BBZ} F. Bianchi, L. Brasco and A.~C. Zagati, On the sharp Hardy inequality in Sobolev-Slobodecki\u{\i} spaces, Math. Ann. {\bf 390} (2024), no.~1, 493--555.
	
	\bibitem{BisRadServadei_book} G. Molica~Bisci, V.~D. R\u adulescu and R. Servadei, {\it Variational methods for nonlocal fractional problems}, Encyclopedia of Mathematics and its Applications, 162, Cambridge Univ. Press, Cambridge, 2016; 
	
	
	\bibitem{BBM} J. Bourgain, H. Brezis and P. Mironescu, Another look at Sobolev spaces, in {\it Optimal control and partial differential equations}, 439--455, IOS, Amsterdam;
	
	
	\bibitem{BozBra} F. Bozzola and L. Brasco, The role of topology and capacity in some bounds for principal frequencies, J. Geom. Anal. {\bf 34} (2024), no.~10, Paper No. 299, 46 pp.; 
	
	\bibitem{BozBra2} F. Bozzola, L. Brasco, Capacitary inradius and Poincar\'e-Sobolev inequalities, ESAIM: COCV, {\bf 31} (2025), Paper no. 26.
	
	\bibitem{BB_variation} F. Bozzola, L. Brasco, Variations on the capacitary inradius, {\it accepted paper:} Discrete Contin. Dyn. Syst. Ser. S (2025);
	
	\bibitem{Brasco_book} L. Brasco, {\it Handbook of Calculus of Variations for Absolute Beginners}, Springer Cham, (2025).
	
	\bibitem{BraPini} L. Brasco, On principal frequencies and inradius in convex sets, in {\it Bruno Pini Mathematical Analysis Seminar 2018}, 78--101, Bruno Pini Math. Anal. Semin., 9, Univ. Bologna, Alma Mater Stud., Bologna; 
	
	\bibitem{BraTol} L. Brasco, On principal frequencies and isoperimetric ratios in convex sets, Ann. Fac. Sci. Toulouse Math. (6) {\bf 29} (2020), no.~4, 977--1005;
	
	\bibitem{BraBriPri} L. Brasco, L. Briani, F. Prinari, Extremals for sharp Poincar\'e-Sobolev constants in Steiner symmetric sets, preprint (2025),  available at {\tt 	arXiv:2505.11084};
	
	\bibitem{BC} L. Brasco, E. Cinti, On fractional Hardy inequalities in convex sets, Discrete Contin. Dyn. Syst. {\bf 38} (2018), no.~8, 4019--4040;
	
	\bibitem{BCV} L. Brasco, E. Cinti and S. Vita, A quantitative stability estimate for the fractional Faber-Krahn inequality, J. Funct. Anal. {\bf 279} (2020), no.~3;
	
	\bibitem{BraDePFran}L. Brasco, G. De~Philippis and G. Franzina, Positive solutions to the sublinear Lane-Emden equation are isolated, Comm. Partial Differential Equations {\bf 46} (2021), no.~10, 1940--1972; 
	
	\bibitem{BraFran} L. Brasco and G. Franzina, An overview on constrained critical points of Dirichlet integrals, Rend. Semin. Mat. Univ. Politec. Torino {\bf 78} (2020), no.~2, 7--50; 
	
	\bibitem{BGV} L. Brasco, D. G\'omez-Castro and J.~L. V\'azquez, Characterisation of homogeneous fractional Sobolev spaces, Calc. Var. Partial Differential Equations {\bf 60} (2021), no.~2, Paper No. 60, 40 pp.; 
	
	\bibitem{BraLin} L. Brasco and E. Lindgren, Uniqueness of extremals for some sharp Poincar\'e-Sobolev constants, Trans. Amer. Math. Soc. {\bf 376} (2023), no.~5, 3541--3584;
	
	\bibitem{BLP} L. Brasco, E. Lindgren and E. Parini, The fractional Cheeger problem, Interfaces Free Bound. {\bf 16} (2014), no.~3, 419--458.
	
	\bibitem{BraMaz} L. Brasco and D. Mazzoleni, On principal frequencies, volume and inradius in convex sets, NoDEA Nonlinear Differential Equations Appl. {\bf 27} (2020), no.~2, Paper No. 12, 26 pp.; 
	
	\bibitem{BraMosSqu} L. Brasco, S. Mosconi, M. Squassina, Optimal decay of extremals for the fractional Sobolev inequality, Calc. Var. Partial Differential Equations, {\bf 55} (2016), Art. 23, 32 pp.
	
	\bibitem{BP} L. Brasco and E. Parini, The second eigenvalue of the fractional $p$-Laplacian, Adv. Calc. Var. {\bf 9} (2016), no.~4, 323--355.
	
	\bibitem{BPS} L. Brasco, E. Parini and M. Squassina, Stability of variational eigenvalues for the fractional $p$-Laplacian, Discrete Contin. Dyn. Syst. {\bf 36} (2016), no.~4, 1813--1845.
	
	
	
	\bibitem{BraPriZag1} L. Brasco, F. Prinari and A.~C. Zagati, A comparison principle for the Lane-Emden equation and applications to geometric estimates, Nonlinear Anal. {\bf 220} (2022), Paper No. 112847, 41 pp.; 
	
	\bibitem{BraPriZag2} L. Brasco, F. Prinari and A.~C. Zagati, Sobolev embeddings and distance functions, Adv. Calc. Var. {\bf 17} (2024), no.~4, 1365--1398;
	
	
	\bibitem{BraRuf} L. Brasco and B. Ruffini, Compact Sobolev embeddings and torsion functions, Ann. Inst. H. Poincar\'e{} C Anal. Non Lin\'eaire {\bf 34} (2017), no.~4, 817--843; 
	
	\bibitem{BS} L. Brasco, A. Salort, A note on homogeneous Sobolev spaces of fractional order, Ann. Mat. Pura Appl. (4),  {\bf 198} (2019), 1295--1330.
	
	\bibitem{Bre} H. Brezis, {\it Functional Analysis, Sobolev Spaces and Partial Differential Equations}, Springer (2010).
	
	\bibitem{Bre_constant} H. Brezis, How to recognize constant functions. A connection with Sobolev spaces, Russian Math. Surveys {\bf 57} (2002), no.~4, 693--708; translated from Uspekhi Mat. Nauk {\bf 57} (2002), no.~4(346), 59--74;
	
	\bibitem{CotTav} A. Cotsiolis and N.~K. Tavoularis, Best constants for Sobolev inequalities for higher order fractional derivatives, J. Math. Anal. Appl. {\bf 295} (2004), no.~1, 225--236;
	
	\bibitem{DiB} E. DiBenedetto, {\it Real analysis}, second edition, Birkh\"auser Advanced Texts: Basler Lehrb\"ucher, Birkh\"auser/Springer, New York, 2016; 
	
	\bibitem{Demengel_book} F. Demengel and G. Demengel, {\it Functional spaces for the theory of elliptic partial differential equations}, Universitext, Springer, London, 2012;
	
	\bibitem{EE} D. E. Edmunds, W. D. Evans, {\it Fractional Sobolev spaces and inequalities}, Cambridge Tracts in Mathematics, {\bf 230}. Cambridge University Press, Cambridge, 2023. 
	
	\bibitem{FFMMM} A. Figalli, N. Fusco, F. Maggi, M. Morini, V. Millot,  Isoperimetry and stability properties of balls with respect to nonlocal energies, Comm. Math. Phys. {\bf 336} (2015), no.~1, 441--507.
	
	\bibitem{FSV} A. Fiscella, R. Servadei and E. Valdinoci, Density properties for fractional Sobolev spaces, Ann. Acad. Sci. Fenn. Math. {\bf 40} (2015), no.~1, 235--253.
	
	\bibitem{FS} R.~L. Frank, R. Seiringer, Non-linear ground state representations and sharp Hardy inequalities, J. Funct. Anal. {\bf 255} (2008), no.~12, 3407--3430.
	
	
	
	\bibitem{franzina-torsion} G. Franzina, Non-local torsion functions and embeddings, Appl. Anal. {\bf 98} (2019), no.~10, 1811--1826;
	
	\bibitem{FraPal} G. Franzina, G. Palatucci, Fractional $p$-eigenvalues, Riv. Math. Univ. Parma (N.S.) {\bf 5} (2014), no.~2, 373--386;
	
	
	\bibitem{Gi} E. Giusti, {\it Direct methods in the calculus of variations}, World Sci. Publ., River Edge, NJ, 2003; 
	
	\bibitem{Grisvard_book} P. Grisvard, {\it Elliptic problems in nonsmooth domains}, Monographs and Studies in Mathematics, 24, Pitman, Boston, MA, 1985;
	
	
	\bibitem{Leoni}  G. Leoni, {\it A first course in Sobolev spaces.} Second edition. Graduate Studies in Mathematics, {\bf 181}. American Mathematical Society, Providence, RI, 2017
	
	\bibitem{Leoni_fractional} G. Leoni, {\it A first course in fractional Sobolev spaces}, Graduate Studies in Mathematics, 229, Amer. Math. Soc., Providence, RI, (2023);
	
	
	
	\bibitem{LL} E. H. Lieb, M. Loss, {\it Analysis}. Second edition. Graduate Studies in Mathematics, {\bf 14}. American Mathematical Society, Providence, RI, 2001.
	
	\bibitem{LindLind} E. Lindgren and P. Lindqvist, Fractional eigenvalues, Calc. Var. Partial Differential Equations {\bf 49} (2014), no.~1-2, 795--826; 
	
	\bibitem{Lombardini} L. Lombardini, Approximation of sets of finite fractional perimeter by smooth sets and comparison of local and global $s$-minimal surfaces, Interfaces Free Bound. {\bf 20} (2018), no.~2, 261--296; 
	
	
	
	\bibitem{Maz} V. Maz'ya, {\it Sobolev spaces with applications to elliptic partial differential equations}. Second, revised and augmented edition. Springer, Heidelberg, 2011. 
	
	\bibitem{MazShap} V. Maz'ya and T.~O. Shaposhnikova, On the Bourgain, Brezis, and Mironescu theorem concerning limiting embeddings of fractional Sobolev spaces, J. Funct. Anal. {\bf 195} (2002), no.~2, 230--238;
	
	\bibitem{MS} V. Maz'ya, M. Shubin, Can one see the fundamental frequency of a drum?, Lett. Math. Phys., {\bf 74} (2005), 135--151.
	
	
	\bibitem{Stein_book} E.~M. Stein, {\it Singular integrals and differentiability properties of functions}, Princeton Mathematical Series, No. 30, Princeton Univ. Press, Princeton, NJ, 1970; 
	
	\bibitem{Struwe} M. Struwe, {\it Variational methods. Applications to nonlinear partial differential equations and Hamiltonian systems. Fourth edition}. A Series of Modern Surveys in Mathematics, {\bf 34}. Springer-Verlag, Berlin, 2008.
	
	
	\bibitem{Triebel1} H. Triebel, {\it Theory of function spaces}, reprint of 1983 edition, 
	Modern Birkh\"auser Classics, Birkh\"auser/Springer Basel AG, Basel, 2010;
	
	\bibitem{Visintin} A. Visintin, Generalized coarea formula and fractal sets, Japan J. Indust. Appl. Math. {\bf 8} (1991), no.~2, 175--201; 
	
	\bibitem{Warma} M. Warma, The fractional relative capacity and the fractional Laplacian with Neumann and Robin boundary conditions on open sets, Potential Anal., {\bf 42} (2015), 499--547.
	
	
\end{thebibliography}
\end{document}